\documentclass[11pt, english, reqno]{amsart}

\usepackage[T1]{fontenc}
\usepackage{babel}
\usepackage{mathrsfs}
\usepackage{amsthm}
\makeindex

\usepackage[dvips,top=3.8cm,left=4cm,right=3.8cm,
foot=3.8cm,bottom=4.0cm]{geometry}
\usepackage{amsfonts,amssymb,amsmath,amsthm, booktabs, 
latexsym}
\usepackage{centernot}
\usepackage{tikz-cd}
\usepackage{bbm}

\usepackage{enumerate}
\usepackage{color}

\newtheorem{theorem}{Theorem}[section]
\newtheorem{lemma}[theorem]{Lemma}
\newtheorem{proposition}[theorem]{Proposition}
\newtheorem{corollary}[theorem]{Corollary}
\theoremstyle{definition}
\newtheorem{definition}[theorem]{Definition}

\theoremstyle{remark}
\newtheorem{remark}[theorem]{Remark}

\numberwithin{equation}{section}

\numberwithin{equation}{section}
\newcommand{\be}{\begin{equation}}
\newcommand{\ee}{\end{equation}}
\newcommand{\ba}{\begin{aligned}}
\newcommand{\ea}{\end{aligned}}

\newcommand{\N}{{\mathbb N}}

\newcommand{\R}{{\mathbb R}}

\newcommand{\h}{{\mathcal H}}
\def\va{\varphi}
\def\la{\lambda}
\def\e{\varepsilon}
\def\csi1{\circ\sigma^{-1}}

\def\ol{\overline}
\def\wt{\widetilde}

\def\mc{\mathcal}

\def\lc2{L^2_{\mathrm{loc}}(\mu)}
\newcommand{\B}{{\mathcal B}}
\newcommand{\Bfin}{\B_{\mathrm{fin}}}
\newcommand{\Dfin}{\mc D_{\mathrm{fin}}}
\newcommand{\Lloc}{L^1_{\mathrm{loc}}(\mu)}
\newcommand{\sms}{(V, \mathcal B, \mu)}
\newcommand{\VB}{(V, \mathcal B)}
\newcommand{\vv}{(V\times V, \B\times \B)}
\newcommand{\FVB}{{\mathcal F(V, \B)}}

\newcommand{\VtV}{V\times V}

\begin{document}

\title[Finite energy  space]{Laplace operators in finite energy and 
dissipation spaces}

\author{Sergey Bezuglyi}
\address{Department of Mathematics, University of Iowa, Iowa City,
52242 IA, USA}
\email{sergii-bezuglyi@uiowa.edu}

\author{Palle E.T. Jorgensen}
\address{Department of Mathematics, University of Iowa, Iowa City,
52242 IA, USA}
\email{palle-jorgensen@uiowa.edu}




\begin{abstract} 
Recent applications of large network models to machine learning, 
and to neural network suggest a need for a systematic study of the 
general correspondence, (i) discrete vs (ii) continuous. Even if the 
starting point is (i), limit considerations lead to (ii), or, more 
precisely, to a measure theoretic framework which we make precise. 
Our motivation derives from graph analysis, e.g., studies of (infinite) 
electrical networks of resistors,  but our focus will be (ii), i.e., the 
measure theoretic setting. In electrical networks of resistors, one 
considers pairs (of typically countably infinite), sets $V$ (vertices), 
$E$ (edges) a suitable subset of $V \times V$, and prescribed positive
symmetric functions $c$  on $E$ .  A conductance function $c$  is 
defined on $E$ (edges), or on $V \times V$, but with $E$ as its 
support. From an initial triple $(V, E, c)$ , one gets graph-Laplacians,
generalized Dirichlet spaces (also called energy Hilbert spaces), 
dipoles, relative reproducing kernel-theory, dissipation spaces, 
reversible Markov chains, and more.

Guided by applications to measurable equivalence relations, we extend 
earlier analyses to the non-discrete framework, and, with the use of 
spectral theory, we study correspondences. Our setting is that of 
\textit{standard Borel spaces} $(M, \B)$. Parallel to conductance 
functions in (i), we consider (in the measurable framework) a fixed 
positive, symmetric, and $\sigma$-finite measure  $\rho$ on the 
product space $(M \times M, \B \times \B)$. We study both the 
structures that arise as graph-limits, as well as the induced graph 
Laplacians, Dirichlet spaces, and reversible Markov processes, 
associated directly with the measurable framework. 

Our \textit{main results} include:  spectral theory and Green's functions for
 measure theoretic  graph-Laplace operators;  the theory of reproducing
  kernel Hilbert spaces related to Laplace operators; a rigorous analysis of  
  the Laplacian on Borel equivalence 
relations; a new decomposition theory; irreducibility criteria; dynamical 
systems governed by endomorphisms and measurable fields; orbit 
equivalence 
criteria; and path-space measures and induced dissipation Hilbert spaces.
We consider several applications of our results to other fields such as  
machine learning problems, reproducing kernel Hilbert spaces, Gaussian and
determinantal processes, and joinings.

\end{abstract}

\maketitle

\tableofcontents
\section{\textbf{Introduction}}\label{sect Intro}

In the past decade, many researchers have studied diverse approaches to
notions of limits of large networks. Hence there are two settings, 
(i) discrete vs (ii) continuous. More precisely, a starting point-setting 
often constitutes (i) a suitable class of network structures, each of 
independent interest. In this setting, an extensive (discrete) analysis 
has already been undertaken; see the references cited below. Now, 
the network analysis is typically undertaken before consideration of 
any kind of graph-limit, or limits. But our present focus will be a 
systematic study of (ii) suitable limit structures; the best known is 
perhaps the notion of ``\textit{graphons}''. In any case, the limit 
structures are
continuous, or rather, they are studied in a measure theoretic framework.
(Below, we include a brief technical summary of recent related studies; 
and our present opening comments here are meant merely to help 
non-expert readers with an orientation to the general setting.)

In summary, the setting before limits are taken is discrete: typically, 
a graph consisting of sets of vertices, and edges, as well as associated 
and specified classes of functions, and random processes. By contrast, 
the appropriate limit objects are ``\textit{continuous}''.  More precisely, 
the category suitable for the limits is that of measure space, equipped 
with a definite structure which we make precise below. Prototypes of large 
networks (case (i)) include: the Internet, networks in chip designs, 
social networks, ecological networks, networks of proteins, the human 
brain, a network of neurons, statistical physics (large numbers of 
discrete particles, or spin-observables, realized on graphs, or 
electrical networks of resistors.

While each of these examples serves to illustrate our present analysis, 
we have chosen the latter (electrical networks of resistors) as a main 
reference for our general theory. And our focus will be (ii), i.e., the 
measure theoretic setting. To be more precise, by an electrical 
networks of resistors, we mean a pair (of typically countably infinite) 
sets $V$ (vertices), $E$ (edges) a suitable subset of $V \times V$, 
and a prescribed positive symmetric function $c$  on $E$  representing
conductance.  The conductance function $c$  is defined on $E$ (edges), 
or on $V \times V$, but with $E$ as its support. From an initial triple 
$(V, E, c)$, one is then lead to well defined graph-Laplacians, 
generalized Dirichlet spaces (also called energy Hilbert spaces), dipoles, 
relative reproducing kernel-theory, dissipation spaces, reversible 
Markov chains, and more. While this setting is of interest in its own 
right; see the papers cited below, our main focus here will be the 
structures which arise as limits, hence analogous structures on measure 
spaces; see Section \ref{sect Prelim} below for details. In summary, 
a measure space is a set $M$, and a prescribed sigma-algebra $\B$ of 
subsets of $M$. Our setting will be that of \textit{standard Borel spaces} 
$(M, \B)$. Parallel to the setting of conductance functions in the 
discrete case, our starting point for the ``\textit{continuous}'' analysis 
will now 
be a fixed positive, symmetric, and sigma-finite measure  $\rho$ on 
the product space $(M \times M, \B \times \B)$. We shall study both 
the structures that arise as graph-limits, as well as the induced graph 
Laplacians, Dirichlet spaces, and reversible Markov processes that are 
associated directly with the measurable framework.

The setting in Section \ref{sect Prelim} below is motivated by a list 
of applications, detailed inside the paper. The list includes: a rigorous 
analysis of Borel equivalence relations, a new decomposition theory, 
irreducibility criteria, measurable fields, dynamical systems governed by
endomorphisms, graph-Laplace operators in a measure theoretic 
framework (Sections \ref{subsect_finiteEnergy} - \ref{sect LO in H_E}). 
In our 
present general setting of measure spaces (precise details below), we 
present an explicit correspondence central to the questions in the paper; 
a correspondence between on the one hand (i) \textit{symmetric 
measures} 
on product spaces, and on the other, (ii) \textit{reversible Markov} 
transition 
processes. In our first results (Sections \ref{sect Prelim} and 
\ref{subsect_finiteEnergy}), we make precise the notion of 
\textit{equivalence} for 
the setting of both (i) and (ii), and we prove that equivalence of two 
symmetric measures, leads to equivalence of the corresponding reversible 
Markov processes, and vice versa. Our general framework for these 
reversible Markov processes goes beyond earlier considerations in the 
literature; and hence allows us to attack questions from dynamics which 
were not previously accessible with existing tools.

With view to applications to spectral theory and potential theory, in 
Sections \ref{subsect_finiteEnergy} - \ref{sect energy}, for a given 
symmetric measure, we introduce a 
(measure) \textit{graph Laplacian} $\Delta$, a \textit{finite energy space} 
$\h_E$, and 
a \textit{dissipation Hilbert space} $Diss$,  see Section 
\ref{sect diss space}. 
These are central notions, and their adaptation here is motivated by 
energy space for discrete weighted networks, such as electrical networks 
of resistors, and a host of areas studied extensively during last decades.

Our realization of the reversible Markov processes depends on a 
carefully designed infinite path space, and its associated path-space 
measures; both introduced in Section \ref{sect diss space}. These tools 
are used in turn in our study of \textit{orthogonality relations} 
 (Section \ref{sect diss space}).

In Section  \ref{sect RKHS}, we associate to every transient Markov 
processes a positive definite kernel  and \textit{reproducing kernel 
Hilbert space (RKHS)}. 
We show that the latter is in turn a realization of the energy space. 
This then allows us to make precise a new Green's function approach to 
our study general reversible Markov processes. Applications and examples 
are included.

Inside the paper, we shall include detailed citations, but to help readers 
with orientation, we give the following general pointers to the References.

\textit{Earlier work} by the co-authors, related to the present paper, includes 
\cite{BezuglyiJorgensen2018, BezuglyiJorgensen_2018, 
Jorgensen2012, JorgensenTian2015, JorgensenTian-2016,
JorgensenTian2017}. 

For the theory of Dirichlet
 forms and generalized Laplacians, we refer to
\cite{GimTrutnau2017, Jonsson2005, KoskelaZhou2012, Rozkosz2001}.
Our definition of the finite energy Hilbert space and graph Laplacian
is related to a special  case of Dirichlet forms. The literature on 
Dirichlet forms is extensive; we recommend the following works:
\cite{AlbeverioBernabei2005, AlbeverioMaRockner2015, 
AlbeverioFanHerzberg2011, Albeverio2003, 
AlbeverioKondratievNikifirovTorbin2017, BlumenthalGetoor1968, 
ChenFukushima2012,  Kaimanovich1992, LawlerSokal1988}.  

  For the theory of reproducing kernels and their applications to Markov 
  processes, see 
  \cite{Adams_et_al1994, AplayJorgensen2014, 
  AlpayJorgensen2015, Aronszajn1950, AronszajnSmith1957,  
BerlinetThomas-Agnan2004, SaitohSawano2016}. 

Throughout the 
paper, we shall make use of a number of tools from ergodic theory, 
Borel equivalence relations, orbit equivalence,  and for this the reader
 may wish to consult \cite{CornfeldFominSinai1982, FeldmanMooreI_1977,
FeldmanMooreII_1977, GreschonigSchmidt2000, 
KechrisMiller2004, Kechris2010, Rohlin1949, Revuz1984}. 

The list of applications of our results includes models for online learning 
with Markov sampling. From the prior literature on this, we stress the 
following papers 
\cite{ArgyriouMicchelliPontil2010, GuofanZhou2016, 
Smaleyao2006, SmaleZhou2007, SmaleZhou2009, SmaleZhou2009a}.

We mention here also several adjacent areas where the ideas applied in
this paper might be useful. First of all, the theory of 
\textit{electrical networks}
can regarded as a prototype for many of ours definitions and results. 
The reader can find the necessary information in various sources; for us
the following books and articles are the most relevant 
 \cite{Kigami2001,  JorgensenPearse2010, Jorgensen_Pearse2011, 
 JorgensenPearse2013, JorgensenPearse2016,
JorgensenPearse2014, JorgensenPearse_2017, 
LyonsPeres2016, Woess2000, Woess2009}.

The references to 
\cite{FarsiGillaspyJorgensenKangPacker2018,
FarsiGillaspyJorgensenKangPacker_2018} deal with 
some aspects of the representation theory.

The reader can look at 
\cite{AlpayJorgensenLevanony2011}, \cite{AlpayJorgensen2012},
\cite{AlpayJorgensen2015}, \cite{AlpayJorgensenLevanony2017},
\cite{AlpayJorgensenLewkowicz2018} where Gaussian processes and 
path spaces are discussed.

\section{\textbf{Preliminaries and basic setting}}\label{sect Prelim}

Our aim in Section  \ref{sect Prelim} below is to apply our present 
Hilbert space tools to a systematic study of measurable partitions and 
countable Borel equivalence relation in the context of standard (Borel)
 measure space, and their applications. Measurable partitions in turn 
serve as important tools in Borel dynamics, in ergodic theory, and in 
direct integral analysis in representation theory, to mention only a few.  
A sample of the relevant literature includes \cite{AlbeverioBernabei2005,
AlbeverioMaRockner2015, BJ_book, FeldmanMooreII_1977,
FeldmanMooreI_1977, Gao2009, JacksonKechrisLouveau2002, 
Kechris2010, Simmons2012}.

\subsection{Standard measure space and symmetric measures}
\label{subsect sms symm meas}

Suppose $V$ is a \textit{Polish space}, i.e., $V$ is a separable completely
 metrizable topological space. Let $\B$ denote the $\sigma$-algebra of
 Borel sets   generated by open  sets of $V$. Then $(V, \B)$ is called a 
\textit{standard Borel space},  see e.g., 
\cite{Gao2009, Kanovei2008, Kechris1995, Kechris2010} and papers
\cite{Chersi1989, Loeb1975} for detailed information about standard Borel
spaces.  We recall that 
all uncountable  standard Borel spaces are Borel isomorphic, so that, 
without loss of generality,  we can use any convenient  realization of 
the space $(V, \B)$.   

If $\mu$ is a  positive non-atomic   Borel  measure  on  $(V,  \B)$, then  
 $(V, \mathcal B, \mu)$ is called a \emph{standard 
 measure space}. Given $\sms$, we will call $\mu$ a measure for brevity. 
 As a rule, we will deal  with non-atomic  $\sigma$-finite positive
measures on $(V, \B)$ which  take values in the extended real line 
$\ol \R$.  We  use the name of standard measure space for both 
probability (finite) and  $\sigma$-finite measure spaces. 
Also the same notation, $\B$, is applied for the $\sigma$-algebras of 
Borel sets and measurable  sets of a standard measure space. 
Working with a measure space $\sms$, we always assume that  $\B$ is 
 \textit{complete}  with respect  to $\mu$. By $\mathcal F(V, \B)$. 
we denote  the space of real-valued bounded Borel functions on $(V, \B)$.
For  $f \in  \mathcal F(V, \B)$ and a Borel measure $\mu$ on $(V, \B)$, 
 we write 
 $$
 \mu(f) = \int_V f\; d\mu.
 $$

All objects, considered in the context of measure spaces (such as sets,
functions, transformations, etc), are considered  by modulo sets 
of zero measure.  In most cases, we will  
implicitly use this mod 0 convention not mentioning the sets of 
zero measure explicitly.

For a $\sigma$-finite no-atomic measure $\mu$ on a standard
Borel space  $(V, \B)$, we denote by 
\be\label{eq Bfin}
\Bfin = \B_{\mathrm{fin}}(\mu)  = \{ A \in \B : \mu(A) < \infty\}
\ee
the algebra of Borel sets of finite measure $\mu$. Clearly, any 
$\sigma$-finite  measure 
$\mu$ is uniquely determined by its values on $\Bfin(\mu)$. 

The linear space of simple function over sets from $\Bfin(\mu)$ is denoted
by
\be\label{eq Dfin}
\ba
\mathcal D_{\mathrm{fin}}(\mu) := & 
\left\{ \sum_{i\in I} a_i \chi_{A_i}  :
 A_i \in \Bfin(\mu),\ a_i \in \mathbb R,\ | I | <\infty\right\}\\
  = &\  \mbox{Span}\{\chi_A : A \in \Bfin(\mu)\}, 
  \ea
\ee
and it will play an important role in the further results since simple 
functions from  $\Dfin(\mu)$ form a norm  dense subset in  
$L^p(\mu)$-space, $ p \geq 1$.

\begin{definition}\label{def symmetric set}
Let $E$ be an uncountable Borel subset of the Cartesian product 
$(V \times V, \B\times \B)$  such that:

(i) $(x, y) \in E\ \Longleftrightarrow\ (y, x) \in E$, i.e. $\theta(E) = E$
where $\theta(x, y) = (y,x)$ is the flip automorphism;

(ii) $ E_x := \{y \in V : (x, y) \in E\} \neq\emptyset,\ \ \forall x \in X$;

(iii) for every $x \in V$,  $(E_x, \B_x)$ is a standard Borel space 
where $\B_x$ is  the $\sigma$-algebra of Borel sets induced on 
$E_x$ from $(V, \B)$.
 
The set $E$ is called  \textit{symmetric subset} of $\vv$. 
\end{definition}

It follows from (ii) and (iii)  of Definition \ref{def symmetric set} that 
the projection of a symmetric set $E$ on each margin of the product 
space $\vv$ is $V$. Moreover, condition (iii) 
assumes two cases: the Borel space $E_x$ can be either countable or
 uncountable. We will assume that $(E_x, \B_x)$ is an uncountable 
Borel standard spaces.

There are several natural examples of symmetric sets related to dynamical 
systems. We mention here the case of a  \textit{Borel equivalence relation}
$E$ on a standard Borel space $(V, \B)$. By definition,  $E$ is a Borel 
subset  of $V\times V$ such that $(x, x) \in E$ for all $x \in V$, 
$(x, y)$ is in $E$ iff $(y, x) $ is in $E$, and $(x, y) \in E, (y, z) \in E$ 
implies that $(x, z)\in E$. Let $E_x = \{ y \in V : (x, y) \in E\}$, then 
$E$ is partitioned into ``vertical fibers'' $E_x$.  In particular, it can be the
case when every $E_x$ is countable. Then $E$ is called a \textit{countable
Borel equivalence relation. }

As mentioned in Introduction, our approach is based on the study of
 \textit{symmetric measures} defined on $\vv$, see Definition
 \ref{def symm measure rho} below. We recall that every measure $\rho$
on the Cartesian product $\vv$  can be disintegrated with respect 
to a measurable partition. For this, denote by $\pi_1$ and $\pi_2$ the 
projections from $V \times V$ onto     
the first and second factor, respectively.  Then  $\{\pi_1^{-1}(x) :
 x \in V\}$ and $\{\pi_2^{-1}(y) : y \in V\}$ are  the 
 \textit{measurable partitions} of 
$V \times V$ into vertical and horizontal fibers, see \cite{Rohlin1949, 
CornfeldFominSinai1982, BezuglyiJorgensen2018} for more information on
 properties of measurable partitions. The case of probability
 measures was studied by Rokhlin in \cite{Rohlin1949}, whereas the 
 disintegration of $\sigma$-finite measures is more delicate and 
 has been considered somewhat recently, see for example 
 \cite{GreschonigSchmidt2000}. We refer here to a result from  
 \cite{Simmons2012} whose  formulation is adapted to our needs. 

\begin{theorem}[\cite{Simmons2012}] \label{thm Simmons} 
For a $\sigma$-finite measure space $(V, \B, \mu)$, 
let $\rho$ be a $\sigma$-finite measure on $(V\times V, \B\times \B)$ 
such that $\rho\circ \pi_1^{-1} \ll \mu$. Then there exists a unique 
system of conditional $\sigma$-finite measures $(\wt\rho_x)$ such that
$$
\rho(f)  = \int_V \wt\rho_x(f)\; d\mu(x), \ \ \ f\in \mathcal F(V	\times 
V, \B\times \B).
$$
\end{theorem}

\begin{remark}\label{rem symm meas}
(1) The condition of  Theorem \ref{thm Simmons} assumes that a measure
$\mu$ is prescribed on the Borel space $\VB$. If one begins with a 
measure $\rho$ defined on $\vv$, then the measure $\mu$ can be taken 
as the projection of $\rho$ on $\VB$,  $\rho\circ \pi_1^{-1} = \mu$.

(2) Let $E$ be a Borel symmetric subset of $\vv$, and let $\rho$ be 
a measure on $\vv$ satisfying the condition of Theorem \ref{thm Simmons}.
Then $E$ can be  partitioned into the fibers $\{x\} \times E_x$. By
 Theorem \ref{thm Simmons}, there exists 
a unique  system of  conditional measures $\wt\rho_x$ such that, for any 
 $\rho$-integrable function $f(x, y)$, we have
\be\label{eq disint for rho}
\iint_{V\times V} f(x, y) \; d\rho(x, y)  = \int_V \wt\rho_x(f) \; d\mu(x). 
\ee
It is obvious that, for $\mu$-a.e. $x\in V$, $\mbox{supp}(\wt
\rho_x) = \{x\} \times E_x$ (up to a set of zero measure). To simplify the
 notation, we will  write 
$$
\int_V f\; d\rho_x \ \ \mathrm{and}  \ \iint_{V \times V} f\; d\rho
$$ 
though the measures $\rho_x$ and $\rho$ have the supports $E_x$ and 
$E$,  respectively. 

 (3) It follows from  Theorem \ref{thm Simmons} that the measure $\rho$
 determines the  measurable field of
sets $x \mapsto E_x \subset V$ and  measurable field of $\sigma$-finite
Borel  measures $x \mapsto \rho_x $ on $(V, \B)$,
 where  the measures $\rho_x$ are defined by the relation
\be\label{eq rho_x def}
\wt\rho_x = \delta_x \times \rho_x. 
\ee
Hence, relation (\ref{eq disint for rho}) can be also written in the following 
form,  used in our subsequent computations,
\be\label{eq disint formula}
\iint_{V\times V} f(x, y) \; d\rho(x, y)  = \int_V \left(\int_V f(x, y) \; 
d\rho_x(y)\right)\; d\mu(x).
\ee
In other words, we have a measurable family of measures $(x \mapsto 
\rho_x)$,  
and it defines a new measure $\nu$ on $(V, \B)$ by setting
\be\label{eq def of nu}
\nu(A) := \int_V \rho_x(A)\; d\mu(x), \quad A \in \B.
\ee
We note that the measure $\rho_x$ 
is defined on the subset $E_x$ of $(V, \B)$, $x \in V$.
\end{remark}

The next definition contains the key notion of a 
\textit{symmetric measure}.

\begin{definition}\label{def symm measure rho}
(1)  Let $(V, \B)$ be a standard Borel space.
We say that a measure $\rho$ on $(V \times V, \B\times B)$ is
 \textit{symmetric} if 
 $$
 \rho(A \times B) = \rho(B\times A), \ \ \ \forall A, B \in \B.
 $$ 
 In other words, $\rho$ is invariant with respect to the flip automorphism
$\theta$.
 \end{definition}

\textit{\textbf{Assumption 1}}. \textit{In this paper,  we will consider 
the set of symmetric measures $\rho$ on $(V\times V, \B\times \B)$ 
which satisfy the following property: 
\be\label{eq c finite}
0 < c(x) := \rho_x(V) <\infty, \ \ \  \ \ \ \mu\mbox{-a.e.}\ x\in V,
\ee
where $x \mapsto \rho_x$ is the measurable field of measures
(the system of conditional mesaures) arising
in Theorem \ref{thm Simmons}.}

\textit{Moreover, we will assume that symmetric measures satisfy the 
condition:  $c(x) \in L^1_{\mathrm{loc}} (\mu)$, i.e.,
\be\label{eq c loc integr}
\int_A c(x)\; d\mu(x) < \infty, \qquad \forall A \in \Bfin(\mu).
\ee
}
\medskip

\begin{remark}
It is natural to assume that the function $c(x)$ satisfies properties
\eqref{eq c finite} and \eqref{eq c loc integr}.  They corresponds to
local finiteness of discrete graphs in the theory of weighted networks.  
In several statements below, we need to use the requirement that
 $c \in \lc2$, i.e.,
$$
  \int_A c^2\; d\mu < \infty, \ \ \ \forall A\in \Bfin(\mu).
$$ 
We observe also that the case when the function $c$  is bounded will lead
to bounded Laplace operators and is not interesting for us.

Based on Assumption 1, we note that relation (\ref{eq def of nu}) defines 
the measure $\nu$ which is equivalent to $\mu$. As stated  in Lemma 
 \ref{lem symm measure via int},  $c(x)$ is the Radon-Nikodym 
 derivative of $\nu$ with respect to $\mu$. If we
want to reverse the definition and use $\nu$ as a primary measure, then 
we need to require that the function $c(x)^{-1}$ is locally integrable 
with respect to $\nu$. 
\end{remark}

In the following remark, we discuss some natural properties of symmetric
measures. 

\begin{remark}\label{rem on symm meas}
(1) If $\rho$ is a symmetric measure on $(V \times V, \B\times \B)$,
 then the support of $\rho$, the set $E = E(\rho)$, is symmetric mod 0. 
Hence, without loss of generality, we can assume  that any  symmetric 
measure is supported by the
sets $E$ satisfying Definition \ref{def symmetric set}. We will 
require that, for every $x \in V$, the set $E_x\subset E$ is an 
uncountable standard Borel space. The case 
when $E_x$ is countable arises for a countable Borel 
equivalence relation $E$   on $(V, \B)$. The latter was considered in
\cite{BezuglyiJorgensen_2018}. For countable sets $E_x, x \in V$, we can
take $\rho_x$ as a finite measure which is equivalent to the counting
measure, see, e.g.  \cite{FeldmanMooreI_1977, FeldmanMooreII_1977, 
KechrisMiller2004} for details. 

(2) In general, the notion of a symmetric measure is defined in the context 
of standard Borel spaces $(V, \B)$ and $(V\times V, \B\times \B)$. 
But if a $\sigma$-finite measure $\mu$ is given on $(V, \B)$, then we 
have to include an additional relation between the measures $\rho$ and  
$\mu$. Let $\pi_1 : V \times V \to V$ be the projection on the first 
coordinate. We require that  $\rho\circ \pi_1^{-1} \ll \mu$, see Theorem 
\ref{thm Simmons}.

(3) The symmetry of the set $E$ allows us to define 
a ``mirror'' image of the measure $\rho$. Let $E^y := \{x \in V : (x, y) \in 
E\}$, and let $(\wt\rho^y)$ be the system of conditional measures with 
respect to the partition of $E$ into the sets $E^y \times \{y\}$. Then, 
for the measure 
$$
\wt \rho = \int_V \wt \rho^y d\mu(y),
$$
the relation $\rho = \wt \rho$ holds. 

(4) It is worth noting that, in general, when a measure $\mu$ is defined 
on $(V, \B)$, the set $E(\rho)$ do not need to be 
a set of positive measure with respect to the product measure
$\mu\times \mu$. In other words, we admit both cases: (a) $\rho$ is 
equivalent to $\mu\times\mu$ on the set $E$, and (b) $\rho$ and
$\mu\times\mu$ are mutually singular. 
\end{remark}

The following (important for us) fact is deduced from the definition of
 symmetric measures. We emphasize that formula 
(\ref{eq formula fo symm meas}) will be used repeatedly in many proofs. 

\begin{lemma} \label{lem symm measure via int}
(1) For a symmetric measure $\rho$ and any bounded Borel function
$f$ on $(V\times V, \B\times \B)$, 
\be\label{eq formula fo symm meas}
\iint_{V \times V} f(x, y) \; d\rho(x, y) = \iint_{V \times V} f(y, x) \; 
d\rho(x, y). 
\ee
Equality (\ref{eq formula fo symm meas}) is understood in the sense of 
the extended real line, i.e., the infinite value of the integral is allowed. 

(2) Let $\nu$ be defined as in (\ref{eq def of nu}). Then 
$$
d\nu(x) = c(x) d\mu(x).
$$ 
\end{lemma}

We define below the key notion of  an 
\textit{irreducible symmetric measure}. 

For a standard measure space $(V, \B)$, a  \textit{kernel} 
$k$ is a function $k:  V \times \B \to \ol \R_+ = [0, \infty]$ such that

(i) $x \mapsto k(x, A)$ is measurable for every $A \in \B$;

(ii) for any $x\in V$, $k(x, \cdot)$ is a $\sigma$-finite measure on 
$(V, \B)$.
\medskip

A kernel $k(x, A)$ is called \textit{finite} if $k(x, \cdot)$ is a finite
 measure on  $(V, \B)$ for every $x$. We will also use the notation 
 $k(x, dy)$ for the measure on $\VB$.
   
Given a kernel $k$, one can construct inductively the sequence 
of kernels $(k^n : n \geq 1)$ by setting
\be\label{eq_powers of k}
k^n(x, A) = \int_V k^{n-1}(y, A)\; k(x, dy),\qquad n > 1.
\ee

 It is said that a \textit{set $A \in \B$ 
is \textit{attainable}  from $x\in V$} if there exists  $n\geq 1$ such that 
$\mc K^n(x, A)  >0$; in symbols, we write $x \rightarrow A$. A set 
$A \in \B$ is called  \textit{closed} for the kernel $\mc K$ if 
$\mc K(x, A^c) = 0$ for all $x\in A$. If $A$ is closed, then it follows 
from \eqref{eq_powers of k} that
 $\mc K^n(x, A^c) = 0$ for any $n \in \N$ and $x \in A$.  Hence, $A$ is
  closed if and only if  $ x \nrightarrow   A^c$, see details in 
  \cite{Nummelin1984, Revuz1984}.

Every symmetric measure $\rho$ defines a finite \textit{kernel} 
$\mc K = \mc K(\rho)$ where 
$$
 V \times \B \stackrel{\mc K}\longrightarrow [0, \infty) : \mc K(x, A) = 
 \rho_x(A), 
$$
and $x\mapsto \rho_x(\cdot)$ is the measurable  family of conditional 
measures from Theorem \ref{thm Simmons}. 

 \begin{definition}\label{def_irreducible}
(1)  A kernel $x \to k(x, \cdot)$ is called \textit{irreducible with 
respect  to a $\sigma$-finite  measure  $\mu$ on $\VB$ 
($\mu$-irreducible)} if, for any set $A$ of 
positive measure $\mu$ and $\mu$-a.e. $x\in V$, there exists some $n$ 
such that   $k^n(x, A) >0$,  i.e., any set $A$ of positive measure  is
 attainable from $\mu$-a.e. $x$.

(2) A \textit{symmetric measure} $\rho$ on $\vv $ is called
  \textit{irreducible} if the corresponding kernel $\mc K(\rho) : (x, A) 
 \to\rho_x(A)$ is $\mu$-irreducible where $\mu$ is the projection of 
 measure  $\rho$ onto $\VB$.  

(3) A symmetric measure $\rho$ (or the kernel $ x \to
   \rho_x(\cdot)$) is called $\mu$-\textit{decomposable} if 
 there exists a Borel subset $A$ of $V$ of positive measure $\mu$ 
 such that 
\be\label{eq decomposable 1}
 E \subset (A\times A) \cup (A^c \times A^c)
 \ee
 where $A^c = V \setminus A$ is also of positive measure. 
 Otherwise, $\rho$ is called  \textit{indecomposable}. 
\end{definition} 

Let  $\rho$ be a symmetric measure on $\vv$. 
For any fixed $x\in V$, we define a sequence of subsets: 
$A_0(x) = \{x\}$,  $A_1(x) = E_x,$ 
$$
  A_n(x) =  \bigcup_{y\in A_{n-1}(x)} E_y, \ \ \ n \geq 2.
$$ 
Recall that $E_x$ is the support of the measure $\rho_x$, and $E_x$ can 
be identified with the vertical section of the symmetric set $E$. Note that 
$A_n(x) \in \B$ as $x\to E_x$ is a measurable  field of sets.
 
\begin{lemma} \label{lem-irr measure}
(1) Given $\sms$, a symmetric  measure $\rho$ is irreducible if
 and only if for $\mu$-a.e. $x \in V$  and any set $B\in \B$  of positive
  measure $\mu$ there exists $n\geq 1$  such that 
  \be\label{eq_irr measure via A_n}
  \mu(A_n(x) \cap B) >0. 
  \ee
  
 (2)  Let $\mc K(x, A) = \rho_x(A)$. Suppose that the support of 
$\rho$, the set $E$, satisfies relation \eqref{eq decomposable 1} 
where $\mu(A) >0$ and $\mu(A^c) >0$. Then the sets $A$ and $A^c$ 
are closed ,and  $x\mapsto \rho_x(A)$ is a $\mu$-reducible kernel. 
The converse statement also holds. 
\end{lemma}

This lemma was proved in \cite{BezuglyiJorgensen_2018}.

\subsection{Symmetric measures and associated operators of 
Hilbert spaces}\label{subsect symm meas operators}

Let $\sms$ be a $\sigma$-finite measure space, and let $\rho$ be a 
symmetric measure on $\vv$ supported by a symmetric set $E$. Let  
$x\mapsto \rho_x$ be the measurable family of measures on $\VB$ that 
 \textit{disintegrates} 
$\rho$. Recall that, by Assumption 1, the function  $c(x) = \rho_x(V)$ is
 finite for $\mu$-a.e. $x$. As discussed  Subsection 
 \ref{subsect sms symm meas}, the measure $\rho$   generates a 
finite kernel $\mc K(\rho)$ which we use to define the following
 operators. 

\begin{definition}\label{def R, P, Delta}
For  a symmetric measure $\rho$  on $\vv$, we
define three linear operators $R, P$, and $\Delta$ acting on the space of
bounded Borel functions $\FVB$.\\

(i) The \textit{symmetric operator}$R$:
\be\label{eq def of R} 
R(f)(x) := \int_V f(y) \; d\rho_x(y) = \rho_x(f). 
\ee 

(ii) The\textit{ Markov operator} $P$:
$$
P(f)(x) = \frac{1}{c(x)}R(f)(x)
$$
or
\be\label{eq formula for P}
 P(f)(x) := \frac{1}{c(x)}  \int_V f(y) \; d\rho_x(y) = \int_V f(y) \; 
 P(x, dy)
 \ee
 where  $P(x, dy)$ is the probability measure obtained by normalization
 of $d\rho_x(y)$, i.e. 
 $$
 P(x, dy) := \frac{1}{c(x)}d\rho_x(y).
 $$
 
 (iii) The \textit{graph Laplace operator} $\Delta$:
 \be\label{eq def of Delta}
\Delta(f)(x) := \int_V (f(x) - f(y)) \; d\rho_x(y)
\ee
or 
\be\label{eq Delta via R}
\Delta(f) = c(I - P)(f) = (cI  - R)(f).
\ee
\end{definition}

\begin{remark}\label{rem R and rho}  We assemble in this remark 
several obvious properties of the defined operators. 

(1) The operator $R$ (and therefore $P$) is positive, in the sense that 
$f \geq 0$ implies $R(f) \geq 0$. Moreover, $R(\mathbbm 1) = c(x)$ 
and therefore, using (\ref{eq c finite}), we can write the operator 
$\Delta$ in more symmetric form:
$$
\Delta(f) = R(\mathbbm 1)f  - R(f)
$$
where $\mathbbm 1$ is a function on $V$ identically equal to $1$. 

For the operator $P$, we have  $P(\mathbbm 1) =\mathbbm 1$. 
It justifies the name of \textit{Markov operator} used for $P$. 
Conversely, any Markov operator $P$ defines the probability kernel 
  $x \mapsto P(x, A) = P(\chi_A)(x)$ called also the 
\textit{transition probabilities}. 

(2) The definition of each of the operators $R$, $P$, and $\Delta$
 depends on a measure $\rho$ on $\vv$, and, strictly speaking,
they must be denoted as  $R(\rho)$, $P(\rho)$, and $\Delta(\rho)$.
Since most of our results are proved for a fixed measure $\rho$, we will
omit $\rho$ in our notation. 

More generally, Definition \ref{def R, P, Delta} makes sense for arbitrary 
measure $\rho$. The case of symmetric measure is more interesting and 
allows to prove deeper results. 

(3) It is worth noting that the Laplacian $\Delta$ can be defined using
a different approach. In Theorem \ref{thm mu_f and Delta}, it is proved 
that $\Delta(f)$  is the Radon-Nikodym derivative of a measure $\mu_f$ 
with respect to $\mu$.

(4) Since every measure  $\rho$  on $V \times V$  is uniquely determined
 by its values on a dense subset of functions, 
it suffices to define $\rho$ on the set of the so-called ``cylinder  
functions'' $(f \otimes g)(x, y) := f(x)g(y)$. This observation will be used
below when we prove some relations for cylinder functions first. 

(5) In general, a positive operator $R$ in $\FVB$ is called
 \textit{symmetric} if it satisfies the relation:
\be\label{eq symm in terms R}
\int_V f R(g)\; d\mu = \int_V R(f) g\; d\mu,
\ee
for any $f, g \in F(V, \B)$. It can be easily seen that $R$ is symmetric if
and only if it defines a symmetric measure $\rho$ by the formula
$$
\rho(A\times B) = \int_V \chi_A(x) R(\chi_B)(x)\; d\mu(x).
$$

(6) In Definition \ref{def R, P, Delta}, we do not discuss domains of
the operators $R, P$, and $\Delta$. They depend on the space where an 
operator is considered. In the current paper, we work with $L^2$- spaces 
defined by the measures $\mu, \nu$, and $\rho$. But the most 
intriguing is the case of the finite energy space Hilbert space $h_E$ 
which is defined below.  
\end{remark}

We are interested in the following \textit{question}. Let $x\mapsto \rho_x$ 
be a measurable field of finite measures over a standard Borel space $\VB$. 
How can one describe the set of Borel $\sigma$-finite measures $\mu$ on 
$\VB$ such that the measure $\rho = \int_V \rho_x d\mu(x)$ would be
a symmetric measure on $\vv$. A partial answer was given in Remark
\ref{rem R and rho}: if the operator $R : f \mapsto \int_V f \; d\rho_x$ 
satisfies \eqref{eq symm in terms R}, then the pair $\{(x \mapsto \rho_x),
\mu\}$ determines a symmetric measure. 

In the following proposition we clarify relations between symmetric 
measures $\rho$ and symmetric operators $R$ that were defined in
Remark \ref{rem R and rho}.

\begin{proposition}\label{prop symm rho vs R}
Let $x\mapsto \rho_x$ be a measurable field of finite measures over 
a standard Borel space $\VB$ which defines the operator $R$. 
Suppose that $\mu$ is a measure on 
$\VB$ such that  \eqref{eq symm in terms R} holds. Let $p(x)\in \Lloc$
 be a positive Borel function and $R'(f) := R(fp)$.

(1) The measure 
$$
\rho'(A \times B) := \int_V \chi_A R'(\chi_B)\; d\mu 
$$
is symmetric on $\vv$ if and only if $R\circ M_p = M_p\circ R$ where 
$M_p$ is the operator of multiplication by $p$. 

(2) Given a positive Borel function $p\in  \Lloc$ and the measure 
$d\beta(x) = p(x) d\mu(x)$, the measure $\rho_{\beta} = \int_V 
\rho_x d\beta(x)$ is symmetric on $\vv$ if and only if the operator 
$R'$ is symmetric with respect to the pair $\{(x \mapsto \rho_x),
\beta\}$. 
\end{proposition}

\begin{proof}
To show that (1) holds, we use the fact that the measure 
 $\rho = \int_V \rho_x d\mu(x)$ is symmetric:
 $$
 \ba
\rho'(A\times B) & = \int_V \chi_A R'(\chi_B)\; d\mu \\
& =  \int_V \chi_A \left(\int_V p\chi_B\; d\rho_x\right)\; d\mu \\
& = \iint_{\VtV} \chi_A(x) p(y)\chi_B(y)\; d\rho(x,y)\\
& = \iint_{\VtV} \chi_A(y) p(x)\chi_B(x)\; d\rho(x,y)\\
& = \int_V \chi_B p R(\chi_A)\; d\mu. \\
\ea
$$
On the other hand,
$$
\rho'(B\times A)  = \int_V \chi_B R'(\chi_A)\; d\mu. 
$$
Then $\rho'(A\times B) = \rho'(B\times A)$ if and only if 
$R'(\chi_A) = p R(\chi_A)$. By linearity, the latter is extended to the
relation $R\circ M_p = M_p\circ R$.

(2) For this, we compute
$$
\ba 
\int_V R'(f) g \; d\beta &= \int_V R(fp) g \; d\beta\\
& = \int_V R(fp) g p\; d\mu\\
& = \int_V (fp) R(g p)\; d\mu\\
& = \int_V f R'(g)\; d\beta\\
\ea
$$
\end{proof}

\begin{corollary}
Let $\rho$ and $\rho'$ be two symmetric measures on $\vv$ defined
by the pairs $\{(x \mapsto \rho_x), \mu\}$ and 
$\{(x \mapsto \rho_x), \mu'\}$, respectively. Then the Laplace 
operators $\Delta(\rho)$ and $\Delta(\rho')$ coincide.
\end{corollary}
 
In the following statement we discuss properties of $R, P$, and $\Delta$
as operators acting in $L^2$-spaces, see 
\cite{BezuglyiJorgensen2018, BezuglyiJorgensen_2018} for details. 

\begin{theorem}\label{prop prop of R, P, Delta}
For a standard measure space $\sms$, let $\rho$ be a symmetric 
measure on $\vv$ such that $c(x) = \rho_x(V)$ satisfies Assumption 1.  
Let $d\nu(x) = c(x) d\mu(x)$ be the $\sigma$-finite  measure on 
$(V, \B)$ equivalent to $\mu$, and let the operators $R, P$, and 
$\Delta$ be defined as in Definition \ref{def R, P, Delta}. 

(1) Suppose that the function $x \mapsto \rho_x(A) \in L^2(\mu)$ for 
every $A \in \Bfin$\footnote{This means that the operator $R$ is densely 
defined on functions from $\Dfin(\mu)$; in particular, this property holds
if $c \in L^2(mu)$}.  Then $R$ is a symmetric unbounded operator  in 
$L^2(\mu)$, i.e.,   
\be\label{eq-R symm}
\langle g, R(f) \rangle_{L^2(\mu)} = \langle R(g), f \rangle_{L^2(\mu)}.
\ee
If $c \in L^\infty(\mu)$, then $R :  L^2(\mu) \to 
L^2(\mu)$ is a bounded operator, and 
$$
||R||_{L^2(\mu) \to L^2(\mu)} \leq ||c||_{\infty}.
$$
Relation \eqref{eq-R symm} is equivalent to the symmetry of the measure
$\rho$. 

(2) The operator $R : L^1(\nu) \to L^1(\mu)$ is contractive, i.e.,  
$$
||R(f)||_{L^1(\mu)} \leq ||f||_{L^1(\nu)}, \qquad f \in L^1(\nu).
$$
Moreover,   for any function $f \in L^{1} (\nu)$,  the formula 
\be\label{eq c(x) and rho_x}
\int_V R(f) \; d\mu(x) = \int_V f(x) c(x) \; d\mu(x)
\ee
holds. In other words, $\nu = \mu R$, and 
$$
\frac{d(\mu R)}{d\mu}(x) = c(x).
$$

(3) The bounded operator $P: L^2(\nu) \to L^2(\nu)$ is self-adjoint.
Moreover, $\nu P = \nu$. 

(4) The operator $P$, considered in the spaces $L^2(\nu)$ and $L^1(\nu)$, 
is  contractive, i.e., 
$$
|| P(f) ||_{L^2(\nu)} \leq || f ||_{L^2(\nu)}, \qquad || P(f) ||_{L^1(\nu)} \leq 
|| f ||_{L^1(\nu)}.
$$ 

(5) The spectrum of $P$ in $L^2(\nu)$ is a subset of $[-1, 1]$.

(6) The graph Laplace operator $\Delta : L^2(\mu)  \to L^2(\mu) $ 
is a positive definite essentially
self-adjoint operator with domain containing $\Dfin(\mu)$. Moreover, 
$$
|| f||^2_{\h_E} = \int_V f \Delta(f)\; d\mu
$$
when the integral in the right hand side exists. 
\end{theorem}

\begin{remark}\label{rem nonsymm meas}
Suppose that  a non-symmetric measure $\rho$ is given on the
space $\vv$, i.e, $\rho(A \times B) \neq \rho(B\times A)$, in general. 
However, we will assume that $\rho$ is equivalent to $\rho\circ
\theta$ where $\theta(x, y) = (y, x)$. Then we can define 
the following objects: margin measures 
$\mu_i := \rho\circ \pi_i^{-1}, i =1,2,$, fiber measures $d\rho_x(\cdot)$
and $d\rho^x(\cdot)$ (see Remark \ref{rem on symm meas}),  and 
functions $c_1(x) = \rho_x(V), c_2(x) = \rho^x(V)$.  

Define now the \textit{symmetric measure} $\rho^{\#}$ generated by 
$\rho$:
$$\rho^{\#} := \dfrac{1}{2}(\rho + \rho\circ\theta).
$$
 Then
$$
\rho^{\#}(A \times B) =\frac{1}{2}(\rho(A\times B)+ \rho(B \times A)).
$$
Clearly, $\rho^{\#}$ is equivalent to $\rho$.

Let $E\subset \VtV$ be the support of $\rho$. Then $E^{\#} = E \cup 
\theta(E)$ is the  support of the symmetric measure  $\rho^{\#}$.  
The disintegration of $\rho 
= \int_V\rho_x\; d\mu_1(x)$ with respect to the partition 
$\{x\} \times E_x$ defines the 
disintegration of $\rho^{\#}$. For $\mu^{\#} := \dfrac{1}{2}(\mu_1 +
 \mu_2) $, we obtain that 
 $$
 \rho^{\#} = \int_V (\rho_x + \rho^x) \; d\mu^{\#}.
 $$

Having the symmetric measure $\rho^{\#}$ defined on $\vv$, 
we can introduce 
the operators $R^{\#}$ and $P^{\#}$ as in \eqref{eq def of R} and
\eqref{eq formula for P}. It turns out that, for $f\in \FVB$,
$$
R^{\#}(f)(x) = R_1(f)(x) + R_2(f)(x) 
$$
where
$$
R_1(f) = \int_V f(y) \; d\rho_x(y),\qquad R_2(f) = \int_V f(y) \; 
d\rho^x(y).
$$ 
Similarly, 
$$
P^{\#}(f)(x)  = \frac{1}{c^{\#}(x)} R^{\#}(f)(x)
$$
where
$$
c^{\#}(x) = \rho_x( V) + \rho^x(V).
$$ 
Then we can define the measure
$d\nu^{\#}(x) = c^{\#}(x) d\mu(x)$ such that the operator
$$
P^{\#} (f)(x) = \int_V f(y)  \frac{1}{c^{\#}(x)}\; d\rho^{\#}_x(y)
$$
is self-adjoint in $L^2(\nu^{\#})$. By Theorem \ref{prop_reversible} (see
below), we obtain that the Markov process generated by $x \mapsto 
P^{\#}(x, \cdot)$ is \textit{reversible} where $P^{\#}(x, A) = 
P^{\#}(\chi_A)(x)$. 
\end{remark}

Here we define an important class of functions that will be discussed 
throughout the paper.

\begin{definition}\label{def_harmonic}
A  function $f\in \FVB$ is called \textit{harmonic}, if $Pf = f$.
Equivalently, $f$ is harmonic if $\Delta f = 0$ or $R(f) = cf$. Similarly, 
$h$ is \textit{harmonic} for a kernel $x\to k(x, \cdot)$ if 
$$
\int_V h(y) \; k(x, dy) = h(x).
$$
The set of harmonic functions will be denoted by $\mc Harm(P)$.
\end{definition}

It turns out that $P$-harmonic functions cannot lie in the space $L^2(\nu)$
where  $\nu$ is $P$-invariant.

\begin{theorem}[\cite{BezuglyiJorgensen2018}] 
\label{thm harmonic} Given a standard measure space 
$(V, \B, \nu)$, let $P$ be  a Markov operator
on  $L^2(\nu)$ defined by transition probabilities as in 
(\ref{eq formula for P}). Suppose  that $\nu P = \nu$.  Then  
$$
 L^2(\nu) \cap \h arm(P) = \begin{cases} 0,  & \nu(V) = \infty\\
\mathbb R\mathbbm 1 , & \nu(V) < \infty
 \end{cases}
$$ 
Moreover, $1$ does not belong to the point spectrum of the operator 
$P$ on the space $L^2(\nu)$.  
\end{theorem}

In what follows, we discuss briefly the relationship between symmetric 
measures and the notion of \textit{reversible Markov processes}. 

Recall our setting. Let $\VB$ be a standard Borel space, $\rho$ a 
symmetric $\sigma$-finite measure on $\vv$ satisfying Assumption 1,
$\mu$ the projection of $\rho$ on $\VB$, $c(x) = \rho_x(V)$ where
$x\mapsto \rho_x$ is the system of conditional measures. 
Let  $P$ denote the Markov operator defined by the family of 
transition probabilities $x \mapsto  P(x, \cdot)= c(x)^{-1}d\rho_x(y)$:
\be\label{eq-P via P(x,dy)}
P(f)(x)  = \int_V f(y) \; P(x, dy)
\ee
Furthermore, we can use the kernel $x \to P(x, \cdot) = P_1(x, \cdot)$ to
define the sequence of probability kernels (transition probabilities)
$(P_n(x, \cdot) : n \in \N)$ in accordance with (\ref{eq_powers of k}): 
$$
 P_{n+m}(x, A)  = \int_V P_n(y, A)  P_m(x, dy), \qquad n, m \in \N.
$$
Therefore, one has
$$
P^n(f)(x)  = \int_V f(y) \; P_n(x, dy), \qquad n \in \N,
$$ 
and this relation defines the sequence of probability measures 
$(P_n)$ by setting $P_0(x, A) = \delta_A(x) = \chi_A(x)$ and
$$
P_n(x, A) = P^n(\chi_A) = \int_V \chi_A(y) \; P_n(x, dy), 
\qquad A\in \B, n\in \N.
$$

Using the  Markov operator $P$, we can define the sequence 
of measures $(\rho_n)$ (here $\rho_1 = \rho$) on $\vv$ by  setting 
\be\label{eq_def rho_n}
\rho_n(A \times B) = \langle \chi_A, P^n(\chi_B)\rangle_{L^2(\nu)},
\quad n \in \N. 
\ee
Then it can be easily seen that the following properties hold.

\begin{lemma} For the measures $\rho_n$, defined by \eqref{eq_def rho_n},
we have:

(i) every measure $\rho_n, n \in \N,$ is symmetric on $\vv$, and
$\rho_n $ is equivalent to $\rho$;

(ii) $\rho_x^{(n)}(V) = c(x)$;

(iii) 
\be\label{eq_rho_n via P_N}
d\rho_n(x,y) = c(x) P_n(x, dy)d\mu(x)= P_n(x, dy)d\nu(x);
\ee

(iv) 
$$
|| P^n(\chi_A)||^2_{L^2(\nu)} = \rho_{2n} (A\times A);
$$

(iv) 
$$
\rho_n( A\times B) = \langle \chi_A, RP^{n-1}(\chi_B)
\rangle_{L^2(\mu)}.
$$
\end{lemma}

\begin{definition}\label{def reversible MP-1} 
Suppose that $x \mapsto P(x, \cdot )$ is a measurable family of transition
 probabilities on the space $\sms$, and let $P$ be the  Markov operator 
determined by $x \mapsto P(x, \cdot )$. It is said that the corresponding
 Markov process  is  \textit{reversible} with respect to a measurable 
 functions $c: x \to  (0, \infty)$ on $\VB$ if, for any sets $A, B \in \B$, the
following relation holds:
\be\label{eq def reversible P}
 \int_B c(x) P(x, A)\; d\mu(x) = \int_A c(x) P(x, B)\; d\mu(x).
\ee
\end{definition} 

It turns out that the notion of reversibility is equivalent to the following
properties.

\begin{theorem}[\cite{BezuglyiJorgensen2018, BezuglyiJorgensen_2018}]
\label{prop_reversible}
Let $\sms$ be a standard $\sigma$-finite measure space, $x \mapsto c(x)
\in (0, \infty)$ a measurable function, $c \in \Lloc$. Suppose that 
$x \mapsto P(x, \cdot )$ is a probability kernel. 
The following are equivalent:

(i) $x \mapsto P(x, \cdot )$ is reversible (i.e., it satisfies 
(\ref{eq def reversible P}); 

(ii) $x\to P_n(x, \cdot)$ is reversible for any $n\geq 1$;

(iii) the Markov operator $P$ defined by $x\to P(x, \cdot)$ is  self-adjoint 
on $L^2(\nu)$ and $\nu P= \nu$ where $d\nu(x) = c(x) d\mu(x)$;

(iv) 
$$
c(x) P(x, dy) d\mu(x) = c(y) P(y, dx)d\mu(y);
$$

(v) the operator $R$ defined by the relation $R(f)(x) = c(x)P(f)(x)$
is symmetric;

(vi) the measure $\rho$ on $(V\times V, \B\times \B)$ defined by 
$$
\rho(A \times B) = \int_V \chi_A R(\chi_B)\; d\mu = 
\int_V c(x) \chi_A P(\chi_B)\; d\mu
$$
is symmetric;

(vii) for every  $n \in \N$, the measure $\rho_n$ defined by 
(\ref{eq_def rho_n}) is symmetric.
\end{theorem}
 
We finish this section discussing the following problem. We recall briefly 
our main setting. Let $\rho$ be a 
symmetric measure on $\vv$. Then it generates the following objects:
the marginal measure  $\mu = \rho\circ \pi^{-1}_1$ such that
$\rho = \int_V \rho_x d\mu$, function $c(x) = \rho_x(V)$, and the measure
$\nu = c \mu$. The corresponding Markov operator $P$ is defined by 
the measurable field of transition probabilities $x \mapsto P(x, \cdot)$ 
where $c(x) P(x, dy) = d\rho_x(y)$, see \eqref{eq-P via P(x,dy)}.
 It was proved that the measure $\nu$ is $P$-invariant. 

Suppose now we have two symmetric measures $\rho$ and $\rho'$ 
defined on $\vv$ which are equivalent; let
$$
\frac{d\rho'}{d\rho}(x,y) = r(x, y).
$$
 Then $r(x,y) = r(y,x)$, see \cite{BezuglyiJorgensen_2018} for details.
 Let the collection $(\mu', \rho'_x, c', \nu', P')$ be defined by 
 the measure $\rho'$. 
 
 \begin{proposition}
 Suppose that, for two equivalent symmetric measures $\rho$ and $\rho'$ 
 on $\vv$, the Markov operators $P$ and $P'$ coincide. Then $\nu'$ is a
constant multiple of $\nu$. 
 \end{proposition}
 
\begin{proof}
Since $\rho \sim \rho'$, the margin measures $\mu$ and $\mu'$ are also
equivalent, so that the Radon-Nikodym derivative 
$$
\frac{d\mu}{d\mu'}(x) = m(x)
$$ 
is positive a.e. It can be seen that 
\be\label{eq RN for rho_x}
\frac{d\rho'_x}{d\rho_x}(y) = r_x(y)m(x)
\ee
where $r_x(\cdot) := r(x, \cdot)$. If $P = P'$, then $P(x, dy) = 
P'(x, dy)$ because $P(x, A) = P(\chi_A)(x)$. Hence,
$$
\frac{d\rho'_x}{d\rho_x}(y) = \frac{c'(x)}{c(x)},
$$
and we obtain from \eqref{eq RN for rho_x} that 
$$
r_x(y) = \frac{c'(x)}{c(x)m(x)}.
$$
This means that $r(x, y)$ depends on the variable $x$ only. But by
symmetry of the function $r$, we conclude that $r(x, y)$ is a constant
a.e. This proves the result.  
\end{proof}

\section{\textbf{Finite energy space: definitions and first results}}
\label{subsect_finiteEnergy}
 The finite energy space $\h_E$, see Definition \ref{def f.e. space} below,
  is one of the central notions of this article.
 It can be viewed as a generalization of the energy space for discrete 
 weighted  networks which have been extensively studied during last 
 decades. We partially follow the paper \cite{BezuglyiJorgensen2018} where 
 this space  was  defined and studied. 
 
\subsection{Inner product and norm in $\h_E$}\label{subsect Inn prod}

\begin{definition}\label{def f.e. space} Let $(V, \B, \mu) $ be a standard
 measure space with $\sigma$-finite measure $\mu$. Suppose that
 $\rho$ is a symmetric measure on the Cartesian product $(V\times V,
 \B\times \B)$ whose projection on $V$ is $\mu$.  We say that a Borel 
function $f : V \to \mathbb R$ belongs to the \textit{finite energy space} 
$\mathcal H_E=  \mathcal H_E(\rho) $ if
\be\label{eq def f from H}
\iint_{V\times V}  (f(x) - f(y))^2 \; d\rho(x, y) < \infty.
\ee
\end{definition}

To simplify the notation, we will also use the symbol $\h$ to denote the 
finite energy space  $\h_E$.

\begin{remark}\label{rem H depends on rho} 
(1) It follows from Definition \ref{def f.e. space} that   $\mathcal H_E$ is a
 vector space containing 
all constant functions. We identify functions $f_1$ and $f_2$ such that
$f_1 - f_2 = const$ and, with some abuse of notation, the quotient space
is also denoted by $\h_E$. So that we will call elements $f$ of $\h_E$  
functions assuming that a representative of the equivalence class
$f$ is considered. It will be easily seen that our results do not depend
on the choice of representatives.

(2) Definition \ref{def f.e. space} assumes that a symmetric irreducible 
measure $\rho$ is fixed on $(V\times V, \B\times \B)$. This means that
the space of functions $f$ on $(V, \B)$ satisfying (\ref{eq def f from H})
depends on $\rho$, and, strictly speaking, this space must be denoted as 
$\h_E(\rho)$. As was mentioned above, it is a challenging problem to 
study relations between $\h_E(\rho)$ and $\h_E(\rho')$ for equivalent
measures $\rho$ and $\rho'$, see 
\cite[Theorem 4.11]{BezuglyiJorgensen_2018} for a discussion.
\end{remark}
 
Define a bilinear form $\xi(f, g)$ in the space $\mathcal H$ by the formula
\be\label{eq inner product}
\xi(f, g) := \frac{1}{2} \iint_{V \times V}
(f(x) - f(y))(g(x) - g(y)) \; d\rho(x, y).
\ee
We denote $\xi(f) = \xi(f, f)$. Setting 
$$
\langle f, g \rangle_{\mathcal H} = \xi(f,g),
$$ we define an \textit{inner  product} on the space $\h_E$. Then 
\be\label{eq norm in H_E}
|| f ||^2_{\mathcal H_E} := \frac{1}{2} 
\iint_{V\times V}  (f(x) - f(y))^2 \; d\rho(x, y), \qquad f \in \mathcal H,  
\ee
turns $\h_E$ in a  normed vector space.
As proved in \cite{BezuglyiJorgensen2018}, $\h_E$ is a \textit{Hilbert 
space} with respect to the norm $|| \cdot ||_{\h_E}$. 

We remark that the zero vector in $\h_E$ is represented by a constant
function.  

The definition of the Hilbert  space $\h_E$ and the norm 
given in (\ref{eq norm in H_E}) allows us to define an embedding of
the energy space into $L^2(\rho)$.  The following lemma is based
on the definition of the norm in $\h_E$.

\begin{lemma}
The map 
$$
\partial : f(x)  \mapsto (\partial f)(x, y) := \frac{1}{\sqrt{2}} \ 
(f(x) - f(y))
$$
is an isometric embedding of the space $\h_E$ into $L^2(\rho)$. 
\end{lemma}

We will  use an assumption about \textit{local integrability} of functions
from the energy space. We will discuss the 
following  assumption in detail below.
\medskip

\textbf{\textit{Assumption 2}}. Let a symmetric measure $\rho$ on 
$\vv$ be chosen so that all functions from $\h_E = \h_E(\rho)$ are 
locally square integrable, i.e.,
\be\label{eq assumption 2}
\h_E \subset L^2_{\mathrm{loc}} (\mu).
\ee
Since $L^2_{\mathrm{loc}} (\mu) \subset \Lloc$, we can always assume 
that functions from the space $\h_E$ are locally integrable.
\medskip 

It is worth noting that this assumption is very mild. It holds automatically
for the case of locally finite discrete weighted networks. 

We consider first some immediate properties of functions from
the space $\h_E$. These properties have been discussed in 
\cite{BezuglyiJorgensen2018}.

\begin{proposition}\label{prop ||chi A||}
(1) Let $\rho$ be a symmetric measure on $\vv$ such that $\mu = \rho\circ 
\pi_1^{-1}$. Suppose that $c(x)= \rho_x(V)$ is locally integrable with 
respect  to $\mu$. Then:  
(1) 
\be\label{eq three inclusions}
\mathcal D_{\mathrm{fin}}(\mu)  \subset 
\mathcal D_{\mathrm{fin}}(\nu) \subset \h_E
\ee
where $d\nu(x) = c(x)d\mu(x)$. Moreover, if $A \in \Bfin(\nu)$, then 
\be\label{eq ||chi A||}
|| \chi_A ||^2_{\h_E} = \rho(A \times A^c) \leq \int_A c(x) \; d\mu(x)
= \nu(A),
\ee
where $A^c := V \setminus A$;

(2) for every $A \in \B$, the function $\chi_A$ belongs to $\h_E$
if and only if $\chi_{A^c}$ is in $\h_E$; moreover, $\| \chi_A \|_{\h_E}
 = \| \chi_{A^c} \|_{\h_E}$ and 
 $$
 \langle \chi_A, \chi_{A^c} \rangle_{\h_E} = - \rho(A \times A^c).
 $$
In other words, the function $\chi_A $ is 
 in $\h_E$ if and only if either $\mu(A) < \infty$ or $\mu(A^c) < \infty$:
 $$
 \Bfin(\nu) \cap (\Bfin(\nu))^c \subset \h_E.
 $$ 
\end{proposition}

\begin{proof}
For (1), let $A \in \Bfin(\mu)$, then $\nu(A) = \int_A c d\mu <\infty$
because $c$ is a locally integrable function. Hence,  $\mathcal 
D_{\mathrm{fin}}(\mu)  \subset \mathcal D_{\mathrm{fin}}(\nu)$. 
To prove that $\mathcal D_{\mathrm{fin}}(\nu) \subset \h_E$, we 
compute the norm of $\chi_A$:
$$
|| \chi_A ||^2_{\mathcal H_E} := \frac{1}{2} 
\iint_{V\times V}  (\chi_A(x) - \chi_A(y))^2 \; d\rho(x, y) = 
\rho(A \times A^c).
$$
Indeed, the function $(\chi_A(x) - \chi_A(y))^2$ takes value 1 if and only
if either $x \in A$ and $y\in A^c$ or $y \in A$ and $x\in A^c$. Then the
integral of $(\chi_A(x) - \chi_A(y))^2$ with respect to $\rho$ 
is $2 \rho(A \times A^c)$ because 
$\rho$ is symmetric. To finish the proof of \eqref{eq ||chi A||}, we calculate
$$
\rho(A \times A^c) = \int_A \rho_x(A^c) \; d\mu(x) \leq 
\int_A c(x) \; d\mu(x) = \nu(A).
$$ 
This proves \eqref{eq three inclusions}. 

To see that (2) holds, we recall that the energy norm of any constant 
function is zero. Hence, the relation $\chi_A + \chi_{A^c} = 0$ is true 
when the characteristic functions are considered as elements of $\h_E$. 
It follows that the vectors $\chi_A$ and $\chi_{A^c}$ have the same norm,
and their inner product in $\h_E$ is always non-positive.
The criterion for $\chi_A \in \h_E$ is a direct consequence of (1).
\end{proof}

\begin{remark}\label{lem_Dfin in H} (1) As proved above, the set of 
functions from $\Dfin(\nu)$ 
(and therefore its subset  $\Dfin(\mu)$) belongs to $\h_E$, and, for any
 set $A$ of  finite measure $\nu$, we have  
$$
\ba 
|| \chi_A ||^2_{\h_E} = & \rho(A \times A^c)\\
= & \int_A \rho_x(A^c)\; d\mu(x)\\
= & \int_A ( c(x) - \rho_x(A))\; d\mu(x).
\ea
$$

(2) $\chi_A= 0$ in $\h \ \Longleftrightarrow \ ||\chi_A||_{\h_E} = 0 \ 
 \Longleftrightarrow \ \rho_x(A) = c(x),\ 
\mu\mbox{-a.e.}\ x\in A \  \Longleftrightarrow \ \rho(A\times A) = 
\rho(A \times V)$; 

(3) It is useful to remember that $\Dfin(\mu) \subset \Dfin(\nu)$ 
and  $\Dfin(\mu) \subset L^2(\mu) \cap L^2(\nu) \cap \h_E$. But the set
$\Dfin(\mu)$ is not dense in $\h_E$, see details below. 

(4)  We observe that  statements  (1)  and  (2) of
 Proposition  \ref{prop ||chi A||}  give some criteria
 for a characteristic function $\chi_A$ (and a function from $\Dfin(\mu)$)
 to be in $\h_E$. In particular, $\chi_A$ is not in $\h_E$ if and only if
 both $A$ and $A^c$ have infinite measure. 
 
(5) Furthermore,  since $f = 0$ in $\h_E$ if
and only if $f$ is a constant a.e., and since two functions from $\h_E$ are 
identified if they differ by a constant, we conclude that the equality 
$\chi_A = - \chi_{A^c}$ holds when these functions are considered in 
$\h_E$. 

\end{remark}

The description of the  structure of the Hilbert space $\h_E$ is a very 
intriguing problem. We formulate  in the following statement several 
properties of vectors from $\h_E$ related mostly to characteristic functions.
Some of them have been proved in  \cite{BezuglyiJorgensen2018}. 
More statements will be added in Section \ref{sect energy}.

\begin{theorem}\label{thm_stucture of energy space} 
Let measures $\rho$ and  $\mu$ be as in Proposition \ref{prop ||chi A||},
and $c(x) \in \Lloc$. 

(1) If $\chi_A, \chi_B \in \h_E$, then 
\be\label{eq inner prod via rho}
\ba
\langle \chi_A, \chi_B\rangle_{\h_E} & = \rho((A \cap B )\times V) -
\rho (A \times B) \\
& = \nu(A \cap B ) - \rho (A \times B). \\
\ea
\ee

(2) The following conditions are equivalent to the orthogonality of 
$\chi_A$ and $\chi_B$:
(i)
$$
 \chi_A \ \perp \ \chi_B \ \Longleftrightarrow\ 
\left( \rho((A\setminus B)\times B) = \rho((A \cap B) \times B^c)\right);
$$
(ii)
$$
\chi_A \ \bot \  \chi_B \ \Longleftrightarrow \ \left(\int_{A\cap B} c(x)\;
 d\mu(x)  = \int_A \rho_x(B)\; d\mu(x) \right);
$$
 (iii) if $A \subset B$ and $\mu(A) > 0$, then   
$$
\chi_A \bot \chi_B \  \Longleftrightarrow  \ \rho_x(B^c) = 0 \ 
\mbox{for\ a.e.} \ x \in A.
$$
(iv) if $A \cap B = \emptyset$, then 
$$
\chi_A \bot \chi_B \  \Longleftrightarrow  \  \rho(A \times B) = 0;
$$
and more generally, 
$$ A \cap B = \emptyset  \ \Longrightarrow \
\langle \chi_A, \chi_B\rangle_{\h_E} \leq 0.
$$
\end{theorem}

\begin{proof} (1) 
The computation is based on the definition, given in (\ref{eq norm in H_E}),
and  the property of symmetry for $\rho$:
$$
\ba
\langle \chi_A, \chi_B\rangle_{\h_E} & = \frac{1}{2}\int_{V\times V} 
(\chi_A(x) - \chi_A(y)) (\chi_B(x) - \chi_B(y))\; d\rho(x, y)\\
& =  \int_{V\times V} (\chi_A(x)\chi_B(x)  -  \chi_A(x)\chi_B(y)) \;  
d\rho(x, y)\\
 & = \int_V \int_V \chi_{(A\cap B)\times V}(x, y) \; d\rho(x, y)
  - \int_V\int_V  \chi_{A \times B}(x, y) \; d\rho(x, y)\\
  &= \rho((A \cap B )\times V) - \rho (A \times B)\\
  & = \nu(A\cap B) - \rho(A\times B),
\ea
 $$
because 
$$
\ba
\rho((A \cap B )\times V)  & = \iint \chi_{(A \cap B )\times V)}(x, y) \;
d\rho(x, y)\\
& = \int_{A\cap B} c(x) \; d\mu(x)\\
& = \nu(A\cap B). 
\ea
$$ 

(2) For (i), we  see from \eqref{eq inner prod via rho} that the vectors 
$\chi_A$ and 
$\chi_B$ are orthogonal if and only if $\rho((A\cap B) \times V) = 
\rho(A \times C)$. Then $\langle \chi_A, \chi_B\rangle_{\h_E} = 0$ if 
and  only if
$$
\ba 
\rho((A \cap B )\times B^c) & = \rho((A \cap B )\times V)  - 
\rho((A \cap B )\times B)\\
& = \rho(A \times B) -  \rho((A \cap B )\times B)\\
& =    \rho((A\setminus B)\times B).
\ea
$$

For (ii), we use the first equality in \eqref{eq inner prod via rho} which 
is written in integrals. 

If $A\subset B$, then the condition $\chi_A \bot \chi_B$ is equivalent 
to the property $\rho(A \times B^c) =0$ what proves (iii).

The last equivalence in (2) is immediate from \eqref{eq inner prod via rho}.

\end{proof}

\subsection{Harmonic functions in $\h_E$}\label{subsect harmonic fncts}

Our goal is to study the properties of the Laplace operator $\Delta 
=\Delta(\rho)$ 
considered acting on functions from the finite energy space $\h_E$
according to formula \eqref{eq def of Delta}.
The next theorem is a key statement that has a number of important
consequences. Its proof uses the made assumption that functions
from $\h_E$ belong to $L^2_{\mathrm{loc}}(\mu)$. 

\begin{theorem}\label{thm inner prod via Delta}
Suppose that $c \in L^2_{\mathrm{loc}}(\mu)$. Let $\varphi \in 
\Dfin(\mu)$ and $f\in \h_E$. Then
\be\label{eq-inner prod via Delta}
\langle \varphi, f \rangle_{\h_E} = \int_V\varphi \Delta(f)\; d\mu.
\ee
\end{theorem}

\begin{proof}
Since $\Dfin(\mu)$ is spanned by characteristic functions, 
it suffices to prove (\ref{eq-inner prod via Delta}) for  $\varphi = 
\chi_A$ where $\mu(A) < \infty$. 

It follows from the condition of this theorem and Assumption 2 that, 
for any $f \in \h_E$,  the function $fc$ belongs to  $\Lloc$. Indeed, 
since functions  $c$ and $f$ are  in 
$L^2_{\mathrm{loc}}(\mu)$, we have, by the Schwarz inequality,
$$
\int_A fc \; d\mu < \infty, \quad A \in \Bfin(\mu).
$$

In the  computation given below, we use the following facts: the formula
for $R(f)$, see  (\ref{eq def of R}), the formula for $\Delta(f)$, see 
(\ref{eq Delta via R}), the definition of the inner product in $\h_E$, 
and the fact that the measure $\rho$ is symmetric, see 
(\ref{eq formula fo symm meas}) and (\ref{eq inner product}).

$$
\ba 
\langle \chi_A, f \rangle_{\h_E}  & = \frac{1}{2}
\iint_{\VtV} (\chi_A(x) - \chi_A(y)) (f(x) - f(y))\; d\rho(x, y)\\
& = \iint_{\VtV} (\chi_A(x) f(x) - \chi_A(x) f(y))\; d\rho(x,y)\\
& = \int_{V} \left(\int_V (\chi_A(x) f(x) - \chi_A(x) f(y))\; d\rho_x(y)
\right) d\mu(x)\\
& =  \int_V (\chi_A(x) f(x) c(x) - \chi_A(x) R(f)(x) )\; d\mu(x)\\
& = \int_V \chi_A(x) \Delta(f)(x) \; d\mu(x)\\
\ea
$$ 
\end{proof}

\begin{remark}
(1) To justify the correctness of this computation, we note that in the 
relation
$$
\langle \chi_A, f \rangle_{\h_E}  = \int_V (\chi_A(x) f(x) c(x) - 
\chi_A(x) R(f)(x) )\; d\mu(x),\quad A\in \Bfin(\mu),
$$
the integral $ \int_V \chi_A f c \;d\mu$ is  finite and therefore 
$\int_V  \chi_A R(f) \; d\mu$ is finite too. We can state even more, 
namely, the function $R(f) \in \Lloc$. 

 Given $f\in \h_E$, denote $f_+ = \max(f, 0), f_- = \min(f,0)$; then 
$f_{\pm} \in \h_E$.   Indeed, since $| f_{\pm}(x) - f_{\pm}(y)| \leq
| f(x) - f(y)|$, we see that $|| f_{\pm} ||_{\h_E} \leq || f||_{h_E}$. 

Therefore, we see that $R(f_\pm)$ is locally integrable because $R$ is 
a positive operator, and  the inequality 
$$
|R(f)| \leq R(|f|) = R(f_+) + R(f_-)
$$
implies that $|R(f)| $ is locally integrable.

(2) Another simple consequence of the fact that $cf \in \Lloc$ is that
$f \in L^1_{\mathrm{loc}}(\nu)$.
\end{remark}

\begin{corollary} (1) In conditions of Theorem 
\ref{thm inner prod via Delta}, the function $\Delta(f)$ is locally 
integrable for any $f \in \h_E$.

(2) 
$$
\left(\int_A \Delta(f) \; d\mu\right)^2 \leq \rho(A \times A^c)
|| f ||_{\h_E}^2, \quad A\in \Bfin(\mu).
$$ 

(3) For every $f \in \h_E$, the map 
$$
A \ \mapsto \ \mu_f(A) := \int_A \Delta(f)\; d\mu 
$$
determines a finite additive measure on $\VB$. The measure 
$\mu_f(\cdot)$ is $\sigma$-additive if and only if the function
$\Delta(f)$ is integrable on $\sms$.

\end{corollary}

\begin{proof}
To see that (1) holds, we write $\Delta(f) = cf - R(f)$ and use the 
proved facts that $cf$ and $R(f)$ are locally integrable. 

The second statement is  the Schwarz inequality where we uses the 
formula $|| \chi_A ||^2_{\h_E} = \rho(A \times A^c)$. 

(3) is obvious.
\end{proof}

We denote by $\h arm_E$ the set of harmonic functions in $\h_E$, i.e.,
a function $h \in \h_E$ is \textit{harmonic}  if $\Delta h =0$. 
Equivalently, $h$ is harmonic if $P(h) = h$.

\begin{theorem}\label{thm Royden}
The finite energy space $\h_E$ admits the decomposition 
into the orthogonal sum 
\be\label{eq Royden}
\h_E = \ol{\mathcal D_{\mathrm{fin}}(\mu)} \oplus \h arm_E
\ee
where the closure of $\Dfin(\mu)$ is taken in the norm of the Hilbert
 space $\h_E$.
\end{theorem}

\begin{proof} 
It follows from Theorem  \ref{thm inner prod via Delta} and 
\eqref{eq-inner prod via Delta} that 
if a function $f\in \h_E$ is orthogonal to every characteristic
function $\chi_A, A \in \Bfin(\mu)$, then
$$
\int_V \chi_A \Delta(f) \; d\mu = 0.
$$
Therefore, $\Delta(f)(x) =0$ for $\mu$-a.e. $x\in V$. This means
that $f$ is harmonic in $\h_E$. 

Conversely, the same theorem implies
that harmonic functions are orthogonal to $\Dfin(\mu)$ and therefore
to the closure of $\Dfin(\mu)$. 
\end{proof}

Formula \eqref{eq Royden} is an analogue of the so called \textit{Royden 
decomposition} used in the theory of weighted networks. 

\begin{remark} 
Theorem \ref{thm Royden} together with Theorem 
\ref{thm_stucture of energy space} show that the Hilbert space $\h_E$
 has no canonical orthonormal basis. Possible  candidates 
 $\{\chi_{A_i}, A_i \in \Bfin(\mu)\},$ where  the sets $\{A_i\}$ generate 
 $\B$, are not orthogonal and  their span is not even dense in $\h_E$. 

 Another property of the Hilbert space  $\h_E$ that makes it nonstandard 
 is the fact that the multiplication operator $M_\va : f \mapsto \va f$ is
 not symmetric in $\h_E$ when $\va $ is nonzero. This result follows 
 immediately from comparison the expressions for 
 $\langle \va f, g\rangle_{\h_E}$ and  $\langle f, \va g\rangle_{\h_E}$.
 
 As seen from Remark \ref{lem_Dfin in H} (5), pointwise identities 
 should not be confused with Hilbert space identities in $\h_E$. 
The point is that elements of $\h_E$ are  equivalence 
classes of functions which differ only by a constant. 
 When working with representatives, we typically abuse notation and use
 the same symbol $f$ to denote the equivalence class and the function. 
 
 \end{remark}

\begin{proposition}\label{prop orthogonal}
Let $f \in \h arm_E$ be a harmonic function for the Laplace operator
 $\Delta = \Delta(\rho)$ acting in $\h_E$ where $\rho$ is a symmetric
 measure. Suppose 
that $c \in  L^2_{\mathrm{loc}}(\mu) $ where $c(x) = \rho_x(V)$. 
Then, for any function $\varphi \in \Dfin(\mu)$, we have 
$$
\int_V \Delta(\varphi) f \; d\mu = 0.
$$
\end{proposition}

\begin{proof} By Assumption 2, we can assume that the harmonic
function $f$ is in $L^2_{\mathrm{loc}}(\mu) $. The condition of the 
proposition  means that 
the function $c(x) f(x)$ is locally integrable. Since $\Dfin(\mu)$
is spanned by characteristic functions, it suffices to prove the result for
$\varphi = \chi_A$. We have 
$$
\Delta(f)(x) =  0 \ \Longleftrightarrow \ c(x) f(x) = R(f)(x),
$$
where $R$ is the symmetric operator corresponding to $\rho$.
Therefore, for any $A \in \Bfin(\mu)$, one has
 $$
 \ba
\int_V \chi_A(x) c(x) f(x) \; d\mu(x) & = \int_V \chi_A(x) R(f)(x)\;
d\mu(x)\\
 & = \int_V R(\chi_A)(x) f(x)\; d\mu,
\ea 
$$
and this means that
$$
\int_V (\chi_A c f - R(\chi_A) f)\; d\mu = 0
$$
or equivalently 
$$
\int_V \Delta(\chi_A) f \; d\mu = 0.
$$
\end{proof}

\begin{remark}
(1) Clearly, the condition  $c \in  L^2_{\mathrm{loc}}(\mu) $ can be 
replaced  with $c \in  L^1_{\mathrm{loc}}(\nu)$. 

(2) Fix a harmonic function $f$ and consider the set of all functions
$g$ such that $\int_V g c f \; d\mu$ exists. Then we use the proof 
given above to conclude that 
$$
\int_V \Delta(g)  f \; d\mu  = 0.
$$
\end{remark}

\begin{corollary}\label{cor symm for integrals}
Let $f_1, f_2 $ be any elements of the finite energy space $\h_E$ which is 
defined by a symmetric measure $\rho$. Then there are functions $\varphi_1,
\va_2 \in \Dfin(\mu)$ and $h_1, h_2 \in \h arm_E$ such that 
\be\label{eq bdr term}
\langle f_1, f_2\rangle_{\h_E} = \int_V \va_1 \Delta(\va_2)\; d\mu +
\langle h_1, h_2\rangle_{\h_E}.
\ee
Moreover, if $f_1\in \Dfin(\mu)$, then 
\be\label{eq comp inner prod}
\langle f_1, f_2\rangle_{\h_E} = \iint_{\VtV} f_1(x) (f_2(x) - f_2(y))\; 
d\rho(x, y)
\ee
\end{corollary}

\begin{proof} By Theorem \ref{thm Royden}, every element $f \in \h_E$
is uniquely represented as $\va + h$ where $\va \in \ol{\mathcal 
D_{\mathrm{fin}}(\mu)}$ and $h\in \h arm$.  This property defines 
the functions $\va_i$ and $h_i$ for given $f_i$,  $i =1,2$. 

To show that \eqref{eq bdr term} holds, we use Theorem
\ref{thm inner prod via Delta} and Proposition
\ref{prop orthogonal}:
$$
\ba 
\langle f_1, f_2\rangle_{\h_E}  & = \langle \va_1, f_2\rangle_{\h_E} 
+ \langle h_1, f_2\rangle_{\h_E} \\
& = \int_V \va_1 \Delta(f_2)\; d\mu + \langle h_1, \va_2 + h_2
\rangle_{\h_E} \\
& = \int_V \va_1 \Delta(\va_2)\; d\mu +\langle h_1, h_2\rangle_{\h_E}.
\ea
$$

Relation \eqref{eq comp inner prod} follows directly from the proof 
of Theorem \ref{thm inner prod via Delta} and from  formula 
\eqref{eq formula fo symm meas} which characterizes symmetric measures.
We leave details to the reader.

\end{proof}

Obviously, relation \eqref{eq bdr term} can be written in the form
$$
\langle f_1, f_2\rangle_{\h_E} = \langle \va_1, \Delta(\va_2)
\rangle_{L^2(\mu)} + \langle h_1, h_2\rangle_{\h_E}.
$$ 

\begin{remark}\label{rem L2loc for nu}
(1) Suppose that the finite energy space consists of functions $f$ that
belong to $L^2_{\mathrm{loc}}(\mu) \cap  L^2_{\mathrm{loc}}(\nu)$. 
Then we claim that, for any $\varphi \in \Dfin(\nu)$ and  $f\in \h_E$,
$$
\langle \va, f \rangle_{\h_E} = \int_V \va \Delta(f) \; d\mu.
$$
The proof repeats that of Theorem \ref{thm inner prod via Delta}. The key 
point is that under the made assumption the function $fc $ is in $\Lloc$. 
Indeed, if $\va = \chi_A$ where $A \in \Bfin(\nu)$, then 
$$
\int_A fc \; d\mu  = \int_A f \chi_A\; d\nu
 \leq \sqrt{\nu(A)} \left( \int_A f^2 \; d\nu \right)^{1/2} < \infty.
$$
Having this result, we repeat word for word the computation of 
$\langle \va, f \rangle_{\h_E}$ from the proof of Theorem 
\ref{thm inner prod via Delta}. 

(2) Suppose that $\h_E \subset L^2_{\mathrm{loc}}(\mu) \cap 
 L^2_{\mathrm{loc}}(\nu)$ as in (1). Then we can prove the following 
 version of Proposition \ref{prop orthogonal}: 
 for any $\varphi \in \Dfin(\nu)$ and any harmonic function  $f$, we 
 have 
 $$
 \int_V \Delta(\va) f \; d\mu = 0.
 $$
To prove this result, we again use the fact that $fc $ is in $\Lloc$ and 
then follow the proof of Proposition \ref{prop orthogonal}. We note that
the assumption  $c \in L^2_{\mathrm{loc}}(\mu)$ is not used in this 
version. 

(3) Finally, we can deduce the result about the orthogonal decomposition
of functions from $\h_E$. Assuming that $\h_E \subset 
L^2_{\mathrm{loc}}(\mu) \cap  L^2_{\mathrm{loc}}(\nu)$, we can show
that 
$$
\ol{\Dfin(\nu)}^{\h_E}  = \ol{\Dfin(\mu)}^{\h_E}
$$ 
and therefore 
$$
\h_E = \ol{\mathcal D_{\mathrm{fin}}(\nu)} \oplus \h arm_E.
$$
We return to this property later in more general setting. 
\end{remark}

\subsection{Six applications}\label{subsect Applications}

In this subsection, we consider several types of symmetric measures. 
Each of these types corresponds to  a  possible direction for 
further applications of our approach based on symmetric measures.
We give a few  examples of symmetric measures $\rho$ and discuss
how they can be used to define the objects we are interested in.
\\

\textbf{(I) Gaussian processes.} This example is motivated by works 
on \textit{Gaussian processes} where
Cameron-Martin kernel $(x, y) \mapsto x\wedge y$ plays an important 
role. 

 Let $\sms$ be a standard measure space with a 
$\sigma$-finite measure $\mu$, and let $c : V \to \R_+$ be  a locally 
integrable Borel function. Define the measure $\nu$  on $(V, \B, \mu))$ by 
$d\nu(x) = c(x) d\mu(x)$. Set 
$$
\rho_\nu(A\times B) := \nu(A \cap B)
$$ 
where $A, B \in \B$. Then $\rho_\nu$ can be extended to a $\sigma$-finite
\textit{symmetric} measure  on $\vv$. We note that, for $A\in \B$, 
\be\label{eq_rho_nu}
\rho\circ \pi^{-1}(A) = \rho(A\times V) = \nu(A) = \int_A c(x) \; d\mu(x).
\ee
Hence $\rho\circ\pi^{-1} \ll \mu$ and, by Theorem  \ref{thm Simmons},
the measure admits a decomposition $\rho = \int_V \rho_x\; d\mu(x)$.
Since 
$$
\rho(A\times V) = \int_A \rho_x(V)\; d\mu(x),
$$
we deduce from \eqref{eq_rho_nu} that $\rho_x(V) = c(x)$. 
\\

\begin{lemma}
 The measure $\rho_x$ is atomic for $\mu$-a.e. $x \in V$, and 
$d\rho_x(y) = c(x) \delta_x(y)$. 
\end{lemma}

\begin{proof}

Indeed, we can write 
$$
\ba
\nu(A\cap B) & = \int_V \chi_A(x) \chi_B(x) c(x)\; d\mu(x)\\
& = \int_V \chi_A(x) \left(\int_V \chi_B(y) c(y)\delta_x(y)\right)\; 
d\mu(x)\\
& = \rho(A \times B)\\
& = \int_V \chi_A(x) \left(\int_V \chi_B(y) d\rho_x(y)\right)\; 
d\mu(x).\\
\ea
$$
Since $A$ and $B$ are arbitrary sets, this proves that $\rho_x$ is 
the atomic measure supported at $(x, x)$ with weight $c(x)$. 

\end{proof}

Based on the Claim, we can easily determine the operators $R, P,$ and 
$\Delta$ related to the measure $\rho_\nu$ as well as the finite
energy space $\h_E(\rho_\nu)$. It turns out
that the corresponding Markov process $(P_n)$ is deterministic since
$$
P(x, A) = \delta_x(A), \quad A\in \B,
$$
and therefore $P$ is the identity operator. It follows that $R(f)(x)  = c(x)
f(x)$, and  the Laplacian $\Delta = c(I - P) = 0$.  By a similar argument,  
$\h_E(\rho_\nu) = \{0\}$. 
\\

\textbf{(II) Countable Borel equivalence relations.} 
 For a measure space $\sms$, let $c_{xy}$ be a symmetric 
positive function defined on a 
symmetric set $E \subset  V\times V$. Consider a measurable field of
 finite Borel 
measures $x \mapsto \rho^c_x$ such that (i) the support of 
$\rho^c_x$ is the set $\{x\} \times E_x $, (ii) the measure 
$$
\rho :=\int_V c_{xy} \rho_x^c\;d\mu(x)
$$
 is symmetric. In particular, the set $E$ can be of positive product measure
 $\mu\times\mu$. This case is discussed in Section \ref{sect LO in H_E}.
Another important example is the case of a countable Borel equivalence
relation $E$. 

By definition, 
a symmetric Borel subset $E \subset \VtV$ is a \textit{countable Borel
 equivalence relation} if  it satisfies the following properties:

(i) $(x, y), (y, z) \in E \Longrightarrow\ (x,z) \in E$;

(ii) $E_x =\{y \in V : (x, y) \in E\}$ is countable for every $x$. 

Countable Borel equivalence relations 
have been extensively studied during last decades in the context of the 
descriptive set theory, measurable and Borel dynamics, see e.g.
 \cite{JacksonKechrisLouveau2002, KechrisMiller2004, Kanovei2008, 
 Gao2009, Kechris2010,
FeldmanMooreI_1977, FeldmanMooreII_1977} and references cited therein.

Let $| \cdot |$ be the counting measure on every $E_x$. Suppose that 
$c_{xy}$ is a symmetric function on $E$ such that, for every $x \in V$,
$$
c(x) = \sum_{y \in E_x} c_{xy} \in (0, \infty).
$$ 
Then we can define the atomic measure $\rho_x$ on $V$ by setting
$$
\rho_x(A) = \sum_{y \in E_x \cap A} c_{xy}. 
$$
Finally, define the  measure $\rho$ on $E$:
\be\label{eq rho for cber}
\rho = \int_V \delta_x \times \rho_x\; d\mu(x).
\ee
(We will identify measures $\rho_x$ and $\delta_x \times \rho_x$ 
as we did above.)

\begin{lemma}
The measure $\rho$ is a symmetric measure on $E$ which is
 singular with respect to $\mu \times\mu$.
\end{lemma}

\begin{proof} Since $(\mu \times \mu)(E) =0$, the singularity of $\rho$
is obvious. It follows from the symmetry of the function $c_{xy}$ and 
(\ref{eq rho for cber}) that, for $A, B \in \B$,
$$
\ba
\rho(A \times B) = & \int_A \sum_{y \in E_x \cap B} c_{xy}\; d\mu(x)\\
 = & \int_B \sum_{x \in E_y \cap A} c_{xy}\; d\mu(y)\\
 = & \rho(B \times A).
\ea
$$
\end{proof}

Having the measure $\rho$ defined, we apply the definitions given
in Subsection \ref{subsect symm meas operators} to construct the 
following operators:
$$
R(f)(x) = \int_V f(y) d\rho_x(y) = \sum_{y \in E_x} c_{xy}f(y),
$$
$$
P(f)(x) = \sum_{y \in E_x} \frac{c_{xy}}{c(x)} f(y) = \sum_{y \in E_x} 
p(x,y) f(y),
$$
and
$$
\Delta(f)(x) = c(x)f(x) - \sum_{y \in E_x} c_{xy}f(y).
$$
Functions $f$ from the finite energy Hilbert space$\h_E(\rho)$  are
 determined by the condition:
$$
\int_V  \sum_{y \in E_x} c_{xy}(f(x) - f(y))^2 \; d\mu(x) < \infty.
$$

\begin{definition}\label{def meas field ER}
Let $E$ be a countable Borel equivalence relation on a standard Borel space
$(V, \B)$. A symmetric subset $G \subset $E is called a \textit{graph} 
if $(x, x) \notin G, \forall x \in V$. A \textit{graphing} of $E$ is a graph $G$
such that the connected components of $G$ are exactly the 
$E$-equivalence classes. In other words, a graph $G$ generates $E$.
\end{definition}

The notion of a graphing is useful for the construction of the path space
$\Omega$ related to a Markov process, see Section \ref{sect diss space}. 

The following lemma can be easily proved.

\begin{lemma} Let $\rho$ be a countable equivalence relation on 
$(V, \B)$, and let $\rho$ be a  symmetric measure on $E$. Suppose $G$
is a graphing of $E$. Then $\rho(G) >0$.
\end{lemma} 

We can use the notion of graphing to construct the path space 
$\Omega$ and the family of probability measures $x \mapsto 
\mathbb P_x$ defined on the set of paths with starting point $x$.
This approach is realized in Section \ref{sect diss space} in more
general setting.

For more details regarding integral operators, and analysis of machine 
learning kernels, the reader may consult the following items
 \cite{Atkinson1975, CuckerZhou2007, ChenWheelerKochenderfer2017, 
Ho2017, JorgensenTian2015}  and the papers cited there. 
\\

\textbf{(III) Graphons.} 
Our approach in the study of symmetric measures, and the 
corresponding
graph Laplace operators, is close to the basic setting of the theory of 
\textit{graphons} and \textit{graphon operators}. We refer to several 
basic works in this theory  
 \cite{Avella-Medina_2017, Borgs_2008, Borgs_2012, Janson2013,
 Lovasz2012}. Informally speaking, a graphon is the limit 
of a converging sequence of finite graphs with increasing number of vertices.
Formally, a \textit{graphon} is a symmetric measurable function $W : 
(\mc X, m) \times (\mc X, m) \to [0, 1]$ where $(\mc X, m)$ is a 
probability measure
 space. The  linear operator $\mathbb W: L^2(\mc X, m) \to  
 L^2(\mc X, m) $ acting by the formula
 $$
 \mathbb W(f)(x) = \int_{\mc X} W(x, y)f(y)\; dm(y)
 $$
is called the\textit{ graphon operator}. The properties of $\mathbb W$ 
have been extensively studied in many recent works, see e.g.
\cite{Avella-Medina_2017}. 

Below in Section 3, we consider a similar operator $\wt R$ defined by a 
symmetric measure $\rho$. The principal difference is that we consider
infinite measure spaces and symmetric functions which are 
not bounded, in general. 
\\

\textbf{(IV) Determinantal point processes.} 
One more application of our results can be used in the theory
 of determinantal measures and \textit{determinantal point processes,}  
 see e.g.  \cite{Lyons2003, Hough_et_al_2009, BufetovQiu2015, 
BorodinOlshanski2017}. For example, the result
of \cite[Proposition 4.1]{Ghosh2015} gives the formula for the norm
in the energy space for a specifically chosen symmetric measure $\rho$. 
To make this statement more precise, we quote loosely the proposition 
proved in \cite{Ghosh2015}: 

Let $\Pi$ be a determinantal point process on a locally compact
space $(X, \mu)$ with positive definite determinantal kernel 
$K(\cdot , \cdot)$ such that 
$K$ is an integral operator on $L^2(\mu)$. Then, for
every compactly supported function $\psi$,
$$
Var\left[\int_x \psi \; d[\Pi] \right] = \iint_{X\times X}
|\psi(x) - \psi(y)|^2 |K(x, y)|^2 \; d\mu(x)d\mu(y). 
$$ 
This formula is exactly the formula for the norm in the energy space when 
the symmetric measure $\rho $ is defined by the symmetric function 
$K(x, y)$: $d\rho(x, y) = |K(x, y)|^2d\mu(x)d\mu(y)$, see Section
\ref{sect energy} below.

We refer to the following papers regarding the the theory of  positive 
definite kernels \cite{Aronszajn1950, Adams_et_al1994, PaulsenRaghupathi2016}. The reader will find more 
 references in the  papers cited there. Various applications of  positive 
 definite kernels are discussed in \cite{AlpayJorgensenLevanony2011,
  AlpayJorgensenVolok2014, AplayJorgensen2014, 
 AlpayJorgensenKimsey2015,  AlpayJorgensen2015,
  AlpayJorgensenLewkowicz2015,  AlpayJorgensenLevanony2017}.
  More details and explicit constructions of reproducing kernel Hilbert
  spaces are considered in Section \ref{sect RKHS}.
  \\

\textbf{(V) Dirichlet forms.} 
Another interesting application of symmetric measures and 
finite energy space is related to \textit{Dirichlet forms,} see e.g., 
\cite{AlbeverioFanHerzberg2011, MaRockner1992, MaRockner1995}. 
We mention here the 
Beurlng-Deny formula as given in \cite{MaRockner1992}. It states that a
symmetric Dirichlet form on $L^2(U)$, where $U$ is an open subset in 
$\R^d$,  can be uniquely expressed as follows:
$$
\ba
\mc E(u, v)  & = \sum_{i,i=1}^d \int \frac{\partial u}{\partial x_i}
\frac{\partial v}{\partial x_j}\; d\nu_{ij}\\
& \ \ \  \  +  \int_{(U \times U) \setminus \mathrm{diag}}
(u(x) - u(y)) (v(x ) - v(y)) \; J(dx, dy)\\
& \ \ \  \  +  \int uv\; dk.
\ea
$$
Here $u, v \in C^\infty_0(U)$,  $k$ is a positive Radon measure on $U
 \subset \R^d$, and $J$ is a symmetric measure on $(U \times U) 
 \setminus \mathrm{diag}$. The first term on the right hand side in this
  formula is called the diffusion term, the second, the jump term, and the last, 
  the killing term; a terminology deriving from their use in the study of general 
  Levy processes \cite{Applebaum2009}.  
  We see that the second term in this
 formula corresponds to the inner product in the finite energy space $\h_E$
 (details are in Section \ref{sect energy} below).
 \\

\textbf{(VI) Joinings}. 
The following application of symmetric measures is motivated by
the theory of \textit{joinings} developed in ergodic theory, see e.g.,
\cite{Glasner2003, delaRue2012} and the literature cited therein.

  Let $\sms$ be a standard $\sigma$-finite measure 
space, and let $S$ be a measure preserving surjective Borel endomorphism 
of $\sms$, i.e., $\mu\circ S^{-1} = \mu$.  Define a measure $\rho$ on 
$\vv$ by setting
\be\label{eq_rho joinings}
\rho(A \times B) = \mu(A \cap S^{-1}(B)).
\ee
We note that the measure $\rho$ is invariant with respect to 
$S^{-1} \times S^{-1}$:
$$
\rho(S^{-1}(A) \times S^{-1}(B)) = \mu(S^{-1}(A) \cap S^{-1}[S^{-1}(B)])
= \mu(A \cap S^{-1}(B)) = \rho(A \times B).
$$ 
Moreover, the measure $\rho$ defined in \eqref{eq_rho joinings} is 
 \textit{symmetric} if and only if
\be\label{eq_mu symm}
\mu(A \cap S^{-1}(B)) = \mu(S^{-1}(A) \cap B).
\ee

\begin{lemma}\label{lem_symm joinings}
 Let the measure $\rho$ be defined by
\eqref{eq_rho joinings} where $\mu$ satisfies \eqref{eq_mu symm}. Then:

(1) disintegration of $\rho$ with respect to $\mu$ defines the atomic 
fiber measures $\rho_x$ such that   $d\rho_x(y) = \delta_{Sx}(y)$  for all
 $x\in V$; 
 
(2) the symmetric operator $R = R(\rho)$ coincides with  the Koopman
 operator$ f \mapsto f\circ S$ corresponding to the endomorphism $S$.
\end{lemma}

\begin{proof}
(1) Indeed, setting 
$d\rho_x(y) = \delta_{Sx}(y)$, we obtain that
$$
\ba 
\int_A d\rho_x(B)\; d\mu(x) & = \int_A \left(\int_V \chi_B(y) \; 
d\rho_x(y) \right)\; d\mu(x) \\
& = \int_A \delta_{Sx}(B)\; d\mu(x)\\
& = \int_A \chi_{S^{-1}(B)}(x)\; d\mu(x)\\
& = \mu(A \cap S^{-1}(B)) \\
& = \rho(A\times B).
\ea
$$
This proves that $\rho = \int_V \delta_{Sx} d\mu(x)$.

(2) Since the field of measures $x \mapsto \rho_x$ is determined, we can 
directly compute the operator $R$:
$$
R(f)(x) = \int_V f(y) \; d\rho_x(y) = \delta_{Sx}(f)(y) = f(Sx).
$$
\end{proof}

\begin{remark}
 To avoid a possible confusion, we mention that the Koopman operator 
 $R : f \mapsto f\circ S$ corresponds to a symmetric measure if it satisfies 
  \eqref{eq symm in terms R}. Then 
$$
\int_V \chi_A (\chi_B\circ S) \; d\mu = \int_V  (\chi_A\circ S)
  \chi_B \; d\mu
$$
which is equivalent to the property \eqref{eq_mu symm}.
\end{remark}

It follows from Lemma \ref{lem_symm joinings} that the 
function  $c(x) =\rho_x(V)= 1$ and the Markov  operator $P$ coincides
with  $R$.

\begin{corollary} The Laplace operator $\Delta = \Delta(\rho)$, where 
the symmetric  measure $\rho$ is defined by \eqref{eq_rho joinings} and
\eqref{eq_mu symm}, is the coboundary operator, i.e.,
$$
\Delta(f)(x) = f(x) - f(Sx).
$$

\end{corollary}

\section{\textbf{Embedding of $\Dfin(\mu)$ and $\Dfin(\nu)$  into 
$\h_E$}}

\label{sect Emb L2 into H}
In this section, we focus on a description of subspaces in the finite
energy Hilbert space $\h_E$. We recall that this space is defined by a
symmetric measure $\rho$. It gives us the marginal measure $\mu$,
the function $c$, and therefore one more measure $\nu= c\mu$. 
We show below that the spaces of simple functions can be considered
as subspaces of $\h_E$ and describe their closures in $\h_E$.

\subsection{Locally integrable functions}
 We recall that the following chain of inclusions holds due to  
Assumptions 1 and 2, and the results  proved above:
$$
\Dfin(\mu) \subset \Dfin(\nu) \subset \h_E \subset 
L^2_{\mathrm{loc}}(\mu) \subset \Lloc.
$$
In this section we will describe the closure of subspaces spanned by
characteristic functions into the energy space $\h_E$.

It is useful to have a criterion for local integrability of functions with respect 
to the measures $\mu$ and $\nu$ because, by the made assumption, all 
functions from the energy space $\h_E$ should be locally integrable. 
 
\begin{lemma}\label{lem_loc int}
 Let $f \in \FVB$ and $\rho$ a symmetric measure on $\vv$ such that
 $c(x) = \rho_x(V)$. Then: 
 
(1) $f$  is locally integrable with respect to the measure 
$\mu$ on $\VB$ if and only if, for  any $A \in \Bfin(\mu)$, the function 
$$
F(A, x) =\int_A  \frac{f(y)}{c(y)} \; d\rho_x(y)
$$
is in $L^1(\mu)$; 

(2)  $f$   is locally integrable with respect to the measure 
$\nu$ on $\VB$ if and only if, for  any $A \in \Bfin(\nu)$, the function 
$$
F(A, x) = \frac{1}{c(x)}\int_A f(y) \; d\rho_x(y)
$$
is in $L^1(\nu)$. 

\end{lemma} 

\begin{proof} We prove (1) only because the other statement is proved 
similarly.
Let $A \in \Bfin(\mu)$;  then we use $P$-invariance of $\nu$ and obtain

$$
\ba
\int_A f(x) \; d\mu(x) = & \int_V \frac{f(x)}{c(x)}\; d\nu(x)\\
= & \int_V \frac{\chi_A(x)f(x)}{c(x)}\; d(\nu P)(x)\\
= & \int_V P\left(\frac{\chi_Af}{c}\right)(x) \; d\nu (x)\\
= & \int_V R\left(\frac{\chi_Af}{c}\right)(x) \; d\mu (x)\\
= & \int_V \left( \int_V \frac{\chi_A(y)f(y)}{c(y)} \; d\rho_x(y)\right)
 d\mu (x)\\
 =& \int_V F(A, x) \; d\mu(x).\\
\ea
$$
If $f = f_+ - f_-$, where $f_+$ and $f_-$ are positive an negative parts of
$f$, then the corresponding function $F(A, x)$ is represented as 
 $F_+ (A, x) - F_-(A, x)$. 
Hence the proved equality $\int_A f(x) \; d\mu(x)  =  \int_V F(A, x) \; 
d\mu(x)$ points out that $f \in \Lloc$ if and only $F(A, \cdot)$ is 
$\mu$-integrable for every $A\in \Bfin(\mu)$.
\end{proof}

We observe that condition (2) of Lemma \ref{lem_loc int} can be written in
the following equivalent form:
$$
f \in  L^1_{\mathrm{loc}}(\nu) \ \Longleftrightarrow \ 
[x \mapsto c(x) P(\chi_A f)(x)] \in L^1(\mu)
$$
for any $A \in \Bfin(\nu)$.

\subsection{Embedding of $\Dfin(\nu)$ into $\h_E$}
Let $f $ be a function from $L^2(\nu)$. We will show that this function 
determines an element of the finite energy space. In other words, the 
equivalence class generated by $f$ belongs to $\h_E$. In order to 
 distinguish a function  $f\in L^2(\nu)$ and the corresponding element
of $\h_E$, we denote the latter by $\iota(f)$.

The following formula will be  repeatedly used in further computations of 
the energy norm. This result extends Theorem 
\ref{thm inner prod via Delta}.

\begin{lemma}\label{lem formula for energy norm}
For $f\in \h_E \cap L^2(\nu)$, the following formula holds:
\be\label{eq formula for energy norm}
|| f ||^2_{\h_E} = \iint_{V \times V} f(x)(f(x) - f(y)) \; d\rho(x, y) =
\int_V f(x) \Delta(f)(x) \; d\mu(x).
\ee
\end{lemma}

\begin{proof} We have
$$
\ba 
|| f ||^2_{\h_E} = & \frac{1}{2}\iint_{V\times V} (f(x) - f(y))^2\;
 d\rho(x,y)\\
 = & \frac{1}{2}\iint_{V\times V} (f(x)^2 - 2f(x)f(y) + f(y)^2)\;
 d\rho(x,y)\\
 = & \frac{1}{2}\iint_{V\times V} (2f(x)^2 - 2f(x)f(y))\;  d\rho(x,y)\\
 =& \iint_{V \times V} f(x)(f(x) - f(y)) \; d\rho(x, y)\\
 = & \int_V f(x) \left(\int_V (f(x) - f(y)) \; d\rho_x(y)\right)\; d\mu(x)\\
 =& \int_V f(x) \Delta(f)(x) \; d\mu(x).
\ea
$$
In the above computation, we  used (\ref{eq formula fo symm meas}).  
\end{proof}

\begin{proposition}\label{prop L2nu in H} Let $\rho$ be a symmetric
measure on $\vv$ with $c(x) = \rho_x(V)$. If $\mu$ is the projection of
$\rho$ onto the margin $V$ and $d\nu = c d\mu$, then the map 
$$
L^2(\nu) \stackrel{\iota} \longrightarrow\h_E : \iota (f) = f
$$
is a well defined bounded linear operator such that 
$$
|| \iota ||_{L^2(\nu) \to \h_E} \leq \sqrt{2}.
$$
Moreover the adjoint operator $\iota^* : \h_E  \to L^2(\nu)$ acts by 
the formula:
$$
\iota^*(g) = (I - P)(g), \qquad g \in \h_E.
$$
\end{proposition} 

\begin{proof}
We first show that $\iota(f)$ is in $\h_E$. Indeed, the symmetry of $\rho$
implies that
$$
\ba
|| \iota(f) ||_{\h_E}^2 & = \frac{1}{2}\iint_{\VtV} 
(f(x) - f(y))^2\; d\rho(x, y)\\
& \leq  \iint_{\VtV} (f(x)^2 + f(y)^2)\; d\rho(x, y)\\
& = 2\iint_{\VtV} f(x)^2\; d\rho(x,y)\\
& = 2\int_{V} f(x)^2 c(x)\; d\mu(x)\\
& = 2\int_{V} f(x)^2 \; d\nu(x).\\
\ea
$$
It follows from the last relation that the norm of the bounded linear
operator $\iota$ is bounded by $\sqrt{2}$ what proves the statement.  

To prove the formula for the adjoint operator $\iota^*$ we use Lemma 
\ref{lem formula for energy norm}. Then 
$$
\ba 
\langle \iota(f), g\rangle_{\h_E} & = \frac{1}{2}
\iint_{\VtV} (f(x) - f(y))(g(x) - g(y))\; d\rho(x, y)\\
& = \iint_{\VtV} f(x) (g(x) - g(y)) P(x, dy)d\nu(x)\\
& = \int_{V} f(x) (I  - P)(g)(x)\; d\nu(x)\\
& = \langle f, \iota^*(g)\rangle_{L^2(\nu)},
\ea
$$ 
and we obtain the formula for $\iota^*$.
\end{proof}

\begin{corollary}\label{cor harmonic}
For $\rho, \mu, \nu$, and $c(x)$ as above in Proposition 
\ref{prop L2nu in H}, we have:

(i) 
$$
\ol{\Dfin(\nu)}^{\h_E} =\ol{L^2(\nu)}^{\h_E} = L^2(\nu),
$$
where $L^2(\nu)$ is considered as a subspace of $\h_E$,

(ii)
$$
 L^2(\nu) \cap \h arm_E = \{0\}.
$$
\end{corollary}

\begin{proof}
(i) Indeed, the embedding $\iota : L^2(\nu)  \rightarrow \h_E$ is continuous
by Proposition \ref{prop L2nu in H}, so that $\iota(\Dfin(\nu))$ is dense 
in $L^2(\nu)$ and the image of $L^2(\nu)$ in $\h_E$ is closed. 

(ii) It was proved in \cite{BezuglyiJorgensen2018}, see 
Theorem \ref{thm harmonic}, that
$L^2(\nu)$ does not contain harmonic functions for a $\sigma$-finite 
measure $\nu$. 
\end{proof}

We deduce from the results proved in this subsection that 
the following fact holds.

\begin{corollary}\label{cor closures}
$$
\ol{\Dfin(\nu)}^{\h_E}  = \ol{\Dfin(\mu)}^{\h_E}.
$$
\end{corollary}

It is curious to compare this results with the proper inclusion
$$
\ol{\Dfin(\mu)}^{L^2(\nu)} \subset \ol{\Dfin(\nu)}^{L^2(\nu)}
 = L^2(\nu). 
$$

\subsection{Embedding of $\Dfin(\mu)$ into $\h_E$}

It was proved in Lemma \ref{lem_Dfin in H} that $\Dfin(\mu)$  can be 
viewed as a subset of $\h_E$. We define the map $J$ setting
$$
\Dfin(\mu) \ni \varphi \ \stackrel{J} \longrightarrow \ \varphi \in \h_E.
$$
Here $\varphi$ is a linear combination of characteristic functions,
  so that the operator $J$ can be studied on  
$\chi_A$, $A \in \Bfin(\mu)$. 

In the proof of Theorem \ref{thm J closable} below, we need the density 
of the set 
$$
D^* := \{ f \in \h_E : \Delta(f) \in L^2(\mu)\}.
$$
in the finite energy space $\h_E$.

\begin{lemma}\label{lem D* dense} Suppose that $c \in \lc2$.
Then the set $D^*$ is a dense subset in $\h_E$.
\end{lemma}

\begin{proof} Fix $A, B \in \Bfin(\mu)$. Let $\omega_{A,B}$ be an  
element of $\h_E$ such that 
$$
\Delta (\omega_{A, B})(x)  = c(x) (\chi_A(x) - \chi_B(x)).
$$  
We call $\omega_{A, B}$ a $\nu$-dipole. It is proved, see Section 
\ref{subsect dipoles} below, that 
$\{\omega_{A, B} : A, B \in \Bfin(\mu)\}$ is a dense subset in $\h_E$.
Hence, in order to prove the lemma,  it suffices to show that 
$$
\{\omega_{A, B} : A, B \in \Bfin(\mu)\} \subset D^*.
$$
Indeed, we compute 
$$
\ba
\int_V (\Delta(f))^2 \; d\mu & = \int_V c^2 (\chi_A - \chi_B)\; d\mu \\
&= \int_A c^2\; d\mu - 2\int_{A\cap B} c^2\; d\mu + \int_B c^2\; d\mu,
\ea
$$
and the result follows.
\end{proof}

\begin{theorem}\label{thm J closable}
The operator $J : L^2(\mu) \to \h_E$ is closable, densely defined, and, 
in general, unbounded. The operator $JJ^*$ is a self-adjiont extension of
 the  symmetric Laplace operator $\Delta$ acting in $\h_E$. 
\end{theorem}

\begin{proof}
We first note that $J$ is a densely defined operator because
$\Dfin(\mu)$ is dense in $L^2(\mu)$.  Let $\varphi \in \Dfin(\mu)$ and
$f \in \h_E$. We proved in Theorem \ref{thm inner prod via Delta} that
\be\label{eq innr prd via Delta}
\langle J(\varphi), f \rangle_{\h_E} = \int_V \varphi \Delta(f)\; d\mu.
\ee
Suppose now that, in relation (\ref{eq innr prd via Delta}), the function
  $f$ belongs to $D^*$. Then (\ref{eq innr prd via Delta}) can be written 
  as
\be\label{eq- J and Delta}
  \langle J(\varphi), f \rangle_{\h_E} = \langle \varphi, \Delta(f) 
  \rangle_{L^2(\mu)}.
 \ee
To define the adjoint operator $J^*$, we say that $f \in Dom(J^*)$ if
there exists a finite constant $C_f$ such that 
\be\label{eq-C_f}
\left(\langle J(\varphi), f \rangle_{\h_E}\right)^2 \leq C_f \int_V 
\varphi^2\; d\mu,\qquad \varphi \in \Dfin(\mu).
\ee
Then, by the Riesz representation theorem, there exists   a unique element
$f^* \in L^2(\mu)$ such that 
$$
\langle J(\varphi), f \rangle_{\h_E} = \langle \varphi, 
f^*\rangle_{L^2(\mu)}
$$
We note that, by the Schwarz inequality applied to the right hand side 
of (\ref{eq innr prd via Delta}), relation (\ref{eq-C_f}) holds for any
$f\in D^*$.  This means that the domain of $J^*$ contains $D^*$, and
 we can set $J^*(f) = f^*$. It follows from (\ref{eq innr prd via Delta}) 
 and (\ref{eq- J and Delta}) that $\Delta(f) = J^*(f), f \in D^*$ where
 $\Delta$ is  considered as an operator from $\h_E$ to $L^2(\mu)$. 

Let $H$ be a Hilbert space. It is well known that  a linear operator $T :
 Dom(T) \to H$ is closable if and only if the operator $T^*$ is densely
 defined. As was shown in Lemma \ref{lem D* dense}, $D^*$ is dense
 in $\h_E$, hence $J^*$ is densely defined, and $J$ is closable.
 Therefore $J$ admits a closed  extension that we denote again by
$J$. The operator $JJ^* : \h_E \to \h_R$ is obviously self-adjoint, and as
was proved above $JJ^*$ is a self-adjoint extension of the operator
 $\Delta$ viewed as operator acting in $\h_E$, i.e., 
 $$
 \Delta(f) = JJ^*(f), \qquad f \in D^*. 
 $$
  
  It remains to observe that if $J$ were a bounded operator, then 
  $|| J || = || J^* ||$ and the operator $\Delta$ would be bounded too. It 
  happens only in the case when the function $c $ is essentially bounded. 
  Our assumption about $c$ does not require its boundness. 
\end{proof}

The proved theorem can be used to show that the space $\h arm_E$  of 
harmonic functions from $\h_E$ is nonempty, see Corollary 
\ref{cor harmonic}. We use notations 
introduced in Lemma \ref{lem D* dense} and Theorem \ref{thm J closable}.

\begin{corollary}\label{cor harmonic}
The following equalities hold:
$$
\h_E \ominus J(\Dfin(\mu)) = \{f \in Dom(J^*) : J^*(f) = 0\} = 
\h arm_E.
$$
\end{corollary}

\section{\textbf{Dissipation space}}\label{sect diss space}

\subsection{Path space and Markov chain}\label{subsect path space} 
Let $\sms$ be a $\sigma$-finite standard measure space. Suppose that
a transition probability kernel $x \mapsto P(x, A)$ is defined on 
$\VB$, and let $P$ be the Markov operator defined by the formula
$$
P(f)(x) = \int_V f(y) \; P(x, dy).
$$
In particular, the kernel $x \mapsto P(x, A)$ and the operator $P$ can be
generated by a symmetric measure $\rho = \int_V \rho_x d\mu(x)$ 
on $\vv$ where  
$$
P(x, dy) = \frac{1}{c(x)}d\rho_x(y), \ \ c(x) = \rho_x(V).
$$
Our main interest will be focused on relations between the measure $\rho$ and
properties of $P$. 

Having the kernel $P(x, A)$, we can define inductively the sequence of
probability distributions $(P_n)$ by setting $P_0(x, A) = \chi_A(x), 
P_1(x, A) = P(x, A)$ and
$$
P_{n+1}(x, A) = \int_V P_n(y, A)\; P(x, dy), \ \ n \in \N_0.
$$
It can be shown that
$$
P_n(x, A) = P^n(\chi_A)(x), \ \ n \in \N_0.
$$

We can define simultaneously the sequence of symmetric measures
$(\rho_n)$ by the formula
$$
\rho_n(A \times B) = \int_V \chi_A(x) P_n(x, B)\; d\nu(x),\ \ n\in \N,
$$
see also \eqref{eq_def rho_n}.
\medskip

For every Borel set $A$, consider the series
$$
G(x, A) := \sum_{n\in \N_0} P_n(x, A) = \sum_{n\in \N_0}  P^n(\chi_A)(x).
$$
The Markov process $(P_n)$ is called \textit{transient} if, for every 
set $A$, $G(x, A)$ is finite for a.e. on the space $\sms$. Then $G(x, A)$
is called the\textit{ Green's function}.  We will discuss various properties
of the Green's function in Section \ref{sect RKHS}.
\medskip

We denote by $\Omega$ the infinite Cartesian product  $V\times V \times
 \cdots = V^{\N_0}$.  Let $(X_n (\omega): n = 0,1,...)$ be the 
 sequence of random variables  $X_n : \Omega \to V$ such that 
  $X_n(\omega) =  \omega_n$. We call $\Omega$ as 
the \textit{path space} of the Markov process $(P_n)$. 

 Let  $\Omega_x, x \in V,$ be the set of infinite paths beginning at $x$:
$$
\Omega_x := \{\omega\in \Omega : X_0(\omega) = x\}.
$$
Clearly, $\Omega = \coprod_{x\in V}  \Omega_x$. 

The subset $C(A_0, ... ,A_k) :=  \{\omega \in \Omega : X_0(\omega) \in 
A_0, ... X_k(\omega)\in A_k\}$ of $\Omega$ is called a 
\textit{cylinder set} defined by Borel sets $A_0, A_1, ..., A_k$ taken from 
$\B$, $k \in \N_0$.  
The collection of cylinder sets generates the $\sigma$-algebra 
$\mathcal C$ of Borel subsets of $\Omega$, and $(\Omega, \mc C)$
 is a standard Borel space. 
By definition of $\mc C$, he functions $X_n : \Omega \to V$ are Borel. 

Denote by  $(\mathcal{F}_{\leq n})$ the increasing  sequence of 
$\sigma$-subalgebras  where $\mathcal{F}_{\leq n}$ is the smallest
 subalgebra for which the functions $X_0, X_1, ... , X_n$ are Borel. By
 $\mathcal F_n$, we denote the $\sigma$-subalgebra $X_n^{-1}(\B)$.

 Define a probability measure $\mathbb P_x$ on $\Omega_x$. For 
 a cylinder set $C(A_1, ... , A_n)$ from $\mathcal F_{\leq n}$,  we set
\be\label{eq meas P_x} 
 \mathbb P_x(X_1 \in A_1, ... , X_n \in A_n) 
= \int_{A_{1}}\cdots \int_{A_{n-1}} P(y_{n-1}, A_n) P(y_{n-2},  dy_{n-1})
\cdots  P(x, dy_1).
 \ee
Then $\mathbb P_x$ extends to the $\sigma$-algebra $\mc C$ of
 Borel sets on $\Omega_x$ by  the
 Kolmogorov extension theorem \cite{Kolmogorov1950}.

The values of $\mathbb P_x$ can be written as 
\be\label{eq meas P_x 2}
 \mathbb P_x(X_1 \in A_1, ... , X_n \in A_n) = 
 P(\chi_{A_1} P(\chi_{A_2}P(\ \cdots\ P(\chi_{A_{n-1}} P(\chi_{A_n})) 
 \cdots )))(x).
\ee
The joint distribution of the random variables $X_i$ is given by 
\be\label{eqjoint distr}
d\mathbb P_x(X_1, ... , X_n)^{-1} = P(x, dy_1) P(y_1, dy_2) \cdots 
P(y_{n-1}, dy_n).
\ee

\begin{lemma}\label{lem st meas space}
The measure space $(\Omega_x, \mathbb P_x)$ is a standard 
probability measure space for every  $x\in V$. 
\end{lemma}

\subsection{Orthogonal functions in the dissipation space} 

\begin{definition}\label{def diss space}
On the measurable space $(\Omega, \mathcal C)$, define 
a $\sigma$-finite measure $\lambda$ by 
\be \label{eq def lambda}
\lambda := \int_V \mathbb P_x \; d\nu(x)
\ee
The Hilbert space 
$$
Diss := \left\{\frac{1}{\sqrt 2} f : f \in L^2 (\Omega, \mc C, \Lambda) 
\right\}
$$
is called the \textit{dissipation space}.
\end{definition}

Remark that: (i) the measure $\la$ is infinite if and only if the measure 
$\nu$ is  infinite, and (ii) the dissipation  Hilbert space $Diss $ is 
formed, in fact,  by functions  from $L^2(\Omega, \la)$ which are 
rescaled by the factor $1/\sqrt 2$, i.e., 
$$  
  \| f\|_{\mc D} = \frac{1}{\sqrt 2} \| f\|_{L^2(\la)}, \ \  f \in Diss.
$$

Because the partition of  $\Omega$ into $(\Omega_x : x \in V)$ is
 measurable,  we have the decomposition into the direct integral of
 Hilbert spaces: 
\be\label{eq L^2(lambda) dir int}
L^2(\Omega, \la)  = \int^{\oplus}_V L^2(\Omega_x, \mathbb P_x)\;  
d\nu(x)
\ee

As a consequence of \eqref{eq L^2(lambda) dir int}, we obtain the 
following formula: for a measurable function over $(\Omega, \mc C)$,
 $$
\lambda(f) = \int_{\Omega} f(\omega) \; d\lambda(\omega) = 
\int_V \mathbb E_x(f)\; d\nu(x)
$$ 
where $\mathbb E_x$ denotes the conditional expectation with respect to 
the measures $\mathbb P_x$,
$$
\mathbb E_x(f) = \int_{\Omega_x} f(\omega)\; d\mathbb P_x(\omega). 
$$

Then the inner product in the Hilbert space $Diss$ is determined by the 
formula:
\be\label{eq inner prod in D}
\langle f, g\rangle_{\mc D} =  \frac{1}{2}\int_V \mathbb E_x(fg) \; 
d\nu(x)
= \frac{1}{2} \int_V \int_{\Omega_x}  f(\omega) g(\omega) \; d\mathbb 
P_x(\omega)d\nu(x).
\ee

Since $X_n^{-1}(\B)$ is a $\sigma$-subalgebra of $\mc C$, there exists 
 a  projection 
 $$
 E_n : L^2(\Omega, \mc C, \la) \to L^2(\Omega, X_n^{-1}(\B), \la).
 $$ 
The projection  $E_n$ is called the \textit{conditional expectation} 
with respect to $X_n^{-1}(\B)$ and satisfies the property:
\be\label{eq cond exp E_n}
E_n(f\circ X_n) = f\circ X_n. 
\ee

The operator $P$ and conditional expectations  $\mathbb E_x$ are 
related as follows: for any Borel  functions $f, h$, one has
$$
\mathbb E_x[(h\circ X_n)\ (f\circ X_{n+1})] = \mathbb E_x[(h\circ X_n)\ 
(P(f)\circ X_{n})].
$$
In particular, 
\be\label{eq_P(f) via cond exp}
P(f)\circ X_n = \mathbb E[f\circ X_{n+1} \ | \ \mc F_n ]= E_n(f\circ 
X_{n+1}).
\ee

\begin{lemma} The Hilbert space $L^2(\nu)$ is isometrically embedded 
into the dissipation space $Diss$ by the map
\be\label{eq L2nu in Diss}
W_n (f) = \sqrt{2} (f\circ X_n), \ \ \ n \in \N,
\ee
Moreover, the conditional expectation $E_n$ satisfies the property:
$$
E_n = W_nW_n^*.
$$
\end{lemma}

The lemma follows immediately from the definition of 
 $W_n$ and \eqref{eq cond exp E_n}. 
 \medskip
 
 We return here to the notion of reversible Markov process in connection
 with the dissipation space. 
We proved in \cite{BezuglyiJorgensen2018} that the Markov process
$P_n$ is irreducible if  and only if the initial symmetric measure is 
irreducible. 

\begin{theorem}[\cite{BezuglyiJorgensen2018, 
BezuglyiJorgensen_2018}] \label{prop from A to B}   
Let $\rho$ be a symmetric measure on $\vv$, and let $A$ and $B$ be 
any two sets from $\Bfin(\mu)$. Then:

(1)  
 \be\label{eq rho_n vs lambda}
 \rho_n(A \times B)  = \langle \chi_A, P^n(\chi_B)\rangle_{L^2(\nu)} 
 = \lambda (X_0 \in A,  X_n  \in B),\ \  n \in \N.
 \ee
 
(2) The Markov process $(P_n)$ is irreducible if and only if the 
measure  $\rho$ is irreducible. 

(3) Let the measure $\lambda$ on $\Omega$ be defined by 
(\ref{eq def lambda}). The Markov operator $P$ is reversible if and only if 
$$
\lambda(X_0\in A_0 \ |\ X_1\in A_1) =
\lambda(X_0\in A_1 \ |\ X_1\in A_0). 
$$
\end{theorem}

This theorem can be interpreted as follows.
Relation (\ref{eq rho_n vs lambda}) says that, 
for the Markov process $(P_n)$, the ``probability''   to get in $B$ for 
$n$ steps starting somewhere in $A$ is exactly  $\rho_n(A \times B) > 0$.
And the concept of reversible Markov processes can
be reformulated in terms of the measure $\lambda$: roughly speaking
$\lambda$ must be a symmetric distribution.

\begin{corollary} Let $A_0, A_1, ... , A_n$ be a finite sequence of subsets
 from $\Bfin$. Then 
 $$
\mathbb P_x(X_1 \in A_1, ... , X_n \in A_n)\ |\ x \in A_0) > 0 \ 
\Longleftrightarrow \ \rho(A_{i-1} \times A_i) > 0
$$ 
for $i=1, ... ,n$.
\end{corollary}

In what follows we formulate and prove a key property of functions from
the dissipation space. We will use the fact that $L^2(\nu)$ can be seen as
a subspace of $Diss$, see \eqref{eq L2nu in Diss}

\begin{theorem}[Orthogonal decomposition] \label{thm orthog in Diss}
(1) Let $g_1, g_2$ be functions from $L^2(\nu)$. Then 
\be\label{eq key orth}
\langle  g_1 \circ X_n, \ P(g_2)\circ X_n - g_2\circ X_{n+1} 
\rangle_{\mc D}
= 0.
\ee

(2) For any function $f\in L^2(\nu)$ and any $n \in \N$, the functions 
\be\label{eq orth}
(I - P)(f)\circ X_n \ \perp \ (P(f)\circ X_n - f\circ X_{n+1})
\ee
in the dissipation space $Diss$.
\end{theorem}

\begin{proof} (1)  It follows from (\ref{eq inner prod in D}) that the result 
would follow  if we proved that  the functions $ g_1 \circ X_n$ and 
$ P(g_2)\circ X_n - g_2\circ X_{n+1} $ are orthogonal in 
$L^2(\Omega_x,\mathbb P_x)$ for a.e. $x$. We use here 
\eqref{eq_P(f) via cond exp}
and  (\ref{eq cond exp E_n}) to  compute the inner product:
$$
\ba 
& \langle g_1 \circ X_n,  P(g_2)\circ X_n - g_2\circ X_{n+1}
\rangle_{\mathbb P_x}   \\
= &\mathbb E_x(E_n(g_1 \circ X_n)\ 
(P(g_2)\circ X_n - g_2\circ X_{n+1}))\\
=& \mathbb E_x((g_1 \circ X_n)\  
E_n(P(g_2)\circ X_n - g_2\circ X_{n+1}))\\
=&\mathbb E_x((g_1 \circ X_n)\  
(P(g_2)\circ X_n - E_n(g_2\circ X_{n+1})))\\
=&\mathbb E_x((g_1 \circ X_n)\  
(P(g_2)\circ X_n - P(g_2)\circ X_{n}))\\
= & 0
\ea
$$

(2) To prove (\ref{eq orth}), it suffices to show that 
\be\label{eq orth 1}
(f\circ X_n) \ \perp \ (P(f)\circ X_n - f\circ X_{n+1})
\ee
and 
\be\label{eq orth 2}
P(f)\circ X_n \ \perp \ (P(f)\circ X_n - f\circ X_{n+1})
\ee
Relation (\ref{eq orth 1}) has been proved in (1). It follows from 
\eqref{eq_P(f) via cond exp} 
and the proof of statement (1) that, for $\nu$-a.e. $x\in V$, 

$$
\ba
&\ \mathbb E_x ((P(f)\circ X_n)\ (P(f)\circ X_n - f\circ X_{n+1}))\\ 
= &\ \mathbb E_x ((P(f)\circ X_n)\ E_n(P(f)\circ X_n - f\circ X_{n+1}))\\
= &\ \mathbb E_x ((P(f)\circ X_n)\ (P(f)\circ X_n - E_n(f\circ X_{n+1})))\\
=&\  0. 
\ea
$$
This proves  (\ref{eq orth 2}) and we are done.

\end{proof}

\subsection{Embedding of $\h_E$ into $Diss$}
The importance of Theorem \ref{thm orthog in Diss} is explained by the 
fact that the finite energy space can be isometrically embedded into the
dissipation space. 

We remind first that, for a Markov operator $P$ and the transition 
probabilities   $x \mapsto P(x, \cdot)$,  we defined the path spaces
$(\Omega, \lambda)$ and $(\Omega_x, \mathbb P_x)$, $x \in V$, 
together with  the sequence of random variables $X_n(\omega)$ taking 
values in $(V, \B)$, see Subsection \ref{subsect path space}.
Then we  have the following formulas for the
conditional expectation $\mathbb E_x$ with respect to the probability
measure $\mathbb P_x$:
\be\label{eq E(X_0)}
\mathbb E_x(f \circ X_0) = \int_{\Omega_x} f(X_0(\omega))\; 
d\mathbb P_x(\omega) = \int_{\Omega_x} f(x) \; 
d\mathbb P_x(\omega) = f(x),
\ee
\be\label{eq E(X_1)}
\mathbb E_x(f \circ X_1) = \int_{\Omega_x} f(X_1(\omega))\; 
d\mathbb P_x(\omega) = \int_V f(y) \; P(x, dy) = P(f)(x)
\ee
where $y = X_1(\omega)$.

We define the operator $\partial : \h_E \to Diss$ by setting  
\be\label{eq operator d}
\partial  : f \mapsto f\circ X_1 - f\circ X_0.
\ee
Similarly, we can set
\be\label{eq operator d_n}
\partial_n  : f \mapsto f\circ X_{n+1} - f\circ X_{n}.
\ee

\begin{proposition}\label{prop partial embeds}
(1) The operator $\partial :\h_E \to Diss $ defined in 
(\ref{eq operator d}) is an isometry.

(2) For $f \in \h_E$ and $\nu = c\mu$, we have 
$$
|| f ||^2_{\h_E} = \frac{1}{2} \int_V \mathbb E_x[(f\circ X_1 - 
f\circ X_0)^2]\; d\nu(x).
$$
\end{proposition}

In the next statements we strengthen the result of Proposition 
\ref{prop partial embeds} (2) using the orthogonal decomposition 
given in Theorem \ref{thm orthog in Diss}. In Theorem 
\ref{thm main norm f}, we give an 
explicit, canonical and orthogonal splitting of the energy form 
from Proposition \ref{prop partial embeds} (2)  into two terms, 
each one having a stochastic content, variation and dissipation.

\begin{theorem}\label{thm main norm f}
Let $f \in \h_E$. Then 
\be\label{eq norm of f}
\ba
\| f \|^2_{\h_E} =&  \frac{1}{2}\left( \int_V (P(f^2) - P(f)^2)\; d\nu +
\int_V (P(f) - f)^2\; d\nu\right)\\ 
= & \frac{1}{2}\left( \int_V (P(f^2) - P(f)^2)\; d\nu + 
\| f - P(f)\|^2_{L^2(\nu)}\right).
\ea
\ee
In particular, the integrals in the right hand side of (\ref{eq norm of f})
 are finite and non-negative.  Moreover, 
 $$
 Var_x(f\circ X_1) = P(f^2)(x) -  P(f)^2(x) \geq 0
 $$ 
 and $Var_x(f\circ X_1) \in L^1(\nu)$, for any $f \in \h_E$; hence
 equivalently 
\be\label{eq-var}
\| f \|^2_{\h_E}  =  \frac{1}{2}\left( || Var_x (f\circ X_1) ||_{L^1(\nu)} +
\mathbb E_x( f\circ X_0 - f\circ X_1)^2)_{L^1(\nu)} \right).
 \ee

\end{theorem} 

\begin{proof} By Proposition \ref{prop partial embeds}, it suffices to prove 
that the right hand side of (\ref{eq norm of f}) equals 
$\| \partial f\|^2_{\mc D}$.
Indeed, we can use the orthogonal decomposition given in Theorem
\ref{thm orthog in Diss} and write
$$
\|\partial f\|^2_{\mc D} = \|f \circ X_0 - P(f) \circ X_0 \|^2_{\mc D} 
+ \| P(f)\circ X_0 - f\circ X_1 \|^2_{\mc D}. 
$$
In the proof below, we use the following equality:
$$
\ba 
& Var_x(f\circ X_1)\\
= & \int_V  (P(f)(x) - f(y))^2 \; P(x, dy)\\
  = &  P(f)^2(x) - 2P(f)(x)
\int_V f(y)\; P(x, dy) 
 + \int_V f(y)^2\; P(x, dy)\\
=&   P(f)^2(x) - 2P(f)^2(x) +P(f^2)(x) \\
=& P(f^2)(x) - P(f)^2(x).
\ea
$$ 
Then the computation of $\|\partial f\|^2_{\mc D}$ goes as follows:
$$
\ba
\|\partial f\|^2_{\mc D}  = & \frac{1}{2}\int_V \mathbb{E}_x 
[(I - P)(f)^2\circ X_0]\; d\nu(x)\\
&\ \ + \frac{1}{2}  \int_V  \mathbb{E}_x [ (P(f)
\circ X_0 - f\circ X_1)^2]\; d\nu(x) \\
=& \frac{1}{2}\int_V (f - P(f))^2(x) \; P(x, dy) d\nu(x) \\
& \ \ +  \frac{1}{2} 
\int_V  (P(f)(x) - f(y))^2 \;  P(x, dy)d\nu(x) \\
=& \frac{1}{2}\int_V (f - P(f))^2(x) \; d\nu(x) \\ 
& \ \ +  \frac{1}{2} 
\int_V  (P(f^2)(x) - P(f)^2(x)) \; d\nu(x). \\
\ea
$$
The proof is complete. 
\end{proof}

Theorem \ref{thm main norm f} allows us to deduce a number of important
corollaries. 

\begin{corollary}\label{cor harmonic} (1) If  $f \in \h_E$, then $f - P(f) \in
 L^2(\nu)$ and 
$P(f^2) -  P(f)^2 \in L^1(\nu)$. The operator 
$$
I - P : f \mapsto f - P(f) : \h_E \to L^2(\nu)
$$
is contractive, i.e., $\| I - P\|_{\h_E \to L^2(\nu)} \leq 1$.

(2)
$$
\ba
\| f \|_{\h_E} = 0 \ & \Longleftrightarrow \ \begin{cases} & P(f^2) =  
P(f)^2 \\
& P(f) = f
 \end{cases} \ \ \ \ \qquad \nu-\mbox{a.e.}\\
 & \Longleftrightarrow \ \mathrm{both}\ f \mathrm{and} \ f^2 \ 
 \mathrm{are\ harmonic\ functions}.
 \ea
$$

(3) Let $f \in \h_E$, then 
$$
\ba
f \in \h arm_E \Longleftrightarrow & \| f \|^2_{\h_E} = \frac{1}{2} 
\int_V  (P(f^2)(x) - (Pf)^2(x)) \; d\nu(x)\\
\Longleftrightarrow & \| f \|^2_{\h_E} = \frac{1}{2} \int_V Var_x (f\circ 
X_1)\; d\nu(x).
\ea
$$

(4) 
$$
\int_V Var_x(f\circ X_1)\; d\nu(x) =  \int_V Var_x(f\circ X_n)\; d\nu(x), 
\  \ n\in N. 
$$
\end{corollary} 

\begin{proof} Statement (1) immediately follows from (\ref{eq norm of f}).

To see that (2) holds we use again (\ref{eq norm of f}). The right hand
side is zero if and only if $P(f) = f$  and $P(f^2) = P(f)^2$  a.e. (recall that,
for any function $f$, $P(f^2) \geq P(f)^2$). Since $f$ is harmonic, the latter
means that $f^2$ is harmonic. 

(3) This observation is a consequence of (\ref{eq norm of f}),  Theorem 
\ref{thm main norm f}. 
\end{proof}

\section{\textbf{Laplace operators in $\h_E$}}\label{sect LO in H_E}

\subsection{Norm estimates for $\Delta$}
In this section we use the results of Sections 
\ref{subsect_finiteEnergy} and  \ref{sect diss space} to establish several 
properties of the Laplace operator $\Delta$. This theme will be discussed
in subsequent sections where  Laplace operator $\Delta$  will be acting in
the finite  energy Hilbert space $\h_E$.

\begin{lemma}\label{lem_bndd Delta}
For any function $f \in \h_E$, the function $\Delta(f)$ belongs to 
$L^2(c^{-1}\mu)$ and 
$$
|| f ||^2_{\h_E} \geq \frac{1}{2} \int_V (\Delta(f))^2 c^{-1}\; d\mu 
=  \frac{1}{2} || c^{-1}\Delta(f) ||^2_{L^2(\nu)}.
$$
\end{lemma}

\begin{proof}
Relation (\ref{eq norm of f}) implies that

$$
\ba
|| f ||^2_{\h_E} & \geq \frac{1}{2} \int_V (f - P(f))^2 \; d\nu\\
& =  \frac{1}{2} \int_V [c(f - P(f))]^2 c^{-1}\; d\mu\\
& = \frac{1}{2} \int_V  (\Delta(f))^2c^{-1}\; d\mu\\
& = \frac{1}{2} || \Delta(f) ||^2_{L^2(c^{-1}\mu)}\\
& = \frac{1}{2} || c^{-1}\Delta(f) ||^2_{L^2(\nu)}.
\ea
$$
\end{proof}

The properties of the Laplace operator $\Delta$ depend on the spaces 
in which it acts, see Proposition \ref{prop prop of R, P, Delta} for details.
  In the next statement we show that, under some conditions,
  $\Delta$ can be even a bounded operator.
  
  \begin{proposition}\label{lem bdd Delta}
  Let the operator $\Delta$ be defined on Borel functions as in 
 (\ref{eq def of Delta}). Then if $\Delta$ is considered as an operator 
  from  $\h_E$ to  $ L^2(c^{-1}\mu)$, then   $\Delta$ is bounded and 
  $$
  || \Delta ||_{\h_E \to L^2(c^{-1}\mu)} \leq \sqrt 2.
  $$
  \end{proposition} 

\begin{proof}
We compute, for $f \in \h_E$,

$$
\ba 
|| \Delta (f) ||^2_{L^2(c^{-1}\mu)} & = 
\int_V \left( \int_V (f(x) - f (y) ) \; d\rho_x(y)\right)^2 c(x)^{-1} 
d\mu(x)\\ 
& = \int_V c(x)^2 \left( \int_V (f(x) - f (y) ) \; P(x, dy)\right)^2
 c(x)^{-1}  d\mu(x)\\ 
& \leq \int_V c(x) \left( \int_V (f(x) - f (y) )^2 \; P(x, dy)\right)  d\mu(x)\\
& \leq \int_V  \left( \int_V (f(x) - f (y) )^2 \; d\rho_x(y)\right)  d\mu(x)\\
& = \iint_{\VtV} (f(x) - f(y))^2 \;d \rho(x, y)\\
& = 2 || f ||^2_{\h_E}.
\ea
$$
\end{proof}

It can be easily seen that the same proof as in Proposition 
\ref{lem bdd Delta} works to show that the following relation holds:
 $$
 | \Delta (f) ||^2_{L^2(c^{-1}\mu)}  \leq  
 \int_{V}\int_{V} (f(x) - f(y))^2 \;d \rho_x(y)d\nu(x).
 $$

\subsection{Example: the energy Hilbert space $\h_E$ over a 
probability  space}\label{subsect Energy on prob space}

In this subsection, we consider a particular case when the symmetric 
measure $\rho$ is the  product of finite measure. It turns out
that, in this case, the finite energy space can be identified with a 
$L^2$-space.

Our setting here are: $\sms$ is a probability measure space, $\mu(V) =1$,
 $r: V \to \R_+$ is a non-negative $\mu$-integrable  Borel function, and 
 $\mu_r$ is a measure on $\VB$ such that $d\mu_r(x) = r(x) d\mu(x)$.
 Since $r \in L^1(\mu)$, the measure $\mu_r$ is finite. To define a 
 symmetric measure $\rho = \rho_r$ on $\vv$, we set
$$
d\rho(x, y) = r(x)r(y) d\mu(x) d\mu(y),\quad (x,y) \in \VtV
$$
Then $\rho = \mu_r \times \mu_r$, and the disintegration of $\rho =
\int_V \rho_x d\mu(x)$ gives measures $\rho_x, x\in V,$ such that
$d\rho_x(y) = r(x)r(y) d\mu(y)$. Then the function $c(x) = \rho_x(V)$
is found by 
$$
c(x) = \int_V \; d\rho_x(y) = r(x) \int_V r(y) \; d\mu(y) = 
\mathbb E_\mu(r) r(x).
$$ 
We note that $c \in L^1(\mu)$ because 
$$
\int_V c(x)\; d\mu(x) = \mu_r(V) || r ||_{L^1(\mu)}.
$$

Having the measure $\rho$ on $\vv$, we determine the operators
$R, P$, and $\Delta$ acting on bounded Borel functions $\FVB$ as follows
$$
R(f)(x) = \int_V f(y) \; d\rho_x(y) = \int_V f(y)r(x)r(y) \; 
d\mu(y) = \mathbb E_{\mu_r}(f) r(x),  
$$
$$
P(f)(x) = \frac{1}{c(x)} R(f)(x) = \frac{1}{\mu_r(V)}\mathbb E_{\mu_r}(f)
= \frac{\mathbb E_{\mu_r}(f)}{\mathbb E_{\mu}(r)},
$$
and
\be\label{eq_Delta mu_r}
\Delta(f)(x) = c(x)f(x) - R(f)(x) = r(x) \left(\mathbb E_{\mu}(r)f(x) -
\mathbb E_{\mu_r}(f)\right).  
\ee
We identify here the number $\mathbb E(f)$ with
the constant function $\mathbb E(f) \mathbbm 1$.

In a particular case, when $r(x) = 1$, we obtain that 
$$
P(f)(x) = \int_V f(y) \; d\rho_x(y) = \int_V f(y) \; d\mu(y) = 
\mathbb E(f) \mathbbm 1,
$$
and 
$$
\Delta (f) =  f - P(f) = f - \mathbb E(f) \mathbbm 1
$$
because $c(x) =1$.

\begin{theorem}\label{thm-Energy example}
Let $\sms$, $r$, $\mu_r$ be as above,  and  let $\rho_r = \mu_r  \times 
\mu_r$. Then the energy space $\h_E(\rho_r)$, defined by the symmetric
 measure $\rho_r$, is  isometrically isomorphic to  $L^2(\mu_r)$. The 
 isometric map $\alpha : \h_E \to L^2(\mu_r)$ is defined  as follows:
  for any $f \in \h_E = \h_E(\rho_r)$, 
$$
\alpha(f)  = \frac{1}{\sqrt{\mathbb E_\mu(r)}}
(\mathbb E_\mu(r) f - \mathbb E_{\mu_r}(f)).
$$ 
\end{theorem}

\begin{proof} To prove the theorem, we compute the norms
$\| f \|_{\h_E}^2$ and $|| \alpha(f) ||^2_{L^2(\mu_r)}$.
For $\| f \|_{\h_E}^2$:
$$
\ba
\| f \|_{\h_E}^2 = &\frac{1}{2}  \iint_{\VtV} (f(x) - f(y))^2 \; 
d\mu_r(x)d\mu_r(y)\\
= & \iint_{\VtV} (f^2(x) - f(x)f(y))\; d\mu_r(x)d\mu_r(y)\\
=& \ \mathbb{E}_{\mu_r}(f^2)\mu_r(V)  - 
\left(\int_V f(x)\; d\mu_r(x)\right)^2\\
= & \ \mathbb{E}_{\mu}(r)\mathbb{E}_{\mu_r}(f^2) - 
\mathbb{E}_{\mu_r}(f)^2.
\ea
$$ 
On the other hand, we can find
$$
\ba
|| \alpha(f) ||^2_{L^2(\mu_r)} & = \frac{1}{\mathbb E_\mu(r)}
\int_V (\mathbb E_\mu(r) f - \mathbb E_{\mu_r}(f))^2\; d\mu_r(x)\\
& = \frac{1}{\mathbb E_\mu(r)}
\int_V\left(\mathbb E_\mu(r)^2 f(x)^2   - 2f(x) 
\mathbb E_\mu(r)\mathbb E_{\mu_r}(f) + \mathbb E_{\mu_r}(f)^2\right)
\; d\mu_r(x)\\
& =   \frac{1}{\mathbb E_\mu(r)}[ \mathbb E_\mu(r)^2
\mathbb E_{\mu_r}(f^2) - 2\mathbb E_\mu(r)\mathbb E_{\mu_r}(f)^2 
+   \mathbb E_\mu(r)\mathbb E_{\mu_r}(f)^2 ]\\
& =  \mathbb{E}_{\mu}(r)\mathbb{E}_{\mu_r}(f^2) - 
\mathbb{E}_{\mu_r}(f)^2.
\ea
$$
This proves that $\alpha$ is an isometry. We observe that this proof 
shows that
$$
\alpha(\h_E) = L^2_0(\mu_r),
$$
where $L^2_0(\mu_r)$ is the subspace of functions from $L^2(\mu_r)$ 
with zero integral. 
\end{proof}

\begin{remark}
(1) It is clear that, in the case when $\rho_r = \mu_r\times\mu_r$, the 
space $\h_E(\rho_r)$ does not contain nontrivial harmonic functions because
as follows from \eqref{eq_Delta mu_r}
$$
\Delta(f) = 0 \ \Longleftrightarrow \ f(x) = 
\frac{\mathbb E_{\mu_r}(f)}{\mathbb E_{\mu}(r)} \ 
 \Longleftrightarrow \ f(x) = const.
$$ 

(2) It follows from Theorem \ref{thm-Energy example} that, in the case 
when $r =1$, 
$$
\| f \|_{\h_E} = \| \Delta f \|_{L^2(\mu)}.
$$

(3) A similar approach can be used to study symmetric measures $\rho_{r,q}$
on $\vv$ which are defined by a pair of nonnegative integrable functions
$r, q : V \to \R_+$:
$$
d\rho_{r,q} (x,y) = (r(x) q(y) + r(y) q(x)) d\mu(x) d\mu(y).
$$ 
\end{remark}

Suppose now that, for  the probability measure space $\sms$, the 
function $r(x) = \mathbbm 1$, and therefore $\rho = \mu \times \mu$. 
Then we can make more precise the formula for the inner product in
$\h_E$.

\begin{corollary}
For any $f, g\in \h_E$,
$$
\langle f, g \rangle_{\h_E} =  \mathrm{Cov_\mu}(f, g).
$$
In particular,
$$
\langle \chi_A, \chi_B\rangle_{\h_E} = \mu(A \cap B) - \mu(A) \mu(B)
$$
and $\langle \chi_A, \chi_B\rangle_{\h_E} = 0$ if and only if
the sets $A$ and $B$ are independent.  
\end{corollary}

\begin{proof}
By definition of the inner product in $\h_E$, we calculate
$$
\ba 
\langle f, g\rangle_{\h_E} = & \frac{1}{2} \iint_{VtV} 
(f(x) - f(y)) (g(x) - g(y))\; d\mu(x)d\mu(y)\\
=&  \int_{V} (f(x) - \mathbb{E_\mu}(f)) (g(x) - \mathbb{E_\mu}(g))\; 
d\mu(x)\\
=&\ \mathbb{E_\mu}(fg) - \mathbb{E_\mu}(f)\mathbb{E_\mu}(g)\\
= &\ \mathrm{Cov_\mu}(f, g).
\ea
$$
The proof of the other  formula in the lemma follows from Theorem 
\ref{thm_stucture of energy space} (1)
$$
\ba 
\langle \chi_A, \chi_B\rangle_{\h_E} = & \ \rho((A \cap B) \times V) -
\rho(A \times B)\\
= &\ \mu(A \cap B) - \mu(A) \mu(B).
\ea
$$
Hence. the functions $\chi_A$ and $\chi_B$ are orthogonal if and only if
$\mu(A \cap B) = \mu(A)\mu(B)$.
\end{proof}

In what follows, we continue discussing energy spaces and Laplace operators
on measure spaces with finite measures.

We recall that we consider symmetric measures $\rho = 
\int_V \rho_xd\mu(x)$ on $\vv$ such that the function
 $c: x \mapsto \rho_x(V)$ belongs to $ \Lloc$, see Assumption 1.
Then, for the measure $d\nu(x) = c(x)d\mu(x)$, we  have
$\Dfin(\mu) \subset \Lloc \cap L^1_{\mathrm{loc}}(\nu)$. It follows from
local integrability of $c$ that
\be\label{eq Bfin rho}
\Bfin(\rho) \supset \{ A\times B \in \B\times \B : A \in \Bfin(\mu)\}.
\ee

Let $A\in \Bfin(\mu)$ be a Borel set of positive measure $\mu$. Consider 
the restriction $\rho_A$ of a symmetric measure $\rho$ on the subset 
$A \times A$, i.e., 
$$
\rho_A(C \times D) = \rho((C \times D) \cap (A \times A)).
$$
By \eqref{eq Bfin rho}, $\rho_A$ is a finite symmetric measure for every
$A\in \Bfin(\mu)$.

For the finite measure space $(A, \B_A, \mu_A)$ and  measure 
$\rho_A$, we can define the symmetric operator $R_A$ as follows:
$$
R_A(f)(x) = \int_A f(y) \; d\rho_x^A(y),
$$
where $x \mapsto \rho_x^A$ is the family of fiber measures arising
in disintegration of $\rho_A$. We set $c_A(x) = \rho_x^A(V)$. When 
$R_A$ and $c_A$ are defined, we can construct the graph Laplacian
$$
\Delta_A(f)(x) = c_Af(x) - R_A(f)(x).
$$

Given a $\sigma$-finite measure space $\sms$ and a symmetric measure
$\rho$ on $\vv$, we choose a sequence
$(A_n)$ of Borel sets such that $A_n \subset A_{n+1}$, $\mu(A_n)
< \infty$ and 
\be\label{eq_union of A_n}
\bigcup_{n =1}^\infty A_n = V.
\ee
 For every $A_n$, we define the objects $\rho_n = \rho_{A_n},
\rho_x^{(n)} = \rho_x^{A_n}, c_n(x) = c_{A_n}(x), R_n = R_{A_n}$, and
$\Delta_n = \Delta_{A_n}$ as it was done above.

\begin{lemma}\label{lem R_n to R} 
For the objects defined above, the following sequences converge: 

(1) for any set $C \in \B\times \B$,  $\rho_n(C) \to \rho(C)$ and 
$\rho_x^{(n)}(C) \to \rho_x(C)$;  

(2) $c_n(x) \to c(x)$ a.e. $x\in V$;

(3) for any function $f\in \FVB$, 
$$
R_n(f)(x) \to R(f)(x), \ \ \mbox{and} \ \  \Delta_n(f)(x) \to \Delta(f)(x) \ 
\mbox{a.e.}\ x \in V.
$$
\end{lemma}

\begin{proof}
The proof follows directly from the definitions.
\end{proof}

Suppose $A \in \Bfin(\mu)$ and $\rho_A$ is the restriction of
a symmetric measure $\rho$ onto $A \times A$. Define the finite energy
space $\h_E(\rho_A)$ as the space of functions on $\VB$ satisfying
$$
|| f ||^2_{\h_E(\rho_A)} := \frac{1}{2} \iint_{A \times A}
(f(x) - f(y))^2\; d\rho_A(x,y) < \infty.
$$
As usual, we denote by $\nu_A$ the measure $c_A\mu_A$ where the 
function $c_A$ is defined by $\rho_A$.

\begin{lemma} \label{lem H by rho_n}
The space $\h_E(\rho_A)$ is embedded into 
 $L^2(A, \nu_A)$ and 
$$
|| f ||^2_{\h_E(\rho_A)}  \leq 2 || f ||^2_{L^2(A,\nu_A)}. 
$$
\end{lemma}

\begin{proof} We first remark that $\nu(A) < \infty$ since $c \in \Lloc$.
We need to show that, for any $f\in L^(A, \nu)$, the function $f$ belongs
to  $\h_E(\rho_A)$. Indeed, we have
$$
\ba 
|| f ||^2_{\h_E(\rho_A)} & = \frac{1}{2} \iint_{A \times A}
(f(x) - f(y))^2\; d\rho_A(x,y)\\
& \leq \iint_{A \times A} (f(x)^2 +  f(y)^2)\; d\rho_A(x,y)\\
& = 2\int_A f(x)^2 c_A(x) d\mu_A(x)\\ 
& =2 || f ||^2_{L^2(A,\nu_A)}. 
\ea
$$
\end{proof}

Let $(A_n)$ be  a sequence of Borel subsets of $V$ satisfying 
\eqref{eq_union of A_n}. 
For  measures $\rho_{A_n}= \rho_n$, we can obviously define 
embedding of the Hilbert space
$\h_E(\rho_n)$  into $\h_E(\rho_{n+1})$. 
It follows from Corollary \ref{cor harmonic} and Lemma 
\ref{lem H by rho_n}  that every space $\h_E(\rho_n)$ does not contain
harmonic functions. Therefore, there are functions in $\h_E(\rho)$ which
are not in   $\h_E(\rho_n)$.

\section{\textbf{Properties of functions from the finite energy space}}
\label{sect energy}

In this section, our main object is the finite energy space $\h_E$, see
Definition \ref{def f.e. space}, and 
its properties. In the literature, some authors use the terms Diichlet space
and Dirichlet form for the inner product. We mention here several references
that may be useful for the reader \cite{GimTrutnau2017, Oshima2013,
KoskelaZhou2012, Jonsson2005, Rozkosz2001, Jorgensen2012}.

\subsection{Properties and structure of the energy space}
\label{subsect_propertise of H}

We begin with a discussion of immediate properties of functions from
the space $\h_E$. We first recall what results were proved in 
\cite{BezuglyiJorgensen2018}. 

We recall that in  Section \ref{subsect_finiteEnergy} (see Subsections
 \ref{subsect Inn prod} and \ref{subsect harmonic fncts}) some basic
  properties of functions from 
the energy space $\h_E$ have been already discussed.  It is important
for us to use the inclusions
$$ 
\Dfin(\mu) \subset \Dfin(\nu) \subset \h_E,  \quad \Dfin(\mu) \subset
 L^2(\mu) \cap L^2(\nu) \cap \h_E,
$$
and the assumption \eqref{eq assumption 2} which states that 
$$
\h_E \subset  L^2_{\mathrm{loc}} (\mu).
$$
Since $L^2_{\mathrm{loc}}(\mu) \subset \Lloc$, this 
means  that any function from the energy space is locally integrable. 

In Section \ref{sect Emb L2 into H}, we proved a number of statements
about the closure of $\Dfin(\mu)$ and $\Dfin(\nu)$ in $\h_E$. In 
particular,  the closure of $\Dfin(\mu)$ in $\h_E$ is the subspace
which is orthogonal to the space of harmonic functions, see Theorem 
\ref{thm_stucture of energy space}. Here we reprove the result for 
 the closure of $\Dfin(\nu)$ in $\h_E$ using a different method based
 on the embedding of $\h_E$ into the dissipation space 
 $Diss$ and Theorem \ref{thm main norm f}. 

\begin{proposition}\label{prop L2(nu) is closure}
 The following relation holds
$$
\ol{\Dfin(\nu)}^{\h_E} = \ol{L^2(\nu)}^{\h_E} = L^2(\nu),
$$
where $L^2(\nu)$ is considered as a subspace of $\h_E$.
\end{proposition}

\begin{proof}
We recall that because of the inequality  
$|| \chi_A ||_{\h_E}^2 \leq \nu(A)$, the 
linear space $\Dfin(\nu)$ is a subspace of the energy space $\h_E$, 
hence 
$ \ol{\Dfin(\nu)}^{\h_E} \subset \h_E$. On the other hand, for 
every $f\in L^2(\nu)$ there exists a sequence $(f_n) \subset 
\Dfin(\nu)$ such that 
$$
|| f - f_n ||_{L^2(\nu)} \to 0, \ \ \ n \to \infty.
$$
It suffices to show that every function from $L^2(\nu)$ can be
approximated by functions from $\Dfin(\nu)$ in the norm of $\h_E$.

We use relations (\ref{eq norm of f}) and \eqref{eq-var}, and find the norm 
of $f - f_n $ in $\h_E$:
$$
\| f  - f_n \|^2_{\h_E} 
=  \frac{1}{2}\left( \int_V [P((f - f_n)^2) - P(f- f_n)^2]\; d\nu + 
\| (f - f_n) - P(f - f_n)\|^2_{L^2(\nu)}\right).
$$
By Proposition \ref{prop prop of R, P, Delta} (4), the operator $P$ 
considered in the  space $ L^2(\nu) $ is contractive, and 
$$
\| (f - f_n) - P(f - f_n)\|^2_{L^2(\nu)} \to 0, \ \ \ \ n \to \infty
$$
because $f_n \to f$ in $L^2(\nu)$. 

Next, we note that 
$$
 \int_V P((f - f_n)^2) \; d\nu = \int_V (f - f_n)^2 \; d\nu \ \to 0,  
$$
because  $\nu$ is a $P$-invariant measure. 

To prove that the remaining term in the formula for $\| f  - f_n \|
^2_{\h_E}$ tends to zero, we represent it as inner product in $L^2(\nu)$, 
and conclude that
$$
\int_V P(f- f_n)^2\; d\nu = \langle P(f- f_n), P(f- f_n)
\rangle_{L^2(\nu)}  \ \to 0, \ \ \ n \to \infty.
$$
Therefore, $(f_n)$ is a converging sequence of elements from $\h_E$ and 
the limit, the function $f$, belongs to $\h_E$.  

\end{proof}

It follows from the results proved in Section \ref{sect Emb L2 into H}
and Proposition \ref{prop L2(nu) is closure} that the following corollary
holds.

\begin{corollary} The finite energy Hilbert space admits the orthogonal
decomposition 
$$
h_E = L^2(\nu) \oplus \h arm. 
$$
\end{corollary}

\begin{proof} Indeed, the statement follows from the relation
$$
\ol{\Dfin(\nu)}^{\h_E}  = \ol{\Dfin(\mu)}^{\h_E} \perp \h arm
$$
Theorem \ref{thm Royden}, and Corollary \ref{cor closures}. 
The  proof of the fact that every harmonic function is orthogonal to 
$\chi_A$, $A\in \Bfin(\nu)$, is similar to that of Proposition 
\ref{prop orthogonal}.
\end{proof}

\subsection{Dipoles in the energy Hilbert space}\label{subsect dipoles}

We recall that in the theory of electrical networks the notion of dipoles 
plays a crucial role for the study of properties and structure of he finite
energy space. Let $(V, E, c)$ be an electrical network with conductance
function $c = (c_{xy})$ and the Laplacian $\Delta$, see Introduction
for details. Then one can show that, for any edge $(xy) \in E$, there exists 
a unique element $v_{xy}$ of $\h_E$, called a \textit{dipole},   such that 
$$
\Delta v_{xy} = \delta_x - \delta_y.
$$
It turns out that, for any $f \in \h_E$,
$$
\langle f, v_{xy} \rangle_{\h_E} = f(x) - f(y).
$$

Our goal in this subsection is to formulate and prove similar results
for the measurable analogue of $\h_E$ and $\Delta$. 
discrete concept. 

\begin{definition} 
We say that the family of functions $\{v_{A,B} : A, B \in \Bfin(\mu)\}$ 
consists of \emph{dipoles} (more precisely, $\mu$-dipoles)  if they satisfy
 the equation
\be\label{eq_dipoles def}
\Delta v_{A,B} = \chi_A - \chi_B.
\ee

Similarly, we define $\nu$-\textit{dipoles} as functions $w_{A, B}$ such that
\be\label{eq_dipoles def nu}
\Delta w_{A,B} = c(\chi_A - \chi_B).
\ee
\end{definition}

\begin{proposition} \label{lem dipoles}
For any sets $A, B \in \Bfin(\mu)$, 
\be\label{eq dipole v}
\langle f, v_{A,B}\rangle_{\h_E} = \int_A f\; d\mu - \int_B f\; d\mu,
\ee
and 
\be\label{eq dipole w}
\langle f, w_{A,B}\rangle_{\h_E} = \int_A f\; d\nu - \int_B f\; d\nu.
\ee
\end{proposition}

\begin{proof}
The formulas  follow immediately from our standing assumption that
 functions from $\h_E$ are locally integrable. Indeed, we have
 
 $$
 \ba
\langle f, v_{A,B}\rangle_{\h_E} = & \ \frac{1}{2}\iint_{\VtV}
(f(x) - f(y)) (v_{A, B}(x) - v_{A, B}(y)) \; d\rho(x, y)\\
= & \ \frac{1}{2}\iint_{\VtV} \left[f(x) (v_{A, B}(x) - v_{A, B}(y))
- f(y) (v_{A, B}(x) - v_{A, B}(y))\right]  \; d\rho_x(y)d\mu(x)\\
= &\ \int_V f(x) \left(\int_V (v_{A, B}(x) - v_{A, B}(y))\; d\rho_x(y)
\right) d\mu(x)\\
= &\ \int_V f(x)  \Delta(v_{A,B})(x) \; d\mu(x)\\
= & \ \int_V f(x) (\chi_A(x) - \chi_B(x))\; d\mu(x).
\ea
 $$
Relation of  \eqref{eq dipole w} is proved similarly. 
\end{proof}

\begin{remark}\label{rem measure mu_f}
 We recall that the following formula follows 
from Theorem \ref{thm inner prod via Delta}: for any  function
$f \in \h_E$ and  any set $A\in\Bfin(\mu)$,
\be\label{eq mu_f via inner prod}
\langle f, \chi_A \rangle_{\h_E} = \int_A \Delta f\; d\mu. 
\ee
It is important to remember that this formula is proved under  our 
basic assumptions that the function $c \in \Lloc$ and 
$\h_E \subset L^2_{\mathrm{loc}}(\mu)$.

We use \eqref{eq mu_f via inner prod} to define a new measure 
$\mu_f(\cdot)$ on $\VB$ where $f \in \h_E$. We observe that, in general,
$\mu_f$ is a finite additive measure. It is $\sigma$-additive when the 
function $f\in \h_E$ satisfies the property $\Delta(f) \in L^1(\mu)$. 
This measure $\mu_f$ will be used in Section \ref{sect RKHS} to 
construct a reproducing kernel Hilbert space.
\end{remark} 

\begin{corollary}
The sets 
$\mc D(w) := \mathrm{Span}\{w_{A,B} : A, B \in \Bfin(\mu)\}$ and 
$\mc D(v) := \mathrm{Span}\{v_{A,B} : A, B \in \Bfin(\mu)\}$  are 
dense in the Hilbert space $\h_E$.
\end{corollary}

\begin{proof}
Suppose, for contrary, that there exists a vector $f\in \h_E$ which is 
orthogonal to $\mc D(w)$. Then it follows from \eqref{eq dipole w} that,
for any sets $A,B \in \Bfin(\mu)$, 
$$
\int_A f\; d\nu = \int_B f\; d\nu.
$$  
It is possible only when $f = 0$.
\end{proof}

It remains to show that dipoles $w_{A,B}$ always exist in the finite energy
 space $\h_E$. To do this, we use the approach elaborated in the theory
of electrical networks, see \cite{Jorgensen_Pearse2011}. It is obvious
that one set in $w_{A,B}$ can be fixed because $w_{A,B} = w_{A,A_0} -
w_{A_0, B}$. 

\begin{lemma}\label{lem_existence of dipoles}
Let $A, A_0 \in \Bfin(\mu)$ and $f \in \h_E$. Then 
$$
f \mapsto \int_A f \; d\nu - \int_{A_0} f \; d\nu
$$
is a bounded linear functional on $\h_E$. 
\end{lemma}

\begin{proof}
The idea of the proof is similar to the case of discrete networks, see 
\cite{JorgensenPearse2010, Jorgensen_Pearse2011}.
\end{proof}

The situation with the family of $\mu$-dipoles is slightly different as shown
in the following statement.

\begin{proposition}
For $A,B\in \Bfin(\mu)$, the function $v_{A, B}$ belongs to $\h_E$ if and 
only if $c^{-1} \in L^2_{\text{loc}}(\mu)$.

\end{proposition}

\begin{proof}
Without loss of generality, we can assume that, in the definition of
$v_{A, B}$ (see \eqref{eq dipole v}),  $A \cap B = \emptyset$. Since for
any $f\in \h_E$,  $\Delta(f) = c(I -P)(f)$ and the operator
$I - P : \h_E \to L^2(\nu)$ is a contraction (see Corollary 
\ref{cor harmonic}), we conclude that 
$$
\Delta(f) \in cL^2(\nu) = L^2(c^{-1}\mu).
$$
Applying this fact to $f = v_{A,B}$, we obtain that 
$$
\ba
\infty & > & \int_V (\Delta (v_{A,B}))^2c^{-1}\; d\mu\\
& =& \int_V (\chi_A - \chi_B)^2c^{-1}\; d\mu\\
& =&  \int_V (\chi_A  + \chi_B)c^{-1}\; d\mu\\
& =&  \int_A c^{-1}\; d\mu +  \int_B c^{-1}\; d\mu\\
\ea
$$
which proves the proposition. 
\end{proof} 

\subsection{Application to machine learning problems}
This subsection is devoted to an application of  the graph Laplace 
operator considered in previous sections to the so called learning problem.
The problem we formulate below is an optimization problem with a penalty
term, see \cite{ArgyriouMicchelliPontil2010, GuofanZhou2016, 
PoggioSmale2003, SmaleZhou2009a, SmaleZhou2009, SmaleZhou2007,
Smaleyao2006} for more details. 

We first recall the following results proved in \cite{BezuglyiJorgensen2018}. 
Let the measure  space $\sms$ and the graph Laplace operator $\Delta$
be as above, and let $\h_E$ be the energy space. Then $\Delta$ can be
realized in $L^2(\mu)$ and $\h_E$, and  denote by $\Delta_2$ and
$\Delta_{\h_E}$ the corresponding operators in $L^2(\mu)$ and 
$\h_E$. 

Let  $J$ and $K$ be two densely defined operators that constitute a
 symmetric pair of operators: 
\be\label{eq def J}
L^2(\mu) \ \stackrel{J}\longrightarrow \ \h_E
\ee
and 
\be\label{eq def K}
\h_E \ \stackrel{K}\longrightarrow \ L^2(\mu).
\ee
For  $\varphi \in \mc D_Q, \psi \in \mc C$, where $\mc D_Q$ and $\mc 
C$ are dense subsets, we have
\be\label{eq symm pair J and K}
\langle J \varphi, \psi \rangle_{\h_E}  = 
\langle  \varphi, K\psi \rangle_{L^2(\mu)},
\ee
see details in  \cite[Lemma 8.4]{BezuglyiJorgensen2018}.

It follows that: 

(1) $J^* = K$ and $K^* = J$, 

(2) the operators $J^*J$ and $K^*K$ are self-adjoint in $L^2(\mu)$ and 
 $\h_E$, respectively. 
 
 In \cite[Theorem 8.5]{BezuglyiJorgensen2018}, we proved the following
  result.
 
\begin{theorem} \label{thm Delta_h} 
The Laplace operator $\Delta$ admits
 its realizations in the Hilbert spaces $L^2(\mu)$ and $\h_E$ such that:

(i) $\Delta_2 = J^*J$ is a positive definite essentially self-adjoint operator;

(ii)  $\Delta_{\h}$ is a positive definite and symmetric operator which is
not self-adjoint, in general;  a self-adjoint  extension $\wt \Delta_{\h_E}$ 
of   $\Delta_{\h}$  is given by  the operator $JJ^* = K^*K$.
\end{theorem}

The goal of this subsection is to apply the above results to 
$L^2$-regulation for learning problem. We will show how to find the 
minimum of the function 
\be\label{eq def function Q(h)}
Q(h) = \| \psi - Kh\|^2_{L^2(\mu)} + \gamma \| h \|^2_{\h_E}, \quad
h \in \h_E,
\ee
where $\psi$ is a fixed function from $L^2(\mu)$, $\gamma >0$,
 and $ K : \h_E \to L^2(\mu)$ is defined in (\ref{eq def K}). 
This problem is interpreted as follows. Suppose $\psi$ is a given function
representing some data. Then the term $\| \psi - Kh\|^2_{L^2(\mu)}$ 
corresponds to the least square approximation by functions $h$ from a 
feature space, and  $\gamma \| h \|^2_{\h_E}$ is the so called penalty
term, see \cite{ArgyriouMicchelliPontil2010, GuofanZhou2016,
PoggioSmale2003, SmaleZhou2007, SmaleZhou2009, SmaleZhou2009a, 
Smaleyao2006} for more information.

\begin{proposition}\label{prop approx}
Let $K, J$ be the symmetric pair of operators defined in (\ref{eq def J}) and
(\ref{eq def K}), and let $\wt \Delta_{\h_E} = K^*K$ be the self-adjoint
extension of $\Delta_{\h_E}$. Then, for a given function $\psi \in 
L^2(\mu)$,
$$
\mathrm {arg}\min\{Q(h) : h \in \h_E\} = (\gamma I + \wt 
\Delta_{\h_E})^{-1} J \psi.
$$
\end{proposition} 

\begin{proof}
To minimize $Q$, it suffices to find a function $h$ such that
$$
\frac{d}{d\e} Q(h + \e k)\mid_{\e = 0}\ = 0, \qquad \forall k \in \h_E. 
$$
Clearly, we need to know only the linear term with respect to $\e$ in 
$Q(h + \e k)$ because other terms vanishes after differentiation and 
substitution $\e =0$. We compute
$$
\ba 
& \frac{d}{d\e} Q(h + \e k)\mid_{\e = 0}\\
= & - 2 \langle \psi, Kk\rangle_{L^2(\mu)} + 
2 \langle Kh, Kk\rangle_{L^2(\mu)} + 2 \langle h, k\rangle_{\h_E}\\
=& - 2 \langle J \psi, k \rangle_{\h_E} + 2 \langle K^*K h, 
k \rangle_{\h_E} + 2 \langle h, k\rangle_{\h_E}\\
= &\  2\langle \wt\Delta_{\h_E} h + \gamma h - J\psi, k\rangle_{\h_E}.\\
\ea
$$
It follows that the function $h$ must satisfy the property
$$
 \wt\Delta_{\h_E} h + \gamma h - J\psi = 0
$$
or 
$$
h = (\gamma I + \wt\Delta_{\h_E})^{-1}J\psi
$$
which is the desired conclusion.
\end{proof}

We note that the operator 
$$
(\gamma I + \wt\Delta_{\h_E})^{-1}J = (\gamma I +JK)^{-1}J :
L^2(\mu) \to \h_E
$$
 is bounded, contractive, and  self-adjoint.

\section{\textbf{Reproducing kernel Hilbert spaces}}\label{sect RKHS}

In this section we will show that, for transient Markov processes,
 the energy space $\mathcal H_E$ can be realized as a reproducing kernel
 Hilbert space (RKHS) for a positive definite kernel. 
We give also two more reproducing kernel Hilbert spaces that are related
to the symmetric measure $\rho$   on  $\vv$ and the measure $\nu$
on $\VB$.
 The standard references  for the theory of RKHS are 
 \cite{Aronszajn1950,  AronszajnSmith1957, 
 Adams_et_al1994, PaulsenRaghupathi2016, SaitohSawano2016}, 
 see also more recent  
 results and various applications in \cite{AplayJorgensen2014,
  AlpayJorgensen2015, 
 JorgensenTian2015, JorgensenTian-2016, BerlinetThomas-Agnan2004}.
 
 \subsection{Definition of RKHS}\label{subsect def RKHS}
 We begin with reminding the reader the definition of a RKHS. 
 
 Let $S$ be a set, and 
 let $K: S\times S \to \R$ be a \textit{positive definite kernel}, i.e., the 
 function $K(s, t)$ has the property 
 $$
 \sum_{i, j =1}^N \alpha_i\alpha_j K(s_i, s_j) \geq 0
 $$
 which holds for any $N \in \N$ and for any $s_i \in S, \alpha_i \in \R, \
  i=1,..., N$.
 (For a complex-valued kernel $K$ some obvious changes must be made).
 
\begin{definition} \label{def RKHS}
Fix $s \in S$ and denote by $K_s$ the function $K_s(t) = K(s, t)$ 
of one variable  $t \in S$.  Let $\mathcal K := 
\mbox{span}\{ K_s : s\in S\}$. 
The \textit{RKHS} $\h(K)$ is the Hilbert space obtained by completion of 
$\mathcal K$ with respect to the inner product defined on $\mathcal K$ by
$$
\left\langle\ \sum_i \alpha_i K_{s_i}, \sum_j\beta_j K_{s_j}\ 
\right\rangle_{\h(K)} := \sum_{i, j =1}^N \alpha_i\beta_j K(s_i, s_j)
$$
\end{definition}

It immediately follows from Definition \ref{def RKHS} that 
$$
\langle K(\cdot, s), K(\cdot, t)\rangle_{\h(K)} = K(s, t).
$$
More generally, this result can be extended to the following property that 
 characterizes functions from the RKHS $\h(K)$. For any $f \in  \h(K)$ and
  any $s\in S$, one has 
\be\label{eq char prop RKHS}
f(s) = \langle f(\cdot), K(\cdot, s)\rangle_{\h(K)}. 
\ee
It suffices to check that (\ref{eq char prop RKHS}) holds for any function
from $\mathcal K$ and then extend it by continuity. 

One can check that the following property characterizes functions from
the reproducing kernel Hilbert space $\h(K)$ constructed by a positive
definite kernel $K$ on the set $S$. We formulate it as a statement for 
further references.

\begin{lemma}\label{lem criterion}
A function $f$ is in $\h(K)$ if and only if there exists a constant $C = C(f)$
such that for any $n \in \N$, any $\{s_1, ... ,s_n\}  \subset S$, and any
$\{\alpha_1, ... , \alpha_n\} \subset \R$, one has
\be\label{eq criterion RKHS}
\left(\sum_{i=1}^n \alpha_i f(s_i)\right) ^2 \leq C(f) 
\sum_{i, j =1}^n \alpha_i\alpha_j K(s_i, s_j).
\ee
\end{lemma}

We follow  \cite{JorgensenTian2017} in the following definition. 
Let $K(s,t)$ be a positive definite kernel as above.  It is said that 
that a measure space $(X, \mathcal A, m)$ and functions $K^*_s : S \to
L^2(m)$ define  a \textit{realization} of $K(s, t)$ if
\be\label{eq realization}
K(s, t) = \langle K^*_s, K^*_t \rangle_{L^2(m)} = \int_X K^*_s K^*_t 
\; dm.
\ee
It is said that the realization is \textit{tight} if the set of functions 
$\{K^*_s(\cdot) : s\in S\}$ is dense in $L^2(X, \mc A, m)$. 

We note that the converse approach can  be also used. Namely, given a 
set of functions $\{K^*_s\}$ from $L^2(X, \mc A, m)$, one can define a 
positive definite kernel $ K(s, t)$ by formula (\ref{eq realization}).

\subsection{Reproducing kernel Hilbert space over  $\Bfin$}\label{Bfin}
We use in this subsection our standard setting: a sigma-finite measure
space $(V, \B, \mu)$, a symmetric measure $\rho $ on $V \times V$, and 
the function $c(x) = \rho_x(V)$.
 Then we define the sequence of transition probabilities $P_n(x, A)$,  
 the positive operator $P$ acting  by the  formula $P(f)(x) = \int_V f(y) \; 
 P(x, dy)$ such that $\nu P = \nu$, where $d\nu(x) = c(x)d\mu(x)$. 
 Let also $\Bfin$ be the algebra of Borel sets of finite measure $\mu$.

 Recall that together with the symmetric measure $\rho$ we have defined
  the  sequence of symmetric 
 measures $(\rho_n)$ such that, for $A, B \in \Bfin$, 
 $$
 \rho_n(A \times B) = \int_A P_n(x, B)\; d\nu(x) = \int_V \chi_A 
 P^n(\chi_B)\; d\nu =  \langle\chi_A, P^n(\chi_B) \rangle_{L^2(\nu)}. 
 $$  
In particular, $\rho_0(A\times B) = \nu(A\cap B)$ for $A, B \in 
\Bfin(\mu)$. 
 
 We will assume that the Markov process defined by $(P_n)$ is 
 \textit{transient}. In other words, this assumption means that the 
 Green's function 
 \be\label{eq_conv trans}
 G(x, A) = \sum_{n =0}^\infty P_n(x, A)
 \ee
 is well defined for any $A\in \Bfin(\mu)$. In order to emphasize that 
 $G(x, A)$
 is a function in $x$ for every fixed $A$, we will use also the notation
 $G_A(\cdot)$. 
 
 For  every $A \in \Bfin(\mu)$ and $n\in \N_0$, the function 
 $P^n(\chi_A)(x)$ belongs to  $\h_E$, hence assuming the convergence 
 of the series in \eqref{eq_conv trans}, we note that the Green's function
$G(x, A)$ can be viewed as an element of the energy space $\h_E$. 
A direct computation gives the formula for the norm of $G(x, A)$:
 \be\label{eq-norm G_A}
 || G(\cdot, A) ||^2_{\h_E} = \sum_{n=0}^\infty \rho_n(A \times A).
 \ee
 (details are given in Theorem \ref{thm on G_A} below).

For the sake of completeness, we  include the following theorem which 
was mostly proved in 
\cite{BezuglyiJorgensen_2018}. 

 \begin{theorem}\label{thm on G_A} 
Let $\sms, \rho_n, \h_E, $ and $G(x, A)$ be  the objects defined as above.
Then  the following  properties hold. 
 
  (1) For any sets $A, B \in \Bfin$, we have
\be\label{eq_inner prod G_A and G_B}
\langle G_A, G_B\rangle_{\h_E} = \sum_{n=0}^\infty \rho_n(A\times B);
\ee  
 
 (2) For any $f \in \h_E$ and $A\in \Bfin(\mu)$, 
 $$
 \langle f, G_A\rangle_{\h_E} = \int_A f\; d\nu.
 $$
 Furthermore, if 
 \be\label{eq def mc G}
 \mc G := \mathrm{span} \{G_A(\cdot) : A \in \Bfin\},
 \ee
 then $\mc G$ is dense in the energy space $\h_E$. 

(3) For $A, B\in \Bfin(\mu)$, we have
$$
\Delta G_A(x) = c(x) \chi_A(x),
$$
and 
$$
\Delta \omega_{A,B} = \Delta G_A - \Delta G_B = c(\chi_A - \chi_B)
$$
is in $L^2(\nu)$, where $\omega_{A,B}$ is defined in Section 
\ref{sect energy}.
 \end{theorem} 
 
We observe that the dipoles $\omega_{A,B}$ can be determined using
the formula $\omega_A{,B} = G_A - G_B$, see also Lemma 
\ref{lem_existence of dipoles}.   
 
 \begin{proof} (1) Clearly, relation \eqref{eq_inner prod G_A and G_B} follows 
 from  \eqref{eq-norm G_A}, so that it suffices to prove the formula
 for the norm of $G_A$ in $\h_E$.
For this, one has

 $$
 \ba 
\| G_A(x) \|_{\h_E}^2 & = \iint_{\VtV}(G_A(x) - P_A(y))^2\; d\rho(x,y)\\
 & = \iint_{\VtV} G_A(x)  (G_A(x) - G_A(y)) \; d\rho(x,y)\\
  & = \iint_{\VtV} G_A(x)  (G_A(x) - P_A(y)) c(x) P(x, dy) d\mu(x))\\
  &= \int_{V} G_A(x)  [G_A(x) -  P(G_A)(x)] c(x) \; d\mu(x))\\
 &= \int_{V} G_A(x)  \left[\sum_{n=0}^\infty P^n(\chi_A)(x) -  
 \sum_{n=0}^\infty P^{n+1}(\chi_A)(x)\right]   c(x) \; d\mu(x))\\
& = \int_V  \sum_{n=0}^\infty P^n(\chi_A)(x) \chi_A(x) \; d\nu(x)\\
& = \sum_{n=0}^\infty \langle\chi_A, P^n(\chi_A  \rangle_{L^2(\nu)}\\
& = \sum_{n=0}^\infty \rho_n(A\times A).  
 \ea
 $$

For (2), we compute
$$
\ba 
\langle f, G_A\rangle_{\h_E} & = \frac{1}{2} \iint_{\VtV} 
(f(x) - f(y))(G_A(x) - G_A(y))\; d\rho(x,y)\\
&= \iint_{\VtV} (f(x) G_A(x) -  f(x)G_A(y))\; d\rho(x, y)\\
& = \int_V  \left[f(x) G_A(x) c(x)  - f(x) \left(\int_V G_A(y) P(x, dy)
\right)c(x)\right]\; d\mu(x)\\
&= \int_V  f(x) c(x) \left[\sum_{n=o}^\infty  P^n(\chi_A)(x) -
\sum_{n=o}^\infty  P^{n+1}(\chi_A)(x)\right] \; d\mu(x)\\
&= \int_V  f(x) \chi_A(x) c(x) \; d\mu(x)\\
& = \int_A f\; d\nu.
\ea
$$ 
It follows from the proved relation that if $\langle f, G_A\rangle_{\h_E}
 =0$ for all $A\in \Bfin(\mu)$, then  $f=0$, and $\mc G$ is dense in 
 $\h_E$.

(3) We compute using the definition of Green's function and the
fact that the series $\sum_n P_n(x, A)$ is convergent for all $x$ and 
all $A\in \Bfin(\mu)$:
$$
\ba
c(x)(I - P)G_A(x) = \ & c(x)(I - P)\sum_{n=0}^\infty  P_n(x, A) \\
= \ & c(x) \sum_{n=0}^\infty  P_n(x, A) - c(x) 
\sum_{n=1}^\infty  P_n(x, A) \\
= \ & c(x) \chi_A(x).
\ea
$$

\end{proof} 

\begin{corollary}
(1) Assuming that, for every $A \in \Bfin(mu)$, the function $G(\cdot, A)$  
belongs to  $L^2(\nu)$, we have
$$
\langle\chi_A, G(\cdot, A)\rangle_{L^2(\nu)} = \sum_{n\in \N_0}
\rho_n(A\times A).
$$ 

(2) If $G(\cdot, A)$  belongs to  $L^1_{\mathrm{loc}}(\nu)$, then
$\sum_{n\in \N_0}\rho_n(A\times A) < \infty$.

\end{corollary} 

 We use now the construction given in Subsection  \ref{subsect def RKHS}. 
 Let $S = \Bfin$, and we set
 \be\label{eq K for Bfin}
 K(A, B) = \sum_{n \in \N_0} \rho_n(A\times B).
 \ee
 
 We first observe that $K(A,  B)$ is a \textit{positive definite kernel}. This
  fact follows from  \eqref{eq_inner prod G_A and G_B} of Theorem 
 \ref{thm on G_A}.

Moreover, one can point out a realization of the kernel $K(A, B)$ in a 
$L^2$-space.  
 
 \begin{proposition}\label{prop K on Bfin pos def} Let the Markov operator 
 $P$ determine a transient Markov process.  Then, for any  $f \in L^2(\nu)$, 
the function $(I - P)^{-1/2}(f)$ belongs to  $L^2(\nu)$. Moreover, 
$A \mapsto K^*_A$  ($A \in \Bfin$) is a realization of the kernel $K$ on
 $L^2(\nu)$ where
$$
K^*_A(\cdot) := (I - P)^{-1/2}(\chi_A)(\cdot).
$$
  \end{proposition}
  
\begin{proof} We first need to show that  $(I - P)^{-1/2} :
L^2(\nu) \to L^2(\nu)$. From the spectral theorem for the self-adjoint
operator $P$ acting on $L^2(\nu)$, we obtain that 
$$
\langle  P(f), f\rangle_{L^2(\nu)} = \int_{-1}^1 t \langle Q(dt)f, 
f\rangle_{L^2(\nu)},
$$
where $Q(dt)$ is the projection valued measure for the  
operator $P$ in $L^2(\nu)$.
For a Borel function $\varphi$, we have 
$$
\langle  \varphi (P)f, f\rangle_{L^2(\nu)} = \int_{-1}^1 \varphi (t)
 \langle Q(dt)f, f\rangle_{L^2(\nu)}. 
$$
Then 
$$
\ba
|| (I - P)^{-1/2}(f) ||_{L^2(\nu)}  = & \langle (I -  P)^{-1}(f), 
f\rangle_{L^2(\nu)}\\ 
 = &\int_{-1}^1 \frac{1}{1-t} \langle Q(dt)f, f\rangle_{L^2(\nu)}\\ 
 = & \int_{-1}^1 \sum_{n=0}^\infty t^n \langle Q(dt)f, 
 f\rangle_{L^2(\nu)}\\ 
= & \sum_{n=0}^\infty \langle  P^n(f), f\rangle_{L^2(\nu)}.
 \ea
$$
Take $f = \chi_A, A \in \Bfin$. Then 
$$
\ba
|| (I - P)^{-1/2}(\chi_A) ||^2_{L^2(\nu)} = & \sum_{n=0}^\infty \langle  
P^n(\chi_A), \chi_A\rangle_{L^2(\nu)}\\
=& \sum_{n=0}^\infty \rho_n(A \times A)\\.
\ea
$$
Since $(P_n)$ is a transient Markov process, we see that the $L^2$-norm 
of $(I - P)^{-1/2}(\chi_A)$ is finite. Then the result follows from the density
of simple functions in $L^2(\nu)$. 

A similar computation can be used in order to show that, for $A, B\in \Bfin$, 
\be\label{eq realization 1}
\langle(I - P)^{-1/2}(\chi_A), (I - P)^{-1/2}(\chi_B) \rangle =
\sum_{n=0}^\infty \rho_n(A \times B).
\ee
Since, by definition, $K(A, B) = \sum_{n=0}^\infty \rho_n(A \times B)$, 
this means that the functions $K^*_A =(I - P)^{-1/2}(\chi_A)$ define a 
realization of the kernel $K(A, B)$.
\end{proof}

\begin{lemma} \label{lem P_n norm} Let $(P_n)$ be a sequence of
 probability measures that
 defines a transient Markov process. Then, for any $A\in \Bfin$, the function
 $P_n(\cdot, A)$ belongs to $\h_E$. Moreover, the following relation
 hold:
 \be\label{eq_P_k inner prod P_L}
\langle P_k(\cdot, A), P_l(\cdot, A) \rangle_{\h_E} =
\rho_{k+l}(A \times A) - \rho_{k+l+1}(A \times A),
 \ee
$$
\langle P_k(\cdot, A), G(\cdot, A) \rangle_{\h_E} =
\rho_{k}(A \times A).
$$
\end{lemma}

\begin{proof} In \cite{BezuglyiJorgensen_2018}, we proved that 
$$
\| P_n( \cdot, A)\|^2_{\h_E} = \rho_{2n}(A\times A) - \rho_{2n+1}
(A\times A), \qquad n \in \N.
$$
Using similar computation, one can generalize this result and prove that 
\eqref{eq_P_k inner prod P_L} holds. We leave details for the reader.
\end{proof}

\begin{corollary} \label{cor energy space loc int} Let $(P_n)$ be a transient
Markov process. Then the energy space $\h_E$ consists of locally 
integrable functions.
\end{corollary}

This result formulated in this corollary follows from Theorem 
\ref{thm on G_A} and \eqref{eq K for Bfin}.
We remark that Corollary \ref{cor energy space loc int} well agrees with
 Assumption 2 made in Section 
\ref{subsect_finiteEnergy}.

\begin{proof} The proof follows from the following fact:  for a function 
$f \in \h_E$ and $A \in \Bfin$, one has
$$
\langle f, G(\cdot, A) \rangle_{\h_E} = \int_A f \; d\nu.
$$
\end{proof}

In the remaining part of this section, we will define and study isometries 
between the three Hilbert spaces:  RKHS $\h(K)$,  energy space $\h_E$,
and $L^2(\nu)$. 

We define the operators $I_1, I_2$, and $I_3$ by setting 
$$
\ba 
  K(\cdot, A)& \stackrel{I_1} \longrightarrow G(\cdot, A)\\
 G(\cdot, A) & \stackrel{I_2} \longrightarrow (I - P)^{-1/2}(\chi_A) \\
(I - P)^{-1/2}(\chi_A) &  \stackrel{I_3} \longrightarrow K(\cdot, A).\\
\ea
$$
We recall that the family of functions  $\{G(\cdot, A)\ |\ A \in \Bfin\}$ is
dense in $\h_E$, so that $I_2$ is a densely defined map.  By linearity, 
the definition of $I_1$ can be extended to a dense subset of functions from
$\h(K)$. One can also show that $I_3$ is also densely defined  operator. 

\begin{lemma}\label{lem I_3 densely def}
Let $\Dfin(\mu) $ be the span of characteristic functions $\chi_A, A \in 
\Bfin$. Then the set $(I - P)^{-1/2}(\Dfin)$ is dense in $L^2(\nu)$.
\end{lemma}

\begin{proof}
The proof is based on the fact that $L^2(\nu)$ does not contain nontrivial
harmonic functions.
\end{proof}

\begin{corollary}
The operators $I_1, I_2$ and $I_3$ implement isometric isomorphisms 
of the Hilbert spaces:
$$
I_1 : \h(K) \to \h_E, \ \ I_2 : \h_E \to L^2(\nu), \ \ I_3 : L^2(\nu) \to
\h(K).
$$
\end{corollary}

\begin{proof} The result follows immediately from the proved above 
formulas for the norm of functions in the Hilbert spaces $\h(K), \h_E$, and
$L^2(\nu)$. Indeed, we deduce from \eqref{eq-norm G_A}, 
\eqref{eq K for Bfin}, and \eqref{eq realization 1} that the functions
$K(\cdot, A), G(\cdot, A)$, and $(I - P)^{-1/2}(\chi_A)$ have the same
norm equal to $\sum_{n=0}^\infty \rho_n(A \times A)$.

\end{proof}

\begin{corollary}
The energy space $\h_E$ is a RKHS. 
\end{corollary}

\subsection{Reproducing kernel Hilbert space generated by symmetric
 measures}

We give here one more construction of a reproducing kernel Hilbert space 
defined by a symmetric measure $\rho$. We take the set $S$ to be the
 same as above, i.e., $S = \Bfin(\mu)$. Define a function $k = k_\rho : 
 \Bfin(\mu)  \times \Bfin(\mu) \rightarrow \R$ as follows:
\be\label{eq def RKHS k}
k_\rho : (A, B) \rightarrow \rho((A \cap B)\times V) - \rho(A \times B)
\ee
 
 It is difficult to determine whether the function $k_\rho(A, B)$ is 
 positive definite using the definition (\ref{eq def RKHS k}) only. 
The next statement follows immediately from Theorem 
\ref{thm_stucture of energy space}.

\begin{lemma} The function $k_\rho$ is positive definite on the set 
$\Bfin$. 
\end{lemma}

\begin{proof}
The proof is based on the formula
$$
\langle \chi_A, \chi_B\rangle_{\h_E} = \rho((A \cap B)\times V) - 
\rho(A \times B)
$$
that is proved in  \cite[Lemma 6.18]{BezuglyiJorgensen2018}.  Then, for
 any $A_1, ... , A_m \in  \Bfin$ and $\alpha_1, ... , \alpha_m \in \R$, 
$$
\sum_{i,j =1}^m \alpha_i\beta_j k_\rho(A_i, B_j) = 
\left\langle\ \sum_{i=1}^m \alpha_i\chi_{A_i}, \sum_{j=1}^m \beta_j
\chi_{B_j}\ \right\rangle_{\h_E}.
$$ 
\end{proof}

Let $\h(k_\rho)$ be the reproducing kernel Hilbert space RKHS) 
 constructed by the positive definite function $k_\rho(\cdot, \cdot)$.
By general theory, the Hilbert space $\h(k_\rho)$ consists 
of functions on $\Bfin(\mu) $ such that, for any $f \in \h(k_\rho)$,
$$
f(B) = \langle f, k_\rho( \cdot, B) \rangle_{\h(k_\rho)}.
$$

In what follows,  we will discuss  relations between the two Hilbert spaces 
$\h(k_\rho)$ and $\h_E$.  

\begin{lemma}\label{lem- i isometry}
Let $f = \sum_{i =1}^n \alpha_i \chi_{A_i}$ where $A_i \in \Bfin(\mu)$. 
Then the map 
$$
i : f \mapsto f : \Dfin(\mu) \to \h_E
$$
can be extended to an isometry $i$ from $\h(k_\rho)$ to $\h_E$.
\end{lemma}

\begin{proof}
Since $\h(k_\rho)$ is the closure of $\Dfin(\mu)$, it suffices to check that
$i$ is an isometry on functions  $f = \sum_{i =1}^n \alpha_i \chi_{A_i}$,  
from $\Dfin(\mu)$:
$$
\| f \|_{\h(k_\rho)} = \| f \|_{h_E}. 
$$
Indeed, we have
$$
\ba
|| \sum_{i =1}^n \alpha_i \chi_{A_i} ||^2_{\h(k_\rho)} = &
\sum_{i,j =1}^n \alpha_i\alpha_j k_\rho(A_i, A_j) \\
= & \sum_{i,j =1}^n \alpha_i\alpha_j (\rho((A_i \cap A_j) \times V) - 
\rho(A_i \times A_j)\\
= & \iint_{\VtV} f(x)^2\; d\rho(x, y) - \iint_{\VtV} f(x) f(y)\; 
d\rho(x, y)\\
= & \frac{1}{2}  \iint_{\VtV} (f(x) - f(y))^2\; d\rho(x, y)\\
= & || f ||^2_{\h_E}.
\ea
$$
\end{proof}

\begin{corollary}
The map $i$ defined in Lemma  \ref{lem- i isometry} implements 
an isometric isomorphism between $\h(k_\rho)$ and the subspace
$ \h_E \ominus Harm = \overline{\Dfin(\mu)}^{\h_E}$ of $\h_E$.
\end{corollary}

For any $f \in \h_E$, define a signed measure $\mu_f$ on $\VB$
by setting
\be\label{eq mu_f def}
\mu_f(A) = \langle \chi_A, f \rangle_{\h_E}.
\ee

\begin{lemma}\label{lem mu_f in RKHS}
The function $A \mapsto \mu_f(A) \in \h(k_\rho)$ for any $f \in \h_E$.
Moreover, for any $f \in h_E \ominus Harm$, we have
$$
|| \mu_f ||^2_{\h(k_\rho)} = || f ||^2_{\h_E}.
$$
\end{lemma}

\begin{proof} To  prove this result, we use the criterion
given in Lemma \ref{lem criterion} for the function
 $\mu_f(\cdot)$. It gives
 $$
 \ba
| \sum_{i=1}^n \alpha_i \mu_f(A) |^2 =  \ & | \sum_{i=1}^n \alpha_i 
\langle \chi_{A_i}, f\rangle_{\h_E} |^2\\
= \ & |  \langle \sum_{i=1}^n \alpha_i  \chi_{A_i}, f\rangle_{\h_E} |^2\\
\leq  \ & || f ||^2_{\h_E} || \sum_{i=1}^n \alpha_i  \chi_{A_i}||^2_{\h_E}\\
= \ & || f ||^2_{\h_E}  \left\langle \sum_{i=1}^n \alpha_i  \chi_{A_i},
\sum_{i=1}^n \alpha_i  \chi_{A_i} \right\rangle_{\h_E}\\
= \ & || f ||^2_{\h_E} \sum_{i,j= 1}^n \alpha_i \alpha_j 
\langle \chi_{A_i}, \chi_{A_j}\rangle_{\h_E}\\
= \ & || f ||^2_{\h_E} \sum_{i,j= 1}^n \alpha_i \alpha_j  k_\rho(A_i, A_j).\\
\ea
$$
Hence, the result follows from (\ref{eq criterion RKHS}).

For the second statement, it suffices to take $f = \sum_{i =1}^n \alpha_1
\chi_{A_i}$ with $A_i \in \Bfin(\mu)$. Then  
$$
\ba
|| \mu_f ||^2_{\h(k_\rho)} & = || \sum_{i =1}^n \alpha_i k_\rho(\cdot, 
A_i)||^2_{\h(k_\rho)} \\
& = \sum_{i,j =1}^n \alpha_i\alpha_j k_\rho(A_i, A_j)\\
& = \sum_{i,j =1}^n \alpha_i\alpha_j \langle \chi_{A_i}, \chi_{A_j}
\rangle_{\h_E}\\
& = || f ||^2_{\h_E}.
\ea
$$
\end{proof}

\begin{theorem}\label{thm mu_f and Delta}
(1)  For any $f \in \h_E$, we have
\be\label{eq_RN for mu_f}
\mu_f(A) = \int_A \Delta f \; d\mu,  
\ee
i.e., 
$$
\frac{d\mu_f}{d\mu}(x) = \Delta(f)(x).
$$

(2) $f\in \h arm \ \Longleftrightarrow \ \mu_f = 0$.

\end{theorem}
\begin{proof}
(1) To prove (\ref{eq_RN for mu_f}), we use that functions from $\h_E$ are
 locally integrable. Then one can compute
 $$
 \ba 
 \langle \chi_A, f\rangle_{\h_E} = &\  \frac{1}{2} \iint_{\VtV} (\chi_A(x) -
  \chi_A(y)) (f(x) - f(y))\; d\rho(x, y)\\ 
  = &\ \iint_{\VtV} (\chi_A(x) f(x) - \chi_A(x) f(y))\; d\rho(x, y)\\
=& \   \int_{V} \chi_A(x) \left( \int_V( f(x) -  f(y) )\; d\rho_x (y)\right)d\mu(x)\\
  = & \ \int_A \Delta (f)(x)\ d\mu(x).
 \ea
 $$
 
 (2) We use (1) to prove the second assertion. Indeed, it follows from 
 (\ref{eq_RN for mu_f}) that if $f \in \h arm$, then $\Delta(f) = 0$ and
 $\mu_f (A) = 0$ for any $A\in \Bfin(\mu)$. 
 
 On the other hand, if  $\mu_f (A) = 0$, then
 $\int_A \Delta(f)\; d\mu = 0$ for any $A\in \Bfin(\mu)$. This means that 
 $\Delta(f) = 0$ a.e.
\end{proof}

It follows from Lemma \ref{lem mu_f in RKHS} that the map
$$
\h_E \ni f \ \stackrel{W} \mapsto \ \mu_f \in \h(k_\rho)
$$
is well defined. 

\begin{corollary}
Let $f = v_{A,B}$ where $A,B \in \Bfin(\mu)$. Then the function 
$C \mapsto \mu_{v_{A,B}}$ satisfies the property
$$
\mu_{v_{A,B}}(C) = \mu(A \cap C) - \mu(B \cap C).
$$
\end{corollary}

\begin{proof} The result follows from th following computation
$$
\ba
\mu_{v_{A,B}}(C)  & =  \langle v_{A,B}, \chi_C \rangle_{\h_E}\\
& = \langle\Delta  v_{A,B}, \chi_C \rangle_{L^2(\mu)}\\
& =  \langle \chi_A - \chi_B, \chi_C \rangle_{L^2(\mu)}\\
&= \mu(A \cap C) - \mu(B \cap C).
\ea
$$

\end{proof}

\begin{corollary} The map $W$ is a co-isometry and $i^* = W$.
\end{corollary}

\begin{proof}
We will show that $ W(f) = \mu_f$ and $i^*(f)$ coincide as functions on 
$\Dfin(\mu)$. Let $A$ be any set from $\Bfin$, then 
$$
\ba 
\mu_f(A) = &\  \langle f, \chi_A\rangle_{\h_E}\\
= & \  \langle f, i(\chi_A)\rangle_{\h_E}\\
= & \  \langle i^*(f), \chi_A\rangle_{\h(k_\rho)}\\
= &\ i^*(f)(A).
\ea
$$
The last equality follows from the reproducing property of $\h(k_\rho)$. 
\end{proof}

\subsection{Reproducing kernel Hilbert space generated by the measure
$\nu$}

Let  $(V, \B, \nu)$ be a $\sigma$-finite measure space,  and let
 $A, B$ be any two elements of the set $\Bfin(\nu)$. Define $K_\nu
 : \Bfin(\nu) \times \Bfin(\nu) \to [o, \infty)$ as follows:
 $$
 K_\nu(A, B) := \nu(A\cap B). 
 $$
Then $K_\nu$ is a positive definite kernel because $K_\nu(A, B) = 
\langle \chi_A, \chi_B\rangle_{L^2(\nu)}$ and 
$$
 \sum_{i, j =1}^n \alpha_i\alpha_j K_\nu(A_i, A_j) = 
 ||\sum_{i=1}^n \alpha_i \chi_{A_i} ||_{L^2(\nu)}.  
 $$

\begin{theorem}\label{thm RKHS for nu}
The kernel $K_\nu(A, B)$ is positive definite, and, for the corresponding
 RKHS $\h_\nu$, a function $F$ on $\Bfin(\nu)$ is in $\h_\nu$ if and only if
there exists a function $f \in L^2_{\mathrm{loc}}(\nu)$ such that
\be\label{eq-F(A) in RKHS}
F(A) = \int_A f\; d\nu, \qquad A\in \Bfin(\nu).
\ee
Moreover,
\be\label{eq ||F||}
|| F ||_{\h_\nu} = || f ||_{\h_E}.
\ee
\end{theorem}

\begin{proof} Let $F$ be a function on $\Bfin(\nu)$ defined by 
\eqref{eq-F(A) in RKHS}. To show that $F(\cdot)$ belongs to $\h_\nu$, we
use \eqref{eq criterion RKHS} of Lemma \ref{lem criterion}, i.e., 
\be\label{eq criterion for F(A)}
\left( \sum_{i=1}^n \xi_i F(A_i) \right)^2 \leq C_F 
 \sum_{i, j =1}^n \xi_i\xi_j K_\nu(A_i, A_j), 
\ee
where $\xi_i \in \R, A_i \in \Bfin(\nu), i =1,...,n$, and the constant $C_F$
depends on $F$ only. Take the function 
$$
\varphi(x) := \sum_{i=1}^n \xi_i \chi_{A_i}(x)
$$
which belongs to $L^2(\nu)$ and find that
$$
||\varphi ||^2_{L^2(\nu)} = \sum_{i, j =1}^n \xi_i\xi_j \nu(A_i\cap A_j)
= \left|\left| \sum_{i=1}^n  \xi_i K(\cdot, A_i)\right|\right|^2_{\h_\nu}.
$$

For $F(\cdot)$ as in \eqref{eq-F(A) in RKHS}, we compute using
the Schwarz inequality
$$
\ba
\left( \sum_{i=1}^n \xi_iF(A_i) \right)^2 & = 
\left( \sum_{i=1}^n \xi_i \int_{A_i} f \; d\nu \right)^2\\
& = \left( \int_V f (\sum_{i=1}^n \xi_i \chi_{A_i} )\; d\nu \right)^2\\
& \leq || f ||^2_{L^2(B, \nu)} || \varphi||^2_{L^2(\nu)}\\
& = C_F \sum_{i, j =1}^n \xi_i\xi_j K_\nu(A_i, A_j) 
\ea
$$
where $B = \bigcup_i A_i$.

Relation \eqref{eq ||F||} can be proved by using the corresponding definitions
of the norms in $\h_{\nu} $ and $\h_E$. We leave the details for the 
reader. 
\end{proof}

\begin{remark}
We note that Theorem \ref{thm RKHS for nu} agrees with the definition
of $\h_\nu$ because,  for a fixed $B \in \Bfin(\nu)$, the function
$K(\cdot, B)$ is represented by \eqref{eq-F(A) in RKHS} with $f =\chi_B$. 

If $V = [0, \infty)$,  $\nu$ is the Lebesgue measure on $V$, 
and $K_\nu (A, B) = \nu(A\cap B)$, then 
$$
 K_\nu([0, s] \cap [0, t]) = s\wedge t, \ \ \ \ s, t \in \R_+.
$$
It follows that the RKHS $\h_\nu$ can be represented as 
$$
\h_\nu = \{ F : F(0) = 0, \ F' \in L^2([0, \infty), \nu) \}
$$
with 
$$
|| F ||^2_{\h_\nu} = \int_0^\infty |F'|^2 \; d\nu.  
$$

\end{remark}

\subsection{Conditionally negative definite kernel} 
In this subsection we discuss the notion of conditionally negative definite
 kernels. 

\begin{definition}
Let $X$ be arbitrary set. Then the map  $N : X \times X \to \R$ is 
called a \textit{conditionally negative definite kernel} if for any $n \in 
\N$, any finite set of points $x_1, ... m x_n$, and any real numbers
$\lambda_1, ... , \lambda_n$, one has
$$
\sum_{i,j =1}^n \lambda_i\lambda_j N(x_i, x_j) \leq 0
$$
provided that $\sum_{i=1}^n \lambda_ i =0$.
\end{definition}

Conditionally negative definite kernels were completely characterized 
in \cite{Schoenberg1938} where the following result was proved.

\begin{theorem}\label{thm CND kernel}
Let $N : X \times X \to \R$ satisfy the following conditions: 
$N(x, y) = N(y,x) \geq 0$, and $N(x, x) =0$. If $N(x, y)$ is a
conditionally negative definite kernel, then there exists a real Hilbert space
$\h(N)$ and a map $\alpha : X \to \h(N)$ such that
$$
N(x, y) = || \alpha(x) - \alpha(y)||^2_{\h(N)}.
$$
\end{theorem} 

It  was also shown in \cite{Joziak2015} that, for any conditionally negative 
 definite kernel $N(x,y)$, there exists a positive definite kernel $K(x, y)$ 
 and a function $F : X \to \R$ such that 
$$
N(x, y) = - K(x, y) + F(x) + F(y).
$$

Let now $\rho$ be a symmetric measure on $\vv$, and let $\h_E$
be the finite energy Hilbert space.  For any sets $A, B \in \Bfin(\mu)$, we
consider the dipoles $\omega_{A,B}$ defined in Section \ref{sect energy}. 
We recall that these functions form a dense subset in $\h_E$ and
satisfy the relation $\Delta(\omega_{A, B}) = c(\chi_A - \chi_B)$.
As was mentioned in Section \ref{sect energy}, we can fix a set $A_0 \in
\Bfin(\mu)$ and represent $\omega_{A, B}$ as the difference
$\omega_{A,A_0} - \omega_{B,A_0}$. 

\begin{lemma}\label{lem N via omega}
Let 
$$
N_\rho(A,B) = || \omega_{A, B}||^2_{\h_E}, \ \ \ A, B \in \Bfin(\mu).
$$
Then $N_\rho$ is a conditionally negative definite kernel. 
\end{lemma}

The lemma follows directly from Theorem \ref{thm CND kernel}.

Applying Theorem \ref{thm CND kernel}, we can define a Hilbert space 
$\h(N)$ and a map $\alpha : \Bfin(\mu) \to \h(N)$ such that 
$N_\rho(A,B) = || \alpha(A) - \alpha(B)||^2_{\h(N)}$. 

\begin{theorem}
Let $\Lambda : \omega_{A,B} \mapsto \alpha(A) - \alpha(B)$ can be 
extended by linearity to an isometric isomorphism $\h_E \cong \h(N)$.
\end{theorem}

\begin{proof}
The proof is based on the given above definitions, Theorem 
\ref{thm CND kernel}, and Lemma \ref{lem N via omega}. We leave 
the details to the reader.
\end{proof}

\textbf{Acknowledgments.} The authors are pleased to thank colleagues and
 collaborators, especially members of the seminars in Mathematical Physics 
 and Operator Theory at the University of Iowa, where versions of this work 
 have been presented. We acknowledge very helpful conversations with 
 among others Professors Paul Muhly, Wayne Polyzou; and conversations at 
 distance with Professors Daniel Alpay, and his colleagues at both Ben Gurion 
 University, and Chapman University. 

\bibliographystyle{alpha}
\bibliography{Energy-references}

\end{document}